\documentclass[a4paper,10pt,twoside]{article}
\usepackage{amsmath,amsthm,amssymb}
\usepackage{times}
\usepackage{enumerate}

\def\titlerunning#1{\gdef\titrun{#1}}
\makeatletter
\def\author#1{\gdef\autrun{\def\and{\unskip, }#1}\gdef\@author{#1}}
\def\address#1{{\def\and{\\\hspace*{18pt}}\renewcommand{\thefootnote}{}%
\footnote {#1}}%
\markboth{\autrun}{\titrun}}
\makeatother
\def\email#1{e-mail: #1}
\def\subjclass#1{{\renewcommand{\thefootnote}{}%
\footnote{\emph{Mathematics Subject Classification (2010):} #1}}}
\def\keywords#1{\par\medskip
\noindent\textbf{Keywords.} #1}

\theoremstyle{plain}
\newtheorem{proposition}{Proposition}[section]
\newtheorem{lemma}[proposition]{Lemma}
\newtheorem{theorem}{Theorem}
\newtheorem{corollary}[proposition]{Corollary}
\newtheorem{fact}{Fact}[section]
\newtheorem*{claim}{Claim}
\newtheoremstyle{addendumstyle}{\topsep}{\topsep}{\itshape}{}{\bfseries}{.}{.5em plus 1pt minus 1pt}{#1 #2 to #3}
\theoremstyle{addendumstyle}
\newtheorem{addendum}{Addendum}

\theoremstyle{definition}
\newtheorem{definition}{Definition}[section]
\theoremstyle{remark}
\newtheorem{Part}{Part}
\numberwithin{equation}{section}

{}
\usepackage[T1]{fontenc}

\usepackage{enumerate}

\usepackage{microtype}

\usepackage[alwaysadjust]{paralist}

\usepackage{cite}

\usepackage[utf8]{inputenc}

\newcommand{\abs}[1]{\mathopen\lvert#1\mathclose\rvert}
\newcommand{\bigabs}[1]{\bigl\lvert#1\bigr\rvert}

\newcommand{\norm}[1]{\mathopen\lVert#1\mathclose\rVert}
\newcommand{\floor}[1]{\lfloor#1\rfloor}
\newcommand{\N}{{\mathbb N}}
\newcommand{\R}{{\mathbb R}}
\renewcommand{\S}{{\mathbb S}}
\newcommand{\cE}{\mathcal{E}}
\newcommand{\cK}{\mathcal{K}}
\newcommand{\cS}{\mathcal{S}}
\newcommand{\cT}{\mathcal{T}}
\newcommand{\cU}{\mathcal{U}}

\DeclareMathOperator{\supp}{supp}
\DeclareMathOperator{\Supp}{Supp}
\DeclareMathOperator{\jac}{jac}
\DeclareMathOperator{\dist}{dist}
\DeclareMathOperator{\Dist}{Dist}
\newcommand{\Id}{\mathrm{Id}}
\DeclareMathOperator{\Int}{int}

\newcommand{\dif}{\,\mathrm{d}}

\usepackage{refcount}
\newcounter{cte}
\newcommand{\Constant}{\refstepcounter{cte} C_{\thecte}}
\newcommand{\NewConstant}{\setcounter{cte}{1} C_{\thecte}}
\newcommand{\SameConstant}{C_{\thecte}}

\begin{document}

%%%%% To ease editing, add:

%\baselineskip=17pt

%%%%%%%%%%%%%%%%

%% In the running head, give an abbreviation of the title. 
\titlerunning{Density for higher order Sobolev spaces into compact manifolds}

\title{Strong density for higher order Sobolev spaces into compact manifolds}

\author{Pierre Bousquet
\and{}
Augusto C. Ponce
\and{}
Jean Van Schaftingen
}

%\date{\today}

\maketitle

\address{P.~Bousquet: 
Universit{\'e} de Toulouse,
Institut de Math\'ematiques de Toulouse, CNRS UMR5219, 
UPS IMT,
F-31062 Toulouse Cedex 9,
France; \email{Pierre.Bousquet@math.univ-toulouse.fr}
\and
A.~C.~Ponce: 
Universit{\'e} catholique de Louvain,
Institut de Recherche en Math{\'e}matique et Physique,
Chemin du cyclotron 2, L7.01.02,
1348 Louvain-la-Neuve,
Belgium; \email{Augusto.Ponce@uclouvain.be}
\and
J.~Van~Schaftingen: 
Universit{\'e} catholique de Louvain,
Institut de Recherche en Math{\'e}matique et Physique,
Chemin du cyclotron 2, L7.01.02,
1348 Louvain-la-Neuve,
Belgium; \email{Jean.VanSchaftingen@uclouvain.be}
}


\begin{abstract}
Given a compact manifold \(N^n\), an integer \(k \in \mathbb{N}_*\) and an exponent \(1 \le p < \infty\), we prove that the class \(C^\infty(\overline{Q}^m; N^n)\) of smooth maps on the cube with values into \(N^n\) is dense with respect to the strong topology in the Sobolev space \(W^{k, p}(Q^m; N^n)\) when the homotopy group \(\pi_{\lfloor kp \rfloor}(N^n)\) of order \(\lfloor kp \rfloor\) is trivial. 
We also prove the density of maps that are smooth except for a set of dimension \(m - \lfloor kp \rfloor - 1\), without any restriction on the homotopy group of \(N^n\).\\[5pt]
\emph{Published in J. Eur. Math. Soc. (JEMS) \textbf{17} (2015), 763--817, by the European Mathematical Society.}
{}

\end{abstract}

%\maketitle
%\setcounter{tocdepth}{1}
%\tableofcontents

\section{Introduction}

There are two natural approaches to define Sobolev maps  with values in a compact manifold. 
More precisely, let \(N^n\) be a compact connected smooth manifold of dimension \(n\) imbedded in \(\R^\nu\) for some  \(\nu \ge 1\) \cite{Nash-1954, Nash-1956},  \(k \in \N_*\) and \(1 \le p < +\infty\). One can first define \( W^{k, p}(Q^m; N^n) \) as the set
\[
\big\{ u \in W^{k, p}(Q^m; \R^\nu) : u \in N^n \text{ a.e.}  \big\},
\]  
where \(Q^{m}\subset \R^m\) is the open unit cube.
The other possibility is to define \(H^{k, p} (Q^m; N^n)\) as the completion of the class of smooth maps \(C^\infty(\overline{Q}^m; N^n)\) with respect to the Sobolev metric
\[
  d_{k,p} (u, v) = \norm{u - v}_{L^p(Q^m)} + \sum_{i = 1}^k \norm{D^i u - D^i v}_{L^p(Q^m)}.
\]

These spaces are the natural framework for the study of harmonic maps  \cite{Moser2005, LinWang2008, Simon1996, GiaquintaMucci2006},
biharmonic maps \cite{ChangWangYang1999, Moser2008, Scheven2008, Struwe2008, Urakawa2011} and polyharmonic maps \cite{AngelsbergPumberger2009, GastelScheven2009, GoldsteinStrzeleckiZatorska2009, GongLammWang2012, LammWang2009} with values into manifolds.
They also arise in some physical models \cite{Brezis2003, Mermin1979}.
For instance, maps into the sphere, the projective space and other manifolds appear in liquid cristal models
\cite{Mucci2012, MucciNicolodi2012, BallZarnescu2011, BrezisCoronLieb1986}.

In contrast with the real-valued case \cite{Deny-Lions, Meyers-Serrin}, these spaces may be different. 
For instance, \(H^{1, p} (Q^2; \S^1) = W^{1, p} (Q^2; \S^1)\) if and only if \(p \ge 2\) \cite[Theorem~3]{BethuelZheng}. 
The goal of this paper is to determine when \(H^{k, p}(Q^m; N^n) = W^{k, p}(Q^m; N^n)\).

This always happens when \(kp \ge m\) \cite[Section~4, Proposition]{SchoenUhlenbeck}, as \(W^{k, p} (\R^m) \cap L^\infty(\R^m)\) is imbedded into the space of functions 
with vanishing mean oscillation \(\mathrm{VMO} (\R^m)\) \cite[Example~1, Eq.~(7)]{Brezis-Nirenberg}.
The main result of this paper completely solves the problem in the case \(kp < m\).
It is remarkable that such an analytical question has a purely topological answer:

\begin{theorem}
\label{theoremDensityManifoldMain}
If \(kp < m\), then \(H^{k, p}(Q^m; N^n) = W^{k, p}(Q^m; N^n)\) if and only if \(\pi_{\floor{kp}}(N^n) \simeq \{0\}\).
\end{theorem}

We denote by \(\floor{kp}\) the integral part of \(kp\) and by \(\pi_{\floor{kp}}(N^n)\) the \(\floor{kp}\)th homotopy group of \(N^n\);
the topological condition \(\pi_{\floor{kp}}(N^n) \simeq \{0\}\) means that every continuous map \(f : \S^{\floor{kp}} \to N^n\) on the \(\floor{kp}\) dimensional sphere is homotopic to a constant map.
The necessity of this assumption has been known for some time \cite[Theorem~2]{BethuelZheng} \cite[Theorem~3]{Escobedo} \cite[Section~4, Example]{SchoenUhlenbeck} \cite[Theorem~4.4]{Mironescu}.

The case \(k = 1\) of Theorem~\ref{theoremDensityManifoldMain} is the main result of Bethuel's seminal work \cite[Theorem~1]{Bethuel} (see also \cite{Hang-Lin-2001, Hang-Lin}).
The case \(k \ge 2\) cannot be handled by merely adapting Bethuel's tools due to the rigidity of \(W^{k, p}\) and requires new ideas. 
A typical issue one faces when dealing with two maps in \(W^{k,p}\) is that they cannot be glued together under the sole assumption that their traces coincide.
Results concerning strong density of smooth maps in higher order Sobolev maps have been known  
in some cases situations where \(N^n\) is a sphere \cite[Theorem~5]{Mironescu} \cite[Theorem~4]{Brezis-Mironescu} \cite[Theorem~2]{Escobedo}. 

\medskip

In the case \(\pi_{\floor{kp}} (N^n) \not \simeq \{0\}\), we prove that \(W^{k, p} (Q^m;N^n)\) is the completion of a set of maps that are smooth outside a small singular set.
For this purpose, given \(i \in \{0, \dotsc, m-1\}\) we denote by \(R_i(Q^m; N^n)\) the set of maps \(u : \overline{Q}^m  \to N^n\) which are smooth on \(\overline{Q}^m \setminus T\),
where \(T\) is a finite union of \(i\) dimensional planes, and such that for every \(j \in \N_*\) and \(x \in \overline{Q}^m \setminus T\), 
\[
  \abs{D^j u(x)} \le \dfrac{C}{\dist{(x, T)}^j}
\]
for some constant \(C \ge 0\) depending on \(u\) and \(j\).

\begin{theorem}
\label{theoremDensityManifoldNontrivial}
If \(kp < m\) and \(\pi_{\floor{kp}} (N^n) \not \simeq \{0\}\), then \(W^{k, p}(Q^m; N^n)\) is the completion of \(R_{i}(Q^m; N^n)\) with respect to the Sobolev metric \(d_{k,p}\) if and only if \(i = m-\floor{kp}-1\).
\end{theorem}

This result was known for an arbitrary manifold \(N^n\) only in the case \(k = 1\) \cite[Theorem~2]{Bethuel} (see also \cite[Theorem~1.3]{Hang-Lin}). 
It is a fundamental tool in the study of the weak density of smooth maps in Sobolev spaces and in the study of topological singularities of Sobolev maps \cite{Hardt-Riviere, PakzadRiviere2003, HangLin2003, GiaquintaModicaSoucek1998, BourgainBrezisMironescu2005, GiaquintaMucci2006}.
Counterparts of Theorems~\ref{theoremDensityManifoldMain} and~\ref{theoremDensityManifoldNontrivial} for fractional Sobolev spaces \(W^{s, p}(Q^{m}; N^{n})\) such that \(0 < s < 1\) have been investigated by Brezis and Mironescu~\cite{BrezisMironescu2015}.

\medskip

We explain the strategy of the proofs of Theorems~\ref{theoremDensityManifoldMain} and \ref{theoremDensityManifoldNontrivial} under the additional assumption \(kp > m - 1\). 
Given a decomposition of \(Q^m\) in cubes of size \(\eta > 0\), we distinguish them between \emph{good cubes} and \emph{bad cubes} --- a notion reminiscent from \cite{Bethuel} --- as follows:
for a map \(u \in W^{k, p}(Q^m; N^n)\) and a cube \(\sigma^m_\eta\) in \(Q^m\) of radius \(\eta > 0\), \(\sigma^m_\eta\) is a \emph{good cube} if
\[
  \frac{1}{\eta^{m - k p}} 
  \int\limits_{\sigma^m_\eta} 
  \abs{D u}^{k p}
   \lesssim 1,
\]
which means that \(u\) does not oscillate too much in \(\sigma^m_\eta\); otherwise \(\sigma^m_\eta\) is a \emph{bad cube}. 
The main steps in the proof of Theorem~\ref{theoremDensityManifoldNontrivial} are the following:

\begin{compactdesc}
\item[Opening] We construct a map \(u_\eta^\mathrm{op}\) which is continuous on a neighborhood of the \(m-1\) dimensional faces of the bad cubes, and equal to \(u\) elsewhere.
This map, which takes its values into \(N^n\), is close to \(u\) with respect to the \(W^{k, p}\) distance because there are not too many bad cubes. 
Since \(k p > m - 1\), \(W^{k, p}\) maps are continuous on faces of dimension \(m-1\).
\item[Adaptive smoothing] By convolution with a smooth kernel, we then construct a smooth map \(u_\eta^\mathrm{sm} \in W^{k,p}(Q^m; N^n)\).  
The scale of convolution is chosen to be of the order of \(\eta\) on the good cubes, and close to zero in a neighborhood of the faces of the bad cubes.  
On the union of these sets, we are thus ensuring that \(u_\eta^\mathrm{sm}\) takes its values in a small neigborhood of \(N^n\). 
\item[Thickening] We propagate diffeomorphically the values of \(u_\eta^\mathrm{sm}\) near the faces of the bad cubes to the interior of these cubes. The resulting map \(u_\eta^\mathrm{th}\) coincides 
with \(u_\eta^\mathrm{sm}\)  on the good cubes and near the faces of the bad cubes, is close to \(u\) with respect to the \(W^{k, p}\) distance and takes its values in a neighborhood of \(N^n\). This construction creates at most one singularity at the center of each bad cube.
\end{compactdesc}

The map obtained by projecting \(u_\eta^\mathrm{th}\) from a neighborhood of \(N^n\) into \(N^n\) itself belongs to the class \(R_{0}(Q^m; N^n)\) and converges strongly to \(u\) with respect to the Sobolev distance \(d_{k, p}\) as \(\eta \to 0\).
This argument works regardless of the \(\floor{kp}\)th homotopy group of \(N^n\); see Theorem~\ref{theoremDensityManifoldNontrivialwithouttopologicalcondition} in Section~\ref{sectionDensityR} below.

The sketch of the proof we have announced in a previous work \cite{Bousquet-Ponce-VanSchaftingen} for \(k=2\) and \(2p > m-1\) is based on the strategy above but was organized differently following \cite{Ponce-VanSchaftingen} (see also \cite{Gastel-Nerf}).
The opening technique was introduced by Brezis and Li~\cite{Brezis-Li} in their study of homotopy classes of \(W^{1, p} (Q^m; N^n)\). 

\medskip

The proof of Theorem~\ref{theoremDensityManifoldMain} in the case \(kp \ge m-1\) relies on the fact that \(R_0(Q^m; N^n)\) is strongly dense in \(W^{k, p}(Q^m; N^n)\) with respect to the Sobolev distance \(d_{k, p}\).
The approximation of a map \(u \in R_0(Q^m; N^n)\) by a map in \(C^\infty(\overline{Q}^m; N^n)\) in this case goes as follows:

\begin{compactdesc}
\item[Continuous extension property]
By the assumption on the homotopy group of \(N^n\),  for any \(\mu < 1\) there exists a smooth map \(u^\textrm{ex}_\mu\) with values into \(N^n\) which coincides with \(u\) outside a neighborhood of radius \(\mu\eta\) of the singular set of \(u\). As a drawback, \(u^\text{ex}_\mu\) may be far from \(u\) with respect to the \(W^{k, p}\) distance. 
\item[Shrinking] 
We propagate diffeomorphically the values of \(u^\text{ex}_\mu\) in the neighborhood of radius \(\mu\eta\) of each singularity of \(u\) into a smaller neighborhood of radius \(\tau\mu\eta\) for \(\tau < 1\). Since \(kp < m\), we obtain a map \(u^\text{sh}_{\tau, \mu}\) which is still smooth but now close to \(u\) with respect to the \(W^{k, p}\) distance. This construction is reminiscent of thickening but does not create singularities.
\end{compactdesc}

The smooth map \(u^\text{sh}_{\tau, \mu}\) converges strongly to \(u\) with respect to the \(W^{k, p}\) distance as \(\tau \to 0\) and \(\mu \to 0\).
The role of this continuous extension property in the case of \(W^{1, p}\) approximation of maps \(u\) with higher dimensional singularities has been clarified by Hang and Lin~\cite{Hang-Lin}.

\tableofcontents

\section{Opening}
For \(a \in \R^m\) and \(r>0\), we denote by \(Q^{m}_r(a)\) the cube of radius  \(r\) with center \(a\); by radius of the cube we mean half of the length of one of its edges. When \(a = 0\), we abbreviate \(Q_r^m = Q^{m}_r(0)\).  

\begin{definition}
A family of closed cubes \(\mathcal{S}^m\) is a cubication of \(A \subset \R^m\) if all cubes have the same radius, if \(\bigcup\limits_{\sigma^m \in \mathcal{S}^m} \sigma^m = A\) and if for every \(\sigma^m_1, \sigma^m_2 \in \mathcal{S}^m\) which are not disjoint, \(\sigma^m_1 \cap \sigma^m_2\) is a common face of dimension \(i \in \{0, \dotsc, m\}\).
\end{definition}

The radius of a cubication is the radius of any of its cubes.

\begin{definition}
Given a cubication \(\mathcal{S}^m\) of \(A \subset \R^m\) and \(\ell \in \{0, \dotsc, m\}\), the skeleton of dimension \(\ell\) is the set \(\mathcal{S}^\ell\) of all \(\ell\) dimensional faces of all cubes in \(\mathcal{S}^m\). A subskeleton of dimension \(\ell\) of \(\cS^m\) is a subset of \(\cS^\ell\).
\end{definition}

Given a skeleton \(\mathcal{S}^\ell\), we denote by \(S^\ell\) the union of all elements of \(\mathcal{S}^\ell\),
\[
S^\ell = \bigcup_{\sigma^\ell \in \mathcal{S}^\ell} \sigma^\ell.
\]

For a given map \(u \in W^{k, p} (U^m; \R^\nu)\) on some subskeleton \(\mathcal{U}^m\) and for any \(\ell \in \{0, \dots, m-1\}\), we are going to construct a map \(u \circ \Phi \in W^{k, p} (U^m; \R^\nu)\) which is constant along the normals to \(U^\ell\) in a neighborhood of \(U^\ell\). In this region, the map \(u \circ \Phi\) will thus be essentially a \(W^{k, p}\) map of \(\ell\) variables. 
Hence, if \(k p > \ell\), then \(u \circ \Phi\) will be continuous there, whereas in the critical case \(\ell = kp\), the map \(u \circ \Phi\) need not be continuous but will still have vanishing mean oscillation. 
In this construction the map \(\Phi\) depends on \(u\) and is never injective.
This idea of opening a map has been inspired by a similar construction of Brezis and Li~\cite{Brezis-Li}.

Given a map \(\Phi : \R^m \to \R^m\), we denote by \(\Supp{\Phi}\) the \emph{geometric support of \(\Phi\)}, namely the closure of the set \(\{x \in \R^m : \Phi(x) \ne x\}\). This should not be confused with the analytic support \(\supp{\varphi}\) of a function \(\varphi:\R^m\to \R\)  which is the closure of the set \(\{x \in \R^m : \varphi(x) \ne 0\}\).

\begin{proposition}\label{openingpropGeneral}
Let \(\ell \in \{0, \dotsc, m-1\}\), \(\eta > 0\), \(0 < \rho < \frac{1}{2}\), and \(\mathcal{U}^\ell\) be a subskeleton of \(\R^m\) of radius \(\eta\). Then, for every \(u\in W^{k, p}(U^\ell + Q^m_{2\rho\eta}; \R^\nu)\), there exists a smooth map \(\Phi : \R^m \to \R^m\) such that
\begin{enumerate}[$(i)$]
\item % 
for every \(i \in \{0, \dotsc, \ell\}\) and 
for every \(\sigma^i \in \mathcal{U}^i\), \(\Phi\) is constant on the  \(m-i\) dimensional cubes of radius \(\rho\eta\) which are orthogonal to \(\sigma^i\),
\label{itemgenopeningprop1}
\item \(\Supp{\Phi} \subset U^\ell + Q^m_{2\rho\eta}\)
and \(\Phi(U^\ell + Q^m_{2\rho\eta}) \subset U^\ell + Q^m_{2\rho\eta}\),
\label{itemgenopeningprop4}
\item \(u \circ \Phi \in W^{k, p}(U^\ell + Q^m_{2\rho\eta}; \R^{\nu})\),
\label{itemgenopeningprop5}
and for every \(j \in \{1, \dotsc, k\}\),
\begin{equation*}
 \eta^{j} \norm{D^j(u \circ \Phi)}_{L^p(U^\ell + Q^m_{2\rho\eta})} \leq C \sum_{i=1}^j  \eta^{i} \norm{D^i u}_{L^p(U^\ell + Q^m_{2\rho\eta})},
\end{equation*}
for some constant \(C > 0\) depending on \(m\), \(k\), \(p\) and \(\rho\),
\item 
for every \(\sigma^\ell \in \cU^\ell\) and for every \(j \in \{1, \dotsc, k\}\),
\begin{equation*}
 \eta^{j} \norm{D^j(u \circ \Phi)}_{L^p(\sigma^\ell + Q^m_{2\rho\eta})} \leq C' \sum_{i=1}^j  \eta^{i} \norm{D^i u}_{L^p(\sigma^\ell + Q^m_{2\rho\eta})},
\end{equation*}
for some constant \(C' > 0\) depending on \(m\), \(k\), \(p\) and \(\rho\).
\label{itemgenopeningprop6}
\end{enumerate}
\end{proposition}

In the case of \(W^{2, p}\) maps, the quantity \(\norm{D(u \circ \Phi)}_{L^p}\) can be estimated in terms of \(\norm{Du}_{L^p}\); hence there is no explicit dependence of \(\eta\). However, concerning the second-order term, estimate in \((\ref{itemgenopeningprop5})\) reads
\begin{equation*}
\norm{D^2 (u \circ \Phi)}_{L^p({U^\ell+Q^m_{2\rho\eta}})} \leq C \norm{D^2 u}_{L^p({U^\ell+Q^m_{2\rho\eta}})} + \frac{C}{\eta} \norm{Du}_{L^p({U^\ell+Q^m_{2\rho\eta}})}.
\end{equation*}
The factor \(\frac{1}{\eta} \) which comes naturally from a scaling argument is one of the differences with respect to the opening of \(W^{1, p}\)  maps. In the proof of Theorem~\ref{theoremDensityManifoldMain}, we shall use the Gagliardo-Nirenberg interpolation inequality to deal with this extra term.

Since the map \(u\) in the statement is defined almost everywhere, the map \(u \circ \Phi\) need not be well-defined by standard composition of maps. By \(u \circ \Phi\), we mean a map \(v\) in \(W^{k, p}\) such that there exists a sequence of smooth maps \((u_n)_{n \in \N}\) converging to
 \(u\) in \(W^{k, p}\) such that \((u_n \circ \Phi)_{n \in \N}\) 
converges to \(v\) in \(W^{k, p}\).  By pointwise convergence, this map \(u \circ \Phi\) inherits several properties of \(\Phi\) and of \(u\). For instance, if \(\Phi\) is constant in a neighborhood of some point \(a\), then so is \(u \circ \Phi\). One can show that under some assumptions on \(\Phi\) which are satisfied in all the cases that we consider \(u \circ \Phi\) does not depend on the sequence  \((u_n)_{n \in \N}\), but we shall not make use of this fact. The only property we shall need from \(u \circ \Phi\) is that its essential range is contained in the essential range of \(u\); this is actually the case in view of Lemma~\ref{lemmaOpeningLp}~\((\ref{OpeningLp2})\) below. In particular, \emph{if \(u\) is a map with values into the manifold \(N^n\), then \(u \circ \Phi\) is also a map with values into  \(N^n\).}

The following proposition is the main tool in the proof of Proposition~\ref{openingpropGeneral}.

\begin{proposition}
\label{openinglemmaGeneral}
Let \(\ell \in \{0, \dotsc, m-1\}\), \(\eta > 0\), \(0 < \underline{\rho} < \overline{\rho}\) and \(A \subset \R^\ell\) be an open set. For every \(u\in W^{k, p}(A \times Q_{\overline{\rho}\eta}^{m - \ell}; \R^\nu)\), there exists a smooth map \(\zeta : \R^{m - \ell} \to \R^{m - \ell}\) such that 
\begin{enumerate}[$(i)$]
\item \(\zeta\) is constant in \(Q_{\underline{\rho}\eta}^{m - \ell}\),
\label{itemopeninglemmaGeneral1} 
\item \(\Supp{\zeta} \subset Q_{\overline{\rho}\eta}^{m - \ell}\) and \(\zeta(Q_{\overline{\rho}\eta}^{m - \ell}) \subset Q_{\overline{\rho}\eta}^{m - \ell}\),
\label{itemopeninglemmaGeneral2} 
\label{itemgenopeninglemma1}
\item if \(\Phi : \R^m \to \R^m\) is defined for every  \(x = (x', x'') \in \R^\ell \times \R^{m - \ell}\) by
\[
\Phi(x) = (x', \zeta(x''))
\]
then \(u \circ \Phi \in W^{k, p}(A \times Q_{\overline{\rho}\eta}^{m - \ell}; \R^{\nu})\),
and for every \(j \in \{1, \dotsc, k\}\),
\label{itemopeninglemmaGeneral3} 
\begin{equation*}
 \eta^{j} \norm{D^j(u \circ \Phi)}_{L^p({A \times Q^{m-\ell}_{\overline{\rho}\eta}})} \leq C \sum_{i=1}^j  \eta^{i} \norm{D^i u}_{L^p({A \times Q^{m-\ell}_{\overline{\rho}\eta}})},
\end{equation*}
for some constant \(C > 0\) depending on \(m\), \(k\), \(p\), \(\underline{\rho}\) and \(\overline{\rho}\).
\end{enumerate}
\end{proposition}

The proof of Proposition~\ref{openinglemmaGeneral} is based on a Fubini type argument, which gives some flexibility on the choice of \(\zeta\).{}
In particular, given finitely many measurable subsets \(A_{1}, \dots, A_{s} \subset A\), the map \(\zeta\) can be chosen such that we have in addition, for every \(r \in \{1, \dots, s\}\) and for every \(j \in \{1, \dotsc, k\}\), 
\begin{equation*}
 \eta^{j} \norm{D^j(u \circ \Phi)}_{L^p({A_{r} \times Q^{m-\ell}_{\overline{\rho}\eta}})} 
 \leq C \sum_{i=1}^j  \eta^{i} \norm{D^i u}_{L^p({A_{r} \times Q^{m-\ell}_{\overline{\rho}\eta}})}.
\end{equation*}

We will temporarily accept this proposition and the observation that follows it, and we prove the main result of the section:

\begin{proof}[Proof of Proposition~\ref{openingpropGeneral}]
We first take a finite sequence \((\rho_i)_{0 \le i \le \ell}\) such that
\[
\rho = \rho_\ell < \ldots < \rho_i < \ldots < \rho_0 < 2\rho.
\]
We construct by induction on \(i \in \{0, \dots, \ell\}\)  a map \(\Phi^i: \R^m\to \R^m\) such that
\begin{enumerate}[(a)]
\item for every \(r \in \{0, \dots, i\}\) and every \(\sigma^r \in \mathcal{U}^r\), \(\Phi^{i}\) is constant on the \(m-r\) dimensional cubes of radius \(\rho_i\eta\) which are orthogonal to \(\sigma^r\),
\label{openingrecursive1}
\item \(\Supp{\Phi^i}\subset U^i+Q^{m}_{2\rho\eta}\) and  \(\Phi^i(U^i+Q^{m}_{2\rho\eta})\subset U^i+Q^{m}_{2\rho\eta}\), 
\label{openingrecursive2}
\item \(u\circ \Phi^i \in W^{k,p}(U^\ell+Q^{m}_{2\rho\eta}; \R^{\nu})\),
\label{openingrecursive3}
\item for every \(\sigma^i \in \cU^i\) and for every \(j \in \{1, \dotsc, k\}\),
\label{openingrecursive4}
\begin{equation*}
 \eta^{j} \norm{D^j(u \circ \Phi^i)}_{L^p(\sigma^i+Q^{m}_{2\rho\eta})} \leq C \sum_{\alpha =1}^j  \eta^{\alpha} \norm{D^\alpha u}_{L^p(\sigma^i+Q^{m}_{2\rho\eta})},
\end{equation*}
for some constant \(C > 0\) depending on \(m\), \(k\), \(p\) and \(\rho\).
\end{enumerate}
The map \(\Phi^\ell\) will satisfy the conclusion of the proposition.

\medskip
If \(i = 0\), then \(\cU^0\) consists of all vertices of cubes in \(\cU^m\). To construct \(\Phi^0\), we apply Proposition~\ref{openinglemmaGeneral} to the map \(u\) around each \(\sigma^0 \in \mathcal{U}^0\) with parameters \(\rho_0 < 2\rho\) and \(\ell = 0 \): in this case, the set \(A \times Q_{\overline{\rho}\eta}^{m - \ell}\) in  Proposition~\ref{openinglemmaGeneral} is simply  \( Q_{2\rho}^m \). 
This gives a map \(\Phi^0\) such that for every \(\sigma^0 \in \mathcal{U}^0\),   \(\Phi^0\) is constant on \(\sigma^0 + Q^m_{\rho_0 \eta}\)
and \(\Phi^0=\Id\) outside \(U^0+Q^m_{2\rho \eta}\). Moreover, \(u\circ\Phi^0\in W^{k,p}(U^\ell+Q^m_{2\rho\eta}; \R^{\nu})\) and for every \(\sigma^0 \in \cU^0\) and for every \(j \in \{1, \dotsc, k\}\),
\begin{equation*}
 \eta^{j} \norm{D^j(u \circ \Phi^i)}_{L^p(\sigma^0 + Q^{m}_{2\rho\eta})} \leq C \sum_{\alpha =1}^j  \eta^{\alpha} \norm{D^\alpha u}_{L^p(\sigma^0 + Q^{m}_{2\rho\eta})},
\end{equation*}

Assume that the maps \(\Phi^0,\dotsc, \Phi^{i-1}\) have been constructed. 
To define \(\Phi^i\), we first apply Proposition~\ref{openinglemmaGeneral}, for each \(\sigma^i \in \mathcal{U}^i\), to the map \(u\circ \Phi^{i-1}\) with \(A = \sigma^i\) and parameters \(\rho_{i} < \rho_{i-1}\) . This gives a smooth map \(\Phi_{\sigma^i} : \R^m \to \R^m\) such that \(\Phi_{\sigma^i}\) is constant on the \(m-i\) dimensional cubes of radius \(\rho_i\eta\)
which are orthogonal to \(\sigma^i \).

Let \(\Phi^i : \R^m \to \R^m\) be defined for \(x \in \R^m\) by
\[
\Phi^i(x)=
\begin{cases}
\Phi^{i-1}(\Phi_{\sigma^i}(x)) &\text{if \(x\in \sigma^i + Q^{m}_{\rho_{i-1}\eta}\) where \(\sigma^{i} \in \cU^{i}\),}\\
\Phi^{i-1}(x)&\text{otherwise.}
\end{cases}
\]
We first explain why \(\Phi^i\) is well-defined. 
For this purpose, let 
\[
x \in (\sigma^i_1 + Q^{m}_{{\rho}_{i-1}\eta})\cap (\sigma^i_2 + Q^{m}_{{\rho}_{i-1}\eta})
\]
for some \(\sigma^i_1 \in \mathcal{U}^i\) and  \(\sigma^i_2 \in \mathcal{U}^i\) such that \(\sigma^i_1 \neq \sigma^i_2\). 
In particular, \(\sigma^i_1\) and \(\sigma^i_2\) are not disjoint, and there exists a smallest dimension \(r \in \{0, \dots, i-1\}\) such that
\[{}
x \in \tau^r + Q^{m}_{\rho_{i-1}\eta}
\quad \text{and} \quad
\tau^r \subset \sigma^i_1\cap \sigma^i_2
\]
for some \(\tau^r \in \mathcal{U}^{r}\).
By the formula of \(\Phi_{\sigma^i_j}\) given in Proposition~\ref{openinglemmaGeneral}, the points \(x\), \(\Phi_{\sigma^i_1}(x)\) and \(\Phi_{\sigma^i_2}(x)\) belong to the same
\(m-r\) dimensional cube of radius \(\rho_{i-1}\eta\) which is orthogonal to \(\tau^r\).
Since by induction hypothesis \(\Phi^{i-1}\) is constant on the \(m-r\) dimensional cubes of radius \(\rho_{i-1}\eta\) which are orthogonal to \(\tau^r\), we have
\[
 \Phi^{i-1}(x) = \Phi^{i-1}(\Phi_{\sigma^i_1}(x)) = \Phi^{i-1}(\Phi_{\sigma^i_2}(x)).
\]
This implies that \(\Phi^i\) is well-defined. Moreover, \(\Phi^i\) is smooth and satisfies properties \eqref{openingrecursive1}--\eqref{openingrecursive3}. 

We prove the estimates given by \eqref{openingrecursive4}. 
If \(e_1, \dots, e_m\) is an orthonormal basis of \(\R^m\) compatible with the skeleton \(\cU^\ell\), then by abuse of notation we denote by \(\sigma^i \times Q_{\alpha\eta}^{m - i}\) the parallelepiped given by
\[
\Big\{ x + \sum_{s = 1}^{m - i} t_s e_{r_s} : x \in \sigma^i \text{ and } \abs{t_s} \le \alpha\eta \Big\},
\]
where \(e_{r_1}, \dotsc, e_{r_{m - i}}\) are orthogonal to \(\sigma^{i}\).
Note that for every \(\sigma^i \in \cU^i\),
\[
\sigma^i + Q^m_{2\rho\eta} = (\sigma^i \times Q^{m-i}_{2\rho\eta}) \cup (\partial\sigma^i + Q^m_{2\rho\eta}),
\]
where \(\partial\sigma^i\) denotes the \(i-1\) dimensional skeleton of \(\sigma^i\).
By property~\((\ref{itemopeninglemmaGeneral3})\) of Proposition~\ref{openinglemmaGeneral}, 
\[{}
\int\limits_{\sigma^i \times Q^{m-i}_{\rho_{i-1}\eta}} \eta^{jp} \abs{D^j (u\circ \Phi^{i-1} \circ \Phi_{\sigma^{i}})}^p 
\le {\NewConstant} \sum_{\alpha = 1}^{j} 
\int\limits_{\sigma^i \times Q^{m - i}_{\rho_{i-1}\eta}} \eta^{\alpha p}\abs{D^\alpha (u\circ \Phi^{i-1})}^p,
\]
and then, since \(\Phi^{i} = \Phi^{i-1} \circ \Phi_{\sigma^{i}}\) on \((\sigma^i \times Q^{m-i}_{2 \rho \eta}) \setminus ( \partial\sigma^i + Q^m_{2\rho\eta} )\) and since the geometric support \(\Supp{\Phi_{\sigma^{i}}}\) is contained in \(\sigma^i \times Q^{m-i}_{\rho_{i-1}\eta}\), we have
\begin{multline}
\label{eqEstimateFace}
\int\limits_{(\sigma^i \times Q^{m-i}_{2 \rho \eta}) \setminus (\partial\sigma^i + Q^m_{2\rho\eta})} \eta^{jp} \abs{D^j (u\circ \Phi^{i})}^p \\
\le {\NewConstant} \sum_{\alpha = 1}^{j} 
\int\limits_{\sigma^i \times Q^{m - i}_{2 \rho \eta}} \eta^{\alpha p}\abs{D^\alpha (u\circ \Phi^{i-1})}^p.
\end{multline}

We claim that the maps \(\Phi_{\sigma^{i}}\) can be chosen such that the additional property holds: for every \(j \in \{1, \dotsc, k\}\),
\begin{equation}
\label{eqEstimateBoundaryFace}
\int\limits_{\partial\sigma^i + Q^{m}_{2\rho\eta}} \eta^{jp} \abs{D^j (u\circ \Phi^{i})}^p 
\le {\Constant} \sum_{\alpha = 1}^{j} \,
\int\limits_{\partial\sigma^i + Q^{m}_{2\rho\eta}} \eta^{\alpha p}\abs{D^\alpha (u\circ \Phi^{i-1})}^p.
\end{equation}
Indeed,  by the remark following Proposition~\ref{openinglemmaGeneral}, for every \(\sigma^{i} \in \mathcal{U}^{i}\) we may further require that \(\Phi_{\sigma^{i}}\) satisfies for every \(i-1\) dimensional cube \(\tau^{i - 1} \subset \partial\sigma^{i}\) and for every \(j \in \{1, \dotsc, k\}\),
\begin{multline*}
\int\limits_{[(\tau^{i-1} + Q^{m}_{2\rho\eta}) \cap \sigma^i] \times Q^{m}_{\rho_{i-1}\eta}} \eta^{jp} \abs{D^j (u\circ \Phi^{i-1} \circ \Phi_{\sigma^{i}})}^p \\
\le {\Constant} \sum_{\alpha = 1}^{j} \,
\int\limits_{[(\tau^{i-1} + Q^{m}_{2\rho\eta}) \cap \sigma^i] \times Q^{m}_{\rho_{i-1}\eta}} \eta^{\alpha p}\abs{D^\alpha (u\circ \Phi^{i-1})}^p.
\end{multline*}
%and then, since \(\Supp{\Phi_{\sigma^{i}}} \subset \sigma^i \times Q^{m-i}_{\rho_{i-1}\eta}\),
%\begin{multline*}
%\int\limits_{[(\tau^{i-1} + Q^{m}_{2\rho\eta}) \cap \sigma^i] \times Q^{m}_{2\rho\eta}} \eta^{jp} \abs{D^j (u\circ \Phi^{i-1} \circ \Phi_{\sigma^{i}})}^p \\
%\le {\Constant} \sum_{\alpha = 1}^{j} \,
%\int\limits_{[(\tau^{i-1} + Q^{m}_{2\rho\eta}) \cap \sigma^i] \times Q^{m}_{2\rho\eta}} \eta^{\alpha p}\abs{D^\alpha (u\circ \Phi^{i-1})}^p,
%\end{multline*}

Next, given \(\tau^{i - 1} \subset \partial\sigma^{i}\), denote by \(\sigma_{1}^{i}, \dots, \sigma_{\theta}^{i}\) the \(i\) dimensional cubes in \(\cU^{i}\) containing \(\tau^{i - 1}\) in their boundaries.
In this case,
\[{}
\tau^{i - 1} + Q^{m}_{2\rho\eta}
\subset \bigg(\bigcup_{\beta = 1}^{\theta} [(\tau^{i-1} + Q^{m}_{2\rho\eta}) \cap \sigma_{\beta}^i] \times Q^{m}_{\rho_{i-1}\eta} \bigg) \cup \big\{x \in \R^{m} : \Phi^{i}(x) = \Phi^{i-1}(x)\big\}.
\] 
Since for every \(\beta \in \{1, \dots, \theta\}\), \(\Phi^{i} = \Phi^{i-1} \circ \Phi_{\sigma_{\beta}^{i}}\) on \([(\tau^{i-1} + Q^{m}_{2\rho\eta}) \cap \sigma_{\beta}^i] \times Q^{m}_{\rho_{i-1}\eta}\), by the previous estimate on each cube \(\sigma_{\beta}^{i}\) and by additivity of the integral, we get
\[{}
\int\limits_{\tau^{i-1} + Q^{m}_{2\rho\eta}} \eta^{jp} \abs{D^j (u\circ \Phi^{i})}^p 
\le {\Constant} \sum_{\alpha = 1}^{j} \,
\int\limits_{\tau^{i-1} + Q^{m}_{2\rho\eta}} \eta^{\alpha p}\abs{D^\alpha (u\circ \Phi^{i-1})}^p.
\]
Adding both sides of this inequality over the \(i - 1\) dimensional cubes \(\tau^{i-1} \subset \partial\sigma^{i}\), we deduce estimate \eqref{eqEstimateBoundaryFace} as we claimed.

By additivity of the integral and by estimates \eqref{eqEstimateFace} and \eqref{eqEstimateBoundaryFace}, we then obtain
\[
\int\limits_{\sigma^i + Q^{m}_{2\rho\eta}} \eta^{jp} \abs{D^j (u\circ \Phi^{i})}^p 
\le {\Constant} \sum_{\alpha = 1}^{j} 
\int\limits_{\sigma^i + Q^{m}_{2\rho\eta}} \eta^{\alpha p}\abs{D^\alpha (u\circ \Phi^{i-1})}^p.
\]
Since by induction hypothesis \(\Phi^{i-1}\) coincides with the identity map outside \(U^{i-1}+Q^m_{2\rho\eta}\), for every \(\alpha \in \{1, \dots, j\}\) we have
\begin{multline*}
\int\limits_{\sigma^i + Q^m_{2\rho\eta}} \eta^{\alpha p}\abs{D^\alpha (u\circ \Phi^{i-1})}^p \\
= \int\limits_{\partial\sigma^i + Q^m_{2\rho\eta}} \eta^{\alpha p}\abs{D^\alpha (u\circ \Phi^{i-1})}^p 
+\int\limits_{(\sigma^i + Q^m_{2\rho\eta}) \setminus (\partial\sigma^i + Q^m_{2\rho\eta})} \eta^{\alpha p}\abs{D^\alpha u}^p.
\end{multline*}
By induction hypothesis, for every \(i-1\) dimensional face \(\tau^{i-1}\) of \(\partial\sigma^i\),
\[
\int\limits_{\tau^{i-1} + Q^{m}_{2\rho\eta}} \eta^{\alpha p} \abs{D^\alpha (u\circ \Phi^{i-1})}^p 
\le {\Constant} \sum_{\beta = 1}^{\alpha} 
\, \int\limits_{\tau^{i-1} + Q^{m}_{2\rho\eta}} \eta^{\beta p}\abs{D^\beta u}^p.
\]
Since the number of overlaps of the sets \(\tau^{i-1} + Q^{m}_{2\rho\eta}\) is bounded from above by a constant only depending on \(m\), we have by additivity of the integral,
\[
\int\limits_{\partial\sigma^i + Q^m_{2\rho\eta}} \eta^{\alpha p}\abs{D^\alpha (u\circ \Phi^{i-1})}^p 
\le {\Constant} \sum_{\beta = 1}^{\alpha} 
\, \int\limits_{\partial\sigma^i + Q^{m}_{2\rho\eta}} \eta^{\beta p}\abs{D^\beta u}^p.
\]
Therefore, 
\[
\int\limits_{\sigma^i + Q^m_{2\rho\eta}} \eta^{j p}\abs{D^j (u \circ \Phi^i)}^p 
\le {\Constant} \sum_{\alpha = 1}^{j} 
\int\limits_{\sigma^i + Q^{m}_{2\rho\eta}} \eta^{\alpha p}\abs{D^\alpha u}^p.
\]

The map \(\Phi^\ell\) satisfies properties \((\ref{itemgenopeningprop1})\)--\((\ref{itemgenopeningprop6})\). The estimate of property~\((\ref{itemgenopeningprop5})\) is a consequence of \((\ref{itemgenopeningprop6})\) and the additivity of the integral.
\end{proof}

We proceed to prove Proposition~\ref{openinglemmaGeneral} by making precise the meaning of \(u \circ \Phi\) in the statement.

Given a continuous function \(\Psi : U \times V \to W\) and \(z \in V\), we denote by \(\Psi_z : U \to W\) the map defined for every \(x \in U\) by 
\[
\Psi_z(x) = \Psi(x, z).
\]
For every measurable function \(g : W \to \R\), the composition \(g \circ \Psi_z\) is well-defined and gives a measurable function defined on \(W\) for every \(z\).

\begin{lemma}
\label{lemmaOpeningLp}
Let \(U, W \subset \R^m\) and \(V \subset \R^l\) be measurable sets and let \(\Psi : U \times V \to W\) be a continuous map such that for every measurable function \(g : W \to \R\),
\[
\int\limits_V \norm{g \circ \Psi_z}_{L^1(U)} \dif z\le C \norm{g}_{L^1(W)}.
\]
If \(u \in L^p(W; \R^\nu)\) and if \((u_n)_{n \in \N}\) is a sequence of measurable functions converging to \(u\) in \(L^p(W; \R^\nu)\), then there exists a subsequence \((u_{n_i})_{i \in \N}\) such that for almost every \(z \in V\),
\begin{enumerate}[$(i)$]
\item the sequence \((u_{n_i} \circ \Psi_z)_{i \in \N}\) converges in \(L^p(U; \R^\nu)\) to a function which we denote by \(u \circ \Psi_{z}\),
\label{OpeningLp1}
\item the essential range of \(u \circ \Psi_{z}\) is contained in the essential range of \(u\).
\label{OpeningLp2}
\end{enumerate}
\end{lemma}

\begin{proof}
Let \((u_n)_{n \in \N}\) be a sequence of measurable functions in \(W\) converging to \(u\) in \(L^p(W; \R^\nu)\). Given a sequence  \((\varepsilon_n)_{n \in \N}\) of positive numbers, let \((u_{n_i})_{i \in \N}\) be a subsequence such that for every \(i \in \N\),
\[
\norm{u_{n_{i + 1}} - u_{n_i}}_{L^p(W)} \le \varepsilon_i.
\]
By the assumption on \(\Psi\),
\[
\int\limits_V \norm{u_{n_{i + 1}} \circ \Psi_{z} - u_{n_i} \circ \Psi_{z}}_{L^p(U)}^p \dif z 
 \le  C \norm{u_{n_{i + 1}} - u_{n_i}}_{L^p(W)}^p \le  C \varepsilon_i^p.
\]

Given a sequence  \((\alpha_n)_{n \in \N}\) of positive numbers, let
\[
Y_i = \Big\{  z \in V : \norm{u_{n_{i + 1}} \circ \Psi_{z} - u_{n_i} \circ \Psi_{z}}_{L^p(U)} > \alpha_i  \Big\}. 
\]
If the series \(\sum\limits_{i = 0}^\infty \alpha_i\) converges, then for every \(t \in \N\) and for every \(z \not\in \bigcup\limits_{i = t}^\infty Y_i\), the sequence \((u_{n_i} \circ \Psi_{z})_{i \in \N}\) is a Cauchy sequence in \(L^p(U; \R^\nu)\).

By the Chebyshev inequality,
\[
\alpha_i^p |Y_i| \le \int\limits_{Y_i} \norm{u_{n_{i + 1}} \circ \Psi_{z} - u_{n_i} \circ \Psi_{z}}_{L^p(U)}^p  \dif z \le  C \varepsilon_i^p.
\]
Hence, for every \(t \in \N\),
\[
\Big| \textstyle \bigcup\limits_{i = t}^\infty Y_i\Big| \le C \displaystyle\sum\limits_{i = t}^\infty \Big(\frac{\varepsilon_i}{\alpha_i}\Big)^p.
\]

Taking the sequences \((\varepsilon_n)_{n \in \N}\) and \((\alpha_n)_{n \in \N}\) such that both  series \(\sum\limits_{i = 0}^\infty \alpha_i\) and \(\sum\limits_{i = 0}^\infty (\varepsilon_i/\alpha_i)^p\) converge, then the set \(E = \bigcap\limits_{t = 0}^\infty \bigcup\limits_{i = t}^\infty Y_i\) is negligible and for every \(z \in V \setminus E\), \((u_{n_i} \circ \Psi_{z})_{i \in \N}\) is a Cauchy sequence in \(L^p(U; \R^\nu)\). This proves assertion \((\ref{OpeningLp1})\).

\medskip
It suffices to prove assertion \((\ref{OpeningLp2})\) when \(W\) has finite Lebesgue measure.
For every \(z \in V \setminus E\), we denote by \(u \circ \Psi_z\) the limit in \(L^p(U; \R^\nu)\) of the sequence \((u_{n_i} \circ \Psi_{z})_{i \in \N}\).

Let \(\theta : \R^\nu\to \R\) be a continuous function such that \( \theta^{-1}(0)\) is equal to the essential range of \(u\)  and \(0 \le \theta \le 1\) in \(\R^\nu\).
For every \(i \in \N\),
\[
\int\limits_V \norm{ \theta \circ  (u_{n_i} \circ \Psi_z)}_{L^1(U)} \dif z \leq C  \norm{ \theta \circ  u_{n_i} }_{L^1(W)}.
\]
By Fatou's lemma,
\[
\int\limits_V \norm{\theta \circ  (u \circ \Psi_z)}_{L^1(U)} \dif z  \le \liminf_{i \to \infty} \int\limits_V \norm{ \theta \circ  (u_{n_i} \circ \Psi_z)}_{L^1(U)} \dif z.
\]
Since \(W\) has finite Lebesgue measure and \(\theta\) is bounded, as \(i\) tends to infinity we get
\[
\int\limits_V \norm{\theta \circ  (u \circ \Psi_z)}_{L^1(U)} \dif z \le C \norm{ \theta \circ u}_{L^1(W)} = 0.
\]
Therefore, for almost every \(z \in V\), \(\norm{\theta \circ  (u \circ \Psi_z)}_{L^1(U)} = 0\), whence  the essential range of \(u \circ \Psi_z\) is contained in the essential range of \(u\). 
\end{proof}

From the previous lemma, we can prove the following property for maps in \(W^{k, p}\):

\begin{lemma}
\label{lemmaOpeningSobolev}
Let \(U, W \subset \R^m\) and \(V \subset \R^l\) be open sets and let \(\Psi : U \times V \to W\) be a smooth map such that for every measurable function \(g : W \to \R\),
\[
\int\limits_V \norm{g \circ \Psi_z}_{L^1(U)} \dif z\le C \norm{g}_{L^1(W)}.
\]
If \(u \in W^{k, p}(W; \R^\nu)\) and if \((u_n)_{n \in \N}\) is a sequence of smooth functions converging to \(u\) in \(W^{k, p}(W; \R^\nu)\), then there exists a subsequence \((u_{n_i})_{i \in \N}\) such that for almost every \(z \in V\) the sequence \((u_{n_i} \circ \Psi_z)_{i \in \N}\) converges to \(u \circ \Psi_{z}\) in \(W^{k, p}(U; \R^\nu)\), and for every \(j \in \{1, \dots, k\}\),
\[
\int\limits_V \norm{D^j (u \circ \Psi_z)}_{L^p(U)} \dif z \leq C' \abs{V}^{1 - \frac{1}{p}} \sum_{i=1}^j \norm{D^i u}_{L^p(W)},
\]
for some constant \(C' > 0\) depending on \(m\), \(p\), \(k\), \(C\) and \(\max\limits_{1 \le j \le k}\sup\limits_{z \in V} {\norm{D^j \Psi_z}_{L^\infty(U)}}\).
\end{lemma}

\begin{proof}
Let \((u_n)_{n \in \N}\) be a sequence of smooth functions in \(W^{k, p}(W; \R^\nu)\) converging  to \(u\) in \(W^{k, p}(W; \R^\nu)\). 
By the previous lemma, there exists a subsequence \((u_{n_i})_{i \in \N}\)
 such that for almost every \(z \in V\),
\((u_{n_i} \circ \Psi_z)_{i \in \N}\) converges to \(u \circ \Psi_z\) in \(L^p\)
and for every \(j \in \{1, \dots, k\}\),
\(((D^j u_{n_i}) \circ \Psi_z)_{i \in \N}\) converges to \((D^j u) \circ \Psi_z\) in \(L^p\). 

For every \(v\in C^{\infty}(W; \R^\nu)\),  for every \(z \in V\) and for each \(j \in \{1, \dots, k\}\),
\[
\begin{split}
\abs{D^j (v\circ \Psi_{z}) (x)} 
& \le C_1 \sum_{i = 1}^j \sum_{\substack{1 \le t_1 \le \ldots \le t_i\\ t_1 + \dots + t_i = j}}{\abs{D^i v(\Psi_z(x))}  \abs{D^{t_1}\Psi_z(x)} \dotsm \abs{D^{t_i} \Psi_z(x)} }  \\
& \le C_2 \sum_{i = 1}^j \abs{D^i v (\Psi_z(x))},
\end{split}
\]
whence
\[
\norm{D^j (v\circ \Psi_{z})}_{L^p(U)}^p \leq C_3 \sum_{i=1}^j \norm{ \abs{D^i v}^p \circ \Psi_z}_{L^1(U)}.
\]
This implies that  for almost every \(z \in V\), \((u_{n_i} \circ \Psi_z)_{i \in \N}\) is a Cauchy sequence in \(W^{k, p}(U; \R^\nu)\), thus \((u_{n_i} \circ \Psi_z)_{i \in \N}\) converges to \(u \circ \Psi_z\) in \(W^{k, p}(U; \R^\nu)\).
Moreover, integrating with respect to \(z\) the above estimate and using the assumption on \(\Psi\) we get
\[
\begin{split}
\int\limits_V \norm{D^j (v\circ \Psi_{z})}_{L^p(U)}^p \dif z 
& \leq C_3 \sum_{i=1}^j \int\limits_V \norm{\abs{D^i v}^p \circ \Psi_z}_{L^1(U)} \dif z\\
& \le C_4 \sum_{i=1}^j \norm{\abs{D^i v}^p}_{L^1(W)} = C_4 \sum_{i=1}^j \norm{D^i v}_{L^p(W)}^p.
\end{split}
\]
Thus, by Hölder's inequality,
\[
\begin{split}
\int\limits_V \norm{D^j (v\circ \Psi_{z})}_{L^p(U)} \dif z  
& \le \abs{V}^{1 - \frac{1}{p}} \biggl( C_4 \sum_{i=1}^j \norm{D^i v}_{L^p(W)}^p \biggr)^{\frac{1}{p}}\\
& \le C_5 \abs{V}^{1 - \frac{1}{p}} \sum_{i=1}^j \norm{D^i v}_{L^p(W)}.
\end{split}
\]
We obtain the desired estimate by taking \(v=u_{n_i}\) and letting  \(n_i\) tend to infinity.
\end{proof}

We now show that the functional estimate in Lemma~\ref{lemmaOpeningLp} and~\ref{lemmaOpeningSobolev} is satisfied for maps \(\Psi\) of the form
\[
\Psi(x, z) = \zeta(x + z) - z.
\]
The strategy is based on an averaging device due to Federer and Fleming~\cite{Federer-Fleming} and adapted by Hardt, Kinderlehrer and Lin~\cite{Hardt-Kinderlehrer-Lin} in the context of Sobolev maps. 
It relies on the following lemma:

\begin{lemma}
\label{lemmaOpeningEstimate}
Let \(U, V, W \subset \R^l\) be measurable sets and let \(\zeta : U + V \to \R^l\) be a continuous map such that for every \(x \in U\) and for every \(z \in V\), \(\zeta (x + z) - z \in W\). Then, for every measurable function \(g : W \to \R\),
\begin{equation*}
\int\limits_{V} \biggl(\int\limits_{U} \abs{g(\zeta(x + z) - z)} \dif x \biggr) \dif z  \le |U + V| 
\int\limits_{W} \abs{g(x)} \dif x.
\end{equation*}
\end{lemma}

\begin{proof}
Let \(\xi : (U + V) \times V \to \R^l\) be the function defined by
\[
\xi(x, z) = \zeta(x + z) - z.
\]
By Fubini's theorem,
\[
\int\limits_{V} \biggl(\int\limits_{U} |(g \circ \xi)(x, z)|  \dif x \biggr) \dif z = \int\limits_{U} \biggl( \int\limits_{V} |g(\zeta(x + z) - z)|  \dif z \biggr) \dif x.
\]
Applying the change of variables \(\tilde z = x + z\) in the variable \(z\) and Fubini's theorem,
\begin{equation*}
\begin{split}
\int\limits_{V} \biggl(\int\limits_{U} |(g \circ \xi)(x, z)|  \dif x \biggr) \dif z
& = \int\limits_{U} \biggl( \int\limits_{x + V} |g(\zeta(\tilde z) + x - \tilde z)|  \dif \tilde z \biggr) \dif x\\
&=
\int\limits_{U + V} \biggl( \int\limits_{(\tilde z - V) \cap U} |g(\zeta(\tilde z) + x - \tilde z)|  \dif x \biggr) \dif \tilde z.
\end{split}
\end{equation*}
We now apply the change of variables \(\tilde x = \zeta(\tilde z) + x - \tilde z\) in the variable \(x\), and use the assumption on \(W\) to conclude
\begin{equation*}
\begin{split}
\int\limits_{V} \biggl(\int\limits_{U} |(g \circ \xi)(x, z)|  \dif x \biggr) \dif z 
&=
\int\limits_{U + V} \biggl( \int\limits_{\zeta(\tilde z) -(V \cap (\Tilde z-U))} |g(\tilde x)|  \dif \tilde x \biggr) \dif \tilde z \\
&\le
\int\limits_{U + V} \biggl( \int\limits_{W} |g(\tilde x)|  \dif \tilde x \biggr) \dif \tilde z\\
&= |U + V| \int\limits_{W} |g(\tilde x)|  \dif \tilde x.
\end{split}
\end{equation*}
This gives the desired estimate.
\end{proof}

\begin{proof}[Proof of Proposition~\ref{openinglemmaGeneral}]
By scaling, it suffices to establish the result when \(\eta = 1\).
We fix \(\Hat{\rho}\) such that \(2\Hat{\rho} < \overline{\rho} - \underline{\rho}\).

Let \(\tilde\zeta : \R^{m-\ell} \to \R^{m-\ell}\) be the smooth map defined by
\[
\tilde\zeta(y) = (1 - \varphi(y))y,
\]
where \(\varphi :  \R^{m-\ell} \to [0, 1]\) is a smooth function such that
\begin{itemize}[\(-\)]
\item for  \(y \in  Q_{\underline{\rho}+\Hat{\rho}}^{m-\ell}\), \(\varphi(y) = 1\),
\item for \(y \in  \R^{m - \ell} \setminus Q_{\overline{\rho}-\Hat{\rho}}^{m-\ell}\), \(\varphi(y) = 0\). 
\end{itemize}
For any \(z \in Q^{m-\ell}_{\hat\rho}\), the function \(\zeta :  \R^{m-\ell} \to \R^{m-\ell}\) defined for \(x'' \in  \R^{m-\ell}\) by
\[
\zeta(x'') = 
\Tilde\zeta(x'' + z) - z
\]
satisfies properties \((\ref{itemopeninglemmaGeneral1})\)--\((\ref{itemopeninglemmaGeneral2})\).

We claim that for some \(z \in Q^{m-\ell}_{\hat\rho}\), the function \(\Phi :  \R^m \to \R^m\) defined for \(x = (x', x'') \in  \R^\ell \times \R^{m - \ell}\) by
\[
\Phi(x) = (x', \zeta(x''))
\]
satisfies property~\((\ref{itemopeninglemmaGeneral3})\).

For this purpose, let
\(\Psi : \R^m \times Q_{\Hat{\rho}}^{m-\ell} \to \R^m\) be the function defined for \(x = (x', x'') \in  \R^\ell \times \R^{m-\ell}\) and \(z \in  Q_{\Hat{\rho}}^{m-\ell}\) by
\[
\Psi(x, z) =  (x', \Tilde \zeta(x''+z)-z) .
\]
For every measurable function \(f : A \times Q_{\overline{\rho}}^{m-\ell} \to \R\), we have by Fubini's theorem,
\begin{multline*}
\int\limits_{Q_{\Hat{\rho}}^{m-\ell}} \norm{ f \circ \Psi_z}_{L^1(A \times Q_{\overline{\rho}}^{m-\ell})} \dif z\\
 = \int\limits_A \biggl[\int\limits_{Q_{\Hat{\rho}}^{m-\ell}} \biggl(\int\limits_{Q_{\overline{\rho}}^{m-\ell}} \bigabs{f(x', \Tilde\zeta(x''+z)-z) } \dif x'' \biggr) \dif z \biggr] \dif x'.
\end{multline*}
Given \(x' \in A\), we apply Lemma~\ref{lemmaOpeningEstimate} with \(U = Q_{\overline{\rho}}^{m-\ell}\), \(V = Q_{\Hat{\rho}}^{m-\ell}\), \(W = Q_{\overline{\rho}}^{m-\ell}\), and
\(
\Tilde\zeta.
\)
We deduce that
\[
\int\limits_{Q_{\Hat{\rho}}^{m-\ell}} \biggl(\int\limits_{Q_{\overline{\rho}}^{m-\ell}} \bigabs{f(x', \Tilde\zeta(x''+z)-z) } \dif x'' \biggr) \dif z 
\le C_1 \int\limits_{Q_{\overline{\rho}}^{m-\ell}} |f(x', x'')| \dif x''.
\]
Thus,
\[
\int\limits_{Q_{\Hat{\rho}}^{m-\ell}} \norm{ f \circ \Psi_z}_{L^1(A \times Q_{\overline{\rho}}^{m-\ell})} \dif z \le C_1 \norm{f}_{L^1( A \times Q_{\overline{\rho}}^{m-\ell})}.
\]

\medskip
By Lemma~\ref{lemmaOpeningSobolev}, for almost every \(z \in Q_{\Hat{\rho}}^{m-\ell}\), \(u \circ \Psi_{z} \in W^{k, p}(A \times Q_{\overline{\rho}}^{m-\ell}; \R^\nu) \) and for every \(j \in \{1, \dots, k\}\),
\[
\int\limits_{Q_{\Hat{\rho}}^{m-\ell}} \norm{D^j (u \circ \Psi_{z})}_{L^p({A \times Q_{\overline{\rho}}^{m-\ell}})} \dif z \leq C_2 \sum_{i=1}^j \norm{D^i u}_{L^p({A \times Q_{\overline{\rho}}^{m-\ell}})}.
\]
We may thus find some \(z \in Q_{\Hat{\rho}}^{m-\ell}\) such that \(u \circ \Psi_{z} \in W^{k, p}(A \times Q_{\overline{\rho}}^{m-\ell}; \R^\nu)\) and for every \(j \in \{1, \dots, k\}\),
\[
\norm{D^j (u \circ \Psi_{z})}_{L^p({A \times Q_{\overline{\rho}}^{m-\ell}})} \leq C_3 \sum_{i=1}^j \norm{D^i u}_{L^p({A \times Q_{\overline{\rho}}^{m-\ell}})}.
\]
The function \(\zeta\) defined in terms of this point \(z\) satisfies the required properties.
\end{proof}

\begin{addendum}[{Proposition~\ref{openingpropGeneral}}]
\label{addendumW1kp}
Let \(\cK^m\) be a cubication containing \(\cU^m\) and let \(q \ge 1\). If \(u \in W^{1, q}(K^m + Q^m_{2\rho\eta}; \R^\nu)\), then the map \(\Phi : \R^m \to \R^m\) can be chosen with the additional property that \(u \circ \Phi \in W^{1, q}(K^m + Q^m_{2\rho\eta}; \R^\nu)\) and for every \(\sigma^m \in \cK^m\),
\begin{equation*}
\norm{D(u \circ \Phi)}_{L^{q}(\sigma^m + Q^m_{2\rho\eta})} \leq C'' \norm{Du}_{L^{q}(\sigma^m + Q^m_{2\rho\eta})},
\end{equation*}
for some constant \(C'' > 0\) depending on \(m\), \(q\) and \(\rho\).
\end{addendum}

\begin{proof}
Since \(u \in W^{1, q}(U^\ell + Q^m_{2\rho\eta}; \R^\nu)\), we may apply Proposition~\ref{openingpropGeneral} with \(k = 1\) and \(p = q\) in order to obtain a map \(\Phi: \R^m \to \R^m\) such that \(u \circ \Phi \in W^{1, q}(U^\ell + Q^m_{2\rho\eta}; \R^\nu)\) and for every \(\sigma^\ell \in \cU^\ell\),
\begin{equation*}
\norm{D(u \circ \Phi)}_{L^{q}(\sigma^\ell + Q^m_{2\rho\eta})} \leq C \norm{Du}_{L^{q}(\sigma^\ell + Q^m_{2\rho\eta})}.
\end{equation*}
Since the choice of the point \(z\) in the proof of Proposition~\ref{openinglemmaGeneral} can be done in a set of positive measure, we may do so by keeping the properties we already have for \(W^{k, p}\).

For every \(\sigma^m \in \cK^m\), if \(\sigma^{m, \ell}\) denotes the skeleton of dimension \(\ell\) of \(\sigma^m\), then by additivity of the integral,
\begin{equation*}
\norm{D(u \circ \Phi)}_{L^{q}((\sigma^{m, \ell} \cap U^\ell) + Q^m_{2\rho\eta})} \leq C \norm{Du}_{L^{q}((\sigma^{m, \ell} \cap U^\ell) + Q^m_{2\rho\eta})}.
\end{equation*}
Since \(\Phi\) coincides with the identity map in \((\sigma^m + Q^m_{2\rho\eta}) \setminus ((\sigma^{m, \ell} \cap U^\ell) + Q^m_{2\rho\eta})\),
\[
\norm{D(u \circ \Phi)}_{L^{q}(\sigma^m + Q^m_{2\rho\eta})} \leq C \norm{Du}_{L^{q}(\sigma^m + Q^m_{2\rho\eta})}.
\]
This concludes the proof.
\end{proof}

\begin{addendum}[{Proposition~\ref{openingpropGeneral}}]
\label{addendumVMO}
Let \(\cK^m\) be a cubication containing \(\cU^m\). 
If \(u\in W^{1, kp}(K^m + Q^m_{2\rho\eta}; \R^\nu)\), then the map \(\Phi : \R^m \to \R^m\) given by Proposition~\ref{openingpropGeneral} and Addendum~\ref{addendumW1kp} above with \(q = kp\) satisfies 
\[
  \lim_{r \to 0} \sup_{Q_r^m(a) \subset U^\ell + Q^m_{\rho\eta}} \frac{r^{\frac{\ell}{kp} - 1} }{\abs{Q_r^m}^2} \int\limits_{Q_r^m (a)}\int\limits_{Q_r^m (a)} \abs{u \circ \Phi(x) - u \circ \Phi (y)} \dif x \dif y = 0
\]
and for every \(\sigma^m \in \cU^m\) and for every \(a \in \sigma^m\) such that \(Q_r^m (a) \subset U^\ell + Q^m_{\rho\eta}\),
\[
\frac{1}{\abs{Q_r^m}^2} \int\limits_{Q_r^m (a)}\int\limits_{Q_r^m (a)} \abs{u \circ \Phi(x) - u \circ \Phi (y)} \dif x \dif y \le \frac{C''' r^{1-\frac{\ell}{kp}}}{\eta^{\frac{m-\ell}{kp}}} \norm{Du}_{L^{kp}(\sigma^m + Q_{2\rho\eta}^m)},
\]
for some constant \(C''' > 0\) depending on \(m\), \(kp\) and \(\rho\).
\end{addendum}

If \(kp \ge \ell\), then the limit above implies that \(u \circ \Phi\) belongs to the space of functions of vanishing mean oscillation \(\textrm{VMO}(U^\ell + Q^m_{\rho\eta}; \R^\nu)\) and the estimate yields an estimate on the BMO seminorm on the domain \(U^\ell + Q^m_{\rho\eta}\) as defined by Jones \cite{Jones1980}. 
If \(kp > \ell > 0\), then the estimate  implies that \(u \circ \Phi \in C^{0, 1 - \frac{\ell}{kp}}(U^\ell + Q^m_{\rho\eta}; \R^\nu)\) with an upper bound on the \(C^{0, 1 - \frac{\ell}{kp}}\) seminorm of \(u \circ \Phi\) \cite{Campanato1963}.
The estimates of this addendum are not really useful when \(kp < \ell\) since in this case \(\lim\limits_{r \to 0} r^{1 - \frac{\ell}{kp}} = +\infty\).

\begin{proof}[Proof of Addendum~\ref{addendumVMO}] 
Fix \(Q^m_r (a) \subset U^\ell + Q^{m}_{\rho \eta}\). Then \(a\in U^\ell+Q^{m}_{\rho\eta-r}\). Hence there exists an \(\ell\) dimensional face  \(\tau^\ell \in \cU^\ell\) such that \(Q^m_r (a) \subset \tau^\ell + Q^{m}_{\rho \eta} \). 
Without loss of generality, we may assume that \(\tau^\ell=Q^\ell_{\eta}\times\{0^{m-\ell}\}\subset \R^\ell\times \R^{m-\ell}\). 
From Proposition~\ref{openingpropGeneral} \((\ref{itemgenopeningprop1})\), the map \(\Phi\) is constant on the \(m-\ell\) dimensional cubes of radius \(\rho\eta\) which are orthogonal to \(Q^\ell_{(1+\rho)\eta}\times \{0^{m-\ell}\}\). 
Writing \(Q^m_r(a) = Q^\ell_r(a') \times Q^{m-\ell}_r(a'')\), then  \(u\circ \Phi\) only depends on the first \(\ell\) dimensional variables in \(Q^m_r(a)\).
Let \(v : Q_{(1 + \rho)\eta}^\ell \to \R^\nu\) be the function defined by
\[
v(x') = (u \circ \Phi)(x', a'').
\]
By Addendum~\ref{addendumW1kp} above with \(q = kp\), \(u \circ \Phi \in W^{1, kp}(Q_{(1 + \rho)\eta}^\ell \times Q_{\rho\eta}^{m-\ell}; \R^\nu)\), whence
\[
v \in W^{1, kp}(Q_{(1 + \rho)\eta}^\ell; \R^\nu).
\]
Note that
\begin{multline*}
\frac{1}{\abs{Q_r^m}^2}\int\limits_{Q^m_{r}(a)}\int\limits_{Q^m_{r}(a)} \abs{u \circ \Phi(x) - u \circ \Phi(y)} \dif x \dif y\\
 = \frac{1}{\abs{Q_r^\ell}^2}\int\limits_{Q^\ell_{r}(a')}\int\limits_{Q^\ell_{r}(a')} \abs{v(x') - v(y')} \dif x' \dif y'.
\end{multline*}
By the Poincar\'e-Wirtinger inequality, 
\[
\frac{1}{\abs{Q_r^\ell}^2}\int\limits_{Q^\ell_{r}(a')}\int\limits_{Q^\ell_{r}(a')} \abs{v(x') - v(y')} \dif x' \dif y' \le C_1 r^{1 - \frac{\ell}{kp}} \norm{Dv}_{L^{kp}(Q^\ell_{r}(a'))}.
\]
Thus,
\[
\frac{1}{\abs{Q_r^m}^2}\int\limits_{Q^m_{r}(a)}\int\limits_{Q^m_{r}(a)} \abs{u \circ \Phi(x) - u \circ \Phi(y)} \dif x \dif y \le C_1 r^{1 - \frac{\ell}{kp}} \norm{Dv}_{L^{kp}(Q^\ell_{r}(a'))}
\]
and this implies the first part of the conclusion.

In order to get the estimate of the oscillation of \(u \circ \Phi\) in terms of \(\norm{D(u\circ \Phi)}_{L^{kp}}\), note that
\[
\norm{D(u\circ \Phi)}_{L^{kp}(Q^\ell_r(a') \times Q^{m-\ell}_{\rho\eta}(a''))} = (2\rho\eta)^\frac{m-\ell}{kp} \norm{Dv}_{L^{kp}(Q^\ell_{r}(a'))}.
\]
This implies for any \(\sigma^m\in\mathcal{U}^m\) such that \(\tau^\ell\subset \sigma^m\)
\[
\begin{split}
\norm{Dv}_{L^{kp}(Q^\ell_{r})} 
& = \frac{1}{(2\rho\eta)^\frac{m-\ell}{kp}} \norm{D( u\circ \Phi)}_{L^{kp}(Q^\ell_r \times Q^{m-\ell}_{\rho\eta})}\\
& \leq  \frac{1}{(2\rho\eta)^{\frac{m-\ell}{kp}}} \norm{D(u\circ \Phi)}_{L^{kp}(\sigma^\ell+Q^m_{\rho\eta})}\\
& \leq  \frac{1}{(2 \rho\eta)^{\frac{m - \ell}{kp}}} \norm{D(u\circ \Phi)}_{L^{kp}(\sigma^m+Q^m_{\rho\eta})}.
\end{split}
\]
Thus,
\[
  \frac{1}{\abs{Q_r^m}^2}\int\limits_{Q^m_{r}(a)}\int\limits_{Q^m_{r}(a)} \abs{u \circ \Phi(x) - u \circ \Phi(y)} \dif x \dif y  
  \le \frac{C_2 r^{1 - \frac{\ell}{k p}}} {(\rho\eta)^{\frac{m - \ell}{kp}}} \norm{D(u\circ \Phi)}_{L^{kp}(\sigma^m + Q_{\rho\eta}^m)}.
\]
By Addendum~\ref{addendumW1kp} above,
\[
\norm{D(u\circ \Phi)}_{L^{kp}(\sigma^m + Q_{\rho\eta}^m)} \le C_3 \norm{Du}_{L^{kp}(\sigma^m + Q_{2\rho\eta}^m)}.
\]
This proves the estimate that we claimed.
\end{proof}

\section{Adaptive smoothing}

Given \(u \in W^{k, p}(\Omega; \R^\nu)\), we would like to consider  a convolution of \(u\) with a parameter which may depend on the point where we compute the convolution itself. 
The main reason is that we want to choose the convolution parameter by taking into account the mean oscillation of \(u\): we choose a large parameter where \(u\) does not oscillate too much and a small parameter elsewhere.

For this purpose, consider a function \(u \in L^1(\Omega; \R^\nu)\). Let \(\varphi\) be a \emph{mollifier}, in other words,
\[
\varphi \in C_c^\infty(B_1^m), \quad \varphi \ge 0\ \text{in \(B_1^m\)} \quad \text{and} \quad \int\limits_{B_1^m} \varphi = 1.
\]
For every \(s \ge 0\) and for every \(x \in \Omega\) such that \(d(x, \partial\Omega) \ge s\), we may consider the convolution
\[
(\varphi_s \ast u)(x) = \int\limits_{B_1^m} \varphi(z) u(x + s z) \dif z.
\]
We may keep in mind that with this definition,
\[
(\varphi_0 \ast u)(x) = \int\limits_{B_1^m} \varphi(z) \dif z \, u(x) = u(x).
\]
This way of writing the convolution has the advantage that we may treat the cases \(s = 0\) and \(s > 0\) using the same formula. 

We now introduce a nonconstant parameter in the convolution given by a nonnegative function \(\psi \in C^\infty(\Omega)\). The convolution
\[
\varphi_\psi \ast u : \big\{x\in \Omega : \dist{(x, \partial\Omega)} \ge \psi(x) \big\} \to \R^\nu
\] 
is well-defined and if \(\psi(a) > 0\) and \(\abs{D\psi (a)} < 1\) at some point \(a \in \Omega\), then by a change of variable in the integral the map \(\varphi_\psi \ast u\) is smooth in a neighborhood of \(a\).

\begin{proposition}
\label{lemmaConvolutionEstimatesLp}
Let \(\varphi \in C_c^\infty(B_1^m)\) be a mollifier and let \(\psi \in C^\infty(\Omega)\) be a nonnegative function such that \(\norm{D\psi}_{L^\infty(\Omega)} < 1\).
Then, for every \(u\in L^{p}(\Omega; \R^\nu)\) and for every open set \(\omega \subset \big\{x\in \Omega : \dist{(x, \partial\Omega)} \ge \psi(x) \big\}\), \(\varphi_{\psi} \ast u \in L^{p}(\omega; \R^\nu)\), 
\begin{equation*}
\label{ineqConvolLp}
\norm{\varphi_\psi \ast u}_{L^p(\omega)} \le \frac{1}{(1 - \norm{D\psi}_{L^\infty(\omega)})^\frac{1}{p}} \norm{u}_{L^p(\Omega)},
\end{equation*}
and
\[
\norm{\varphi_{\psi} \ast u - u}_{L^p(\omega)}\\
\leq \sup_{v \in B_1^m}{\norm{\tau_{\psi v}u - u}_{L^p(\omega)}},
\]
where \(\tau_{\psi v} u (x) = u(x + \psi(x)v)\).
\end{proposition}

For \(p > 1\), it is possible to obtain an estimate for \(\norm{\varphi_\psi \ast u}_{L^p(\omega)}\) without any dependence on \(\psi\) by the theory of the Hardy-Littlewood maximal function \cite{Stein}; this approach fails for \(p=1\).

In the context of the proposition above, one can  prove in a  standard way the following statement:  given \(u\in L^{p}(\Omega;\R^\nu)\),  \(0 \le \beta < 1\) and  \(\varepsilon >0\), there exists \(\delta>0\) such that for any nonnegative function \(\psi \in C^\infty(\Omega)\) satisfying \(\norm{\psi}_{L^\infty(\Omega)} \le \delta\)  and \(\norm{D\psi}_{L^\infty(\Omega)} \le \beta\), and for every open set \(\omega \subset \big\{x\in \Omega : \dist{(x, \partial\Omega)} \ge \psi(x) \big\}\),
\[
\sup_{v \in B_1^m}{\norm{\tau_{\psi v}u - u}_{L^p(\omega)}} \le \varepsilon.
\]

We may pursue these estimates for maps in \(W^{k, p}(\Omega; \R^\nu)\):

\begin{proposition}
\label{lemmaConvolutionEstimates}
Let \(\varphi \in C_c^\infty(B_1^m)\) be a mollifier and let \(\psi \in C^\infty(\Omega)\) be a nonnegative function such that \(\norm{D\psi}_{L^\infty(\Omega)} < 1\).
For every \(k \in \N_*\), for every  \(u\in W^{k, p}(\Omega; \R^\nu)\) and for every open set \(\omega \subset \big\{x\in \Omega : \dist{(x, \partial\Omega)} \ge \psi(x) \big\}\),  \(\varphi_{\psi} \ast u \in W^{k, p}(\omega; \R^\nu)\) and for every \(j \in \{1, \dots, k\}\),
\[
\eta^{j} \norm{D^j(\varphi_{\psi} \ast u)}_{L^p(\omega)} \leq \frac{C}{(1 - \norm{D\psi}_{L^\infty(\omega)})^\frac{1}{p}} \sum_{i=1}^j  \eta^{i} \norm{D^i u}_{L^p(\Omega)},
\]
and
\begin{multline*}
\eta^{j} \norm{D^j(\varphi_{\psi} \ast u) - D^j u}_{L^p(\omega)}\\
\leq \sup_{v \in B_1^m}{\eta^{j} \norm{\tau_{\psi v}(D^j u) - D^j u}_{L^p(\omega)}} + \frac{C'}{(1 - \norm{D\psi}_{L^\infty(\omega)})^\frac{1}{p}} \sum_{i=1}^j  \eta^{i} \norm{D^i u}_{L^p(A)},
\end{multline*}
for some constants \(C > 0\) and \(C' > 0\) depending on \(m\), \(k\) and \(p\), where
\[
 A = \bigcup_{x \in  \omega \cap\supp{D\psi}} B_{\psi(x)}^m(x)
\] 
and \(\eta > 0\) is such that for every \(j \in \{2, \dotsc, k\}\),
\[
  \eta^{j} \norm{D^j \psi}_{L^\infty(\omega)} \le \eta.
\]
\end{proposition}

\begin{proof}
We only prove the second estimate. We assume for simplicity that \(u \in C^\infty(\Omega; \R^\nu)\).
For every \(x \in \omega\),
\[
(\varphi_\psi \ast u)(x) - u(x) = \int\limits_{B_1^m}  \varphi(z) \big[ u(x + \psi(x) z) - u(x)] \dif z .
\]
For every \(j \in \{1, \dots, k\}\), we have by the chain rule for higher order derivatives,
\begin{multline*}
\abs{D^j (\varphi_\psi \ast u)(x) - D^j u(x)}\\
 \le \int\limits_{B_1^m}  \varphi(z) \bigabs{D^ju(x + \psi(x) z)\circ (\Id +D\psi(x)\otimes z)^j - D^ju(x)} \dif z 
\\
+
C_1\sum_{i=1}^{j-1} \sum_{
\substack{\alpha_1 + 2 \alpha_2 + \dotsb + j \alpha_j = j\\
\alpha_1 + \alpha_2 + \dotsb + \alpha_j = i
}
}
(1 + \abs{D\psi(x)})^{\alpha_1} \abs{D^2\psi(x)}^{\alpha_2} \dotsm \abs{D^j\psi(x)}^{\alpha_j}
\times\\
\times \int\limits_{B_1^m} \varphi(z) \abs{D^i u(x + \psi(x)z)} \dif z.
\end{multline*}
Since \(\norm{D\psi}_{L^\infty(\Omega)} \le 1\), for every \(z \in B_1^m\),
\[
\bigabs{(\Id +D\psi(x)\otimes z)^j - \Id} \le C_2 \abs{D\psi(x)},
\]
and we have
\begin{multline*}
\abs{D^j (\varphi_\psi \ast u)(x) - D^j u(x)}\\
\\
\le \int\limits_{B_1^m}  \varphi(z) \bigabs{D^ju(x + \psi(x) z) - D^ju(x)} \dif z + C_2 \abs{D\psi(x)}\int\limits_{B_1^m}  \varphi(z) \bigabs{D^ju(x + \psi(x) z)} \dif z\\
+
C_1 \sum_{i=1}^{j-1} \sum_{
\substack{\alpha_1 + 2 \alpha_2 + \dotsb + j \alpha_j = j\\
\alpha_1 + \alpha_2 + \dotsb + \alpha_j = i
}
}
(1 + \abs{D\psi(x)})^{\alpha_1} \abs{D^2\psi(x)}^{\alpha_2} \dotsm \abs{D^j\psi(x)}^{\alpha_j} \times\\
\times \int\limits_{B_1^m} \varphi(z) \abs{D^i u(x + \psi(x)z)} \dif z.
\end{multline*}
The second and the third terms in the right hand side are supported on \(\supp{D\psi}\) since \(\alpha_s \ne 0\) for some \(s > 1\). Moreover, by the choice of \(\eta\),
\[
\begin{aligned}
(1 + \abs{D\psi(x)})^{\alpha_1} \abs{D^2\psi(x)}^{\alpha_2} \dotsm \abs{D^j\psi(x)}^{\alpha_j}
& \le (1 + 1) ^{\alpha_1} \Big(\frac{\eta}{\eta^2}\Big)^{\alpha_2} \dotsm \Big(\frac{\eta}{\eta^j}\Big)^{\alpha_j}\\
& = 2^{\alpha_1} \frac{\eta^{\alpha_1 + \alpha_2 + \dotsb + \alpha_j}}{\eta^{\alpha_1 + 2 \alpha_2 + \dotsb + j \alpha_j}} = 2^{\alpha_1} \frac{\eta^i}{\eta^j} \le 2^{j} \frac{\eta^i}{\eta^j}.
\end{aligned}
\]
Therefore,
\begin{multline*}
\abs{D^j (\varphi_\psi \ast u)(x) - D^j u(x)}
 \le \int\limits_{B_1^m}  \varphi(z) \bigabs{D^ju(x + \psi(x) z) - D^ju(x)} \dif z 
\\
+ C_3 \sum_{i=1}^j \frac{\eta^i}{\eta^j} \chi_{\supp{D\psi}}(x)
\int\limits_{B_1^m} \varphi(z) \abs{D^i u(x + \psi(x)z)} \dif z.
\end{multline*}
By the Minkowski inequality,
\begin{multline*}
\bigg( \int\limits_{\omega} \bigg(  \int\limits_{B_1^m}  \varphi(z) \abs{D^ju(x + \psi(x) z) - D^j u(x)} \dif z \bigg)^p \dif x \bigg)^{\frac{1}{p}} \\
\begin{aligned}
& \le  \int\limits_{B_1^m} \bigg(  \int\limits_{\omega} \abs{D^ju(x + \psi(x) z) - D^ju(x)}^p \dif x \bigg)^{\frac{1}{p}} \varphi(z) \dif z\\
&  \le  \sup_{v \in B_1^m}{\norm{\tau_{\psi v}(D^j u) - D^j u}_{L^p(\omega)}} \int\limits_{B_1^m} \varphi(z) \dif z\\
& = \sup_{v \in B_1^m}{\norm{\tau_{\psi v}(D^j u) - D^j u}_{L^p(\omega)}} ,
\end{aligned}
\end{multline*}
and for every \(i \in \{1, \dots, j\}\), we also have
\begin{multline*}
\bigg( \int\limits_{\omega \cap \supp{D\psi}} \bigg(  \int\limits_{B_1^m}  \varphi(z) \abs{D^i u(x + \psi(x) z)} \dif z \bigg)^p \dif x \bigg)^{\frac{1}{p}} \\
\le  \int\limits_{B_1^m}  \varphi(z) \bigg(  \int\limits_{\omega \cap \supp{D\psi}} \abs{D^i u(x + \psi(x) z)}^p \dif x \bigg)^{\frac{1}{p}}  \dif z. 
\end{multline*}
Using the change of variable \(y = x + \psi(x) z\) with respect to the variable \(x\),  we deduce by definition of \(A\) that
\begin{multline*}
\bigg( \int\limits_{\omega \cap \supp{D\psi}} \bigg(  \int\limits_{B_1^m}  \varphi(z) \abs{D^i u(x + \psi(x) z)} \dif z \bigg)^p \dif x \bigg)^{\frac{1}{p}} \\
\begin{aligned}
& \le \int\limits_{B_1^m}  \varphi(z) \bigg( \frac{1}{1 - \norm{D\psi}_{L^\infty(\omega)}}  \int\limits_{A} \abs{D^i u(y)}^p \dif y \bigg)^{\frac{1}{p}}  \dif z\\
& = \frac{1}{(1 - \norm{D\psi}_{L^\infty(\omega)})^\frac{1}{p}} \norm{D^i u}_{L^p(A)}. 
\end{aligned}
\end{multline*}
This gives the desired estimate for \(u \in C^\infty(\Omega; \R^\nu)\). The case of functions in \(W^{k, p}(\Omega; \R^\nu)\) follows by density.
\end{proof}

%%%%%%%%%%%%%%%%%%%%%%%%%%%%%%%%%%%%%%%%%%%%%%%%%%%%%
%%%%%%%%%%%%%%%%%%%%%%%%%%%%%%%%%%%%%%%%%%%%%%%%%%%%%
%%%%%%%%%%%%%%%%%%%%%%%%%%%%%%%%%%%%%%%%%%%%%%%%%%%%%

\section{Thickening}
Given a map \(u \in W^{k, p}(U^m; \R^\nu)\) which behaves nicely near the skeleton \(U^\ell\), we would like to construct a map \(u \circ \Phi\) that does not depend on the values of \(u\) away from the skeleton \(U^\ell\). The price to pay is that the map \(u \circ \Phi\) will be singular on the dual skeleton \(T^{\ell^*}\); these singularities will however be mild enough to allow \(u \circ \Phi\) to be in \(R_{\ell^*} (U^m; \R^\nu)\) and to satisfy \(W^{k, p}\) estimates for \(k p < \ell +1\). 
The thickening construction is related to homogenization of functions on cubes that are used in the study of density problems for \(k = 1\) \cite{Bethuel, BethuelZheng, Hang-Lin}.

The precise meaning of dual skeleton we use is the following:

\begin{definition}
Given \(\ell \in \{0, \dotsc, m-1\}\) and the \(\ell\) dimensional skeleton \(\mathcal{S}^\ell\) of a cubication \(\cS^m\), the dual skeleton \(\mathcal{T}^{\ell^*}\) of \(\cS^\ell\) is the skeleton of dimension \(\ell^* = m - \ell - 1\) composed of all cubes of the form \(\sigma^{\ell^*} + x - a\), where \(\sigma^{\ell^*} \in \cS^{\ell^*}\), \(a\) is the center and \(x\) the vertex of a cube of \(\cS^m\).
\end{definition}

The integer \(\ell^*\) gives the greatest dimension such that \(S^\ell \cap T^{\ell^*} = \emptyset.\)

\medskip

The proposition below provides the main properties of the map \(\Phi\):

\begin{proposition}\label{propositionthickeningfromaskeleton}
Let \(\ell \in \{0, \dotsc, m-1\}\), \(\eta > 0\), \(0  < \rho < 1\), \(\mathcal{S}^m\) be a cubication of \(\R^m\)  of radius \(\eta\), \(\mathcal{U}^m\) be a subskeleton of \(\cS^m\) and \(\cT^{\ell^*}\) be the dual skeleton of \(\cU^\ell\).
There exists a smooth map \(\Phi : \R^m \setminus T^{\ell^*} \to \R^m\) such that 
\begin{enumerate}[$(i)$]
\item \(\Phi\) is injective,
\item for every \(\sigma^m \in \cS^m\), \(\Phi(\sigma^m \setminus T^{\ell^*}) \subset \sigma^m \setminus T^{\ell^*}\),
\label{itempropositionthickeningfromaskeleton2}
\item \(\Supp{\Phi} \subset U^m + Q^{m}_{\rho\eta}\) and \(\Phi(U^m\setminus T^{\ell^*}) \subset U^\ell + Q^{m}_{\rho\eta}\),
\label{itempropositionthickeningfromaskeleton3}
\item for every \(j \in \N_*\) and for every \(x \in \R^m \setminus T^{\ell^*}\),
\[
\abs{D^j \Phi(x)} \le  \frac{C\eta}{\bigl(\dist(x, T^{\ell^*})\bigr)^j},
\]
for some constant \(C > 0\) depending on \(j\), \(m\) and \(\rho\),
\label{itempropositionthickeningfromaskeleton5}
\item for every \(0 < \beta < \ell + 1\), for every \(j \in \N_*\) and for every \(x \in \R^m \setminus T^{\ell^*}\),
\[
\eta^{j-1} \abs{D^j \Phi(x)} \le  C' \bigl(\jac{\Phi}(x)\bigr)^\frac{j}{\beta},
\]
for some constant \(C' > 0\) depending on \(\beta\), \(j\), \(m\) and \(\rho\).
\label{itempropositionthickeningfromaskeleton4}
\end{enumerate}
\end{proposition}

This proposition gives \(W^{k, p}\) bounds on \(u \circ \Phi\) for every \(W^{k, p}\) function \(u\).
The proposition and the corollary below will be applied in the proof of Theorem~\ref{theoremDensityManifoldNontrivial} with \( \ell = \lfloor kp \rfloor \).

\begin{corollary}
\label{corollaryEstimateThickening}
Let \(\Phi : \R^m \setminus T^{\ell^*} \to \R^m\) be the map given by Proposition~\ref{propositionthickeningfromaskeleton}. If  \(\ell+1 > kp\), then for every \(u \in W^{k, p}(U^m + Q^{m}_{\rho\eta}; \R^\nu)\), \(u \circ \Phi \in W^{k, p}(U^m + Q^{m}_{\rho\eta}; \R^\nu)\) and for every \(j \in \{1, \dotsc, k\}\),
\[
 \eta^{j} \norm{D^j(u \circ \Phi)}_{L^p(U^m + Q^{m}_{\rho\eta})} \leq C'' \sum_{i=1}^j  \eta^{i} \norm{D^i u}_{L^p(U^m + Q^{m}_{\rho\eta})},
\]
for some constant \(C'' > 0\) depending on \(m\), \(k\), \(p\) and \(\rho\).
\end{corollary}

\begin{proof}
We first establish the estimate for a map \(u\) in \(C^{\infty}(U^m + Q^{m}_{\rho\eta} ; \R^{\nu})\). By the chain rule for higher-order derivatives, for every \(j \in \{1, \dots, k\}\) and for every \(x \in U^m \setminus T^{\ell^*}\),
\[
\abs{D^j (u \circ \Phi) (x)}^p \le C_1 \sum_{i=1}^j \sum_{\substack{1 \le t_1 \le \dotsc \le t_i\\ t_1 + \dotsb + t_i = j}} \abs{D^i u(\Phi(x))}^p \abs{D^{t_1} \Phi(x)}^p \dotsm \abs{D^{t_i} \Phi(x)}^p.
\]
Let \(0 < \beta < \ell + 1\). If \(1 \le t_1 \le \dotsc \le t_i\) and \(t_1 + \dotsb + t_i = j \), then by  property~\((\ref{itempropositionthickeningfromaskeleton4})\) of Proposition~\ref{propositionthickeningfromaskeleton},
\[
\abs{D^{t_1} \Phi(x)}^p \dotsm \abs{D^{t_i}\Phi(x)}^p \le C_2 \frac{\bigl(\jac{\Phi}(x)\bigr)^\frac{t_1 p}{\beta}}{\eta^{(t_1 - 1)p}} \dotsm \frac{\bigl(\jac{\Phi}(x)\bigr)^\frac{t_i p}{\beta}}{{\eta^{(t_i - 1)p}}} = C_2 \frac{\bigl(\jac{\Phi}(x)\bigr)^\frac{jp}{\beta}}{\eta^{(j-i)p}}.
\]
Since \(kp < \ell + 1\), we may take \(\beta = jp\). Thus,
\[
\abs{D^{t_1} \Phi(x)}^p \dotsm \abs{D^{t_i}\Phi(x)}^p \le C_2 \frac{\jac{\Phi}(x)}{\eta^{(j-i)p}}
\]
and this implies
\[
\eta^{jp}\abs{D^j (u \circ \Phi) (x)}^p \le C_3 \sum_{i=1}^j \eta^{ip}\abs{D^i u(\Phi(x))}^p \jac{\Phi}(x).
\]
Since \(\Phi\) is injective and \(\Supp{\Phi} \subset U^m + Q^m_{\rho\eta}\), we have \(\Phi\big( (U^m + Q^{m}_{\rho\eta}) \setminus T^{\ell^*} \big) \subset U^m + Q^m_{\rho\eta}\).
Thus, by the change of variable formula,
\[
\begin{aligned}
\int\limits_{(U^m + Q^{m}_{\rho\eta}) \setminus T^{\ell^*}} \eta^{jp}\abs{D^j (u \circ \Phi)}^p 
&\le  C_3 \sum_{i=1}^j \int\limits_{(U^m + Q^{m}_{\rho\eta}) \setminus T^{\ell^*}} \eta^{ip}\abs{(D^i u) \circ \Phi}^p \jac{\Phi}\\
&\le C_3 \sum_{i=1}^j \int\limits_{U^m + Q^{m}_{\rho\eta}} \eta^{ip}\abs{D^i u}^p
\end{aligned}
\]
and \(u \circ \Phi \in W^{k, p}((U^m + Q^{m}_{\rho\eta}) \setminus T^{\ell^*}; \R^\nu)\). Since \(\ell > 0\), the dimension of the skeleton \(T^{\ell^*}\) is strictly less than \(m - 1\). Thus, 
\(u \circ \Phi \in W^{k, p}(U^m + Q^{m}_{\rho\eta}; \R^\nu)\).
By density of smooth maps in \(W^{k, p}(U^m + Q^{m}_{\rho\eta}; \R^\nu)\), we deduce that for every \(u \in W^{k, p}(U^m + Q^{m}_{\rho\eta}; \R^\nu)\), the function \(u \circ \Phi\) also belongs to this space and satisfies the estimate above.
\end{proof}

We describe the construction of the map \(\Phi\) given by Proposition~\ref{propositionthickeningfromaskeleton} in the case of only one \(\ell\) dimensional cube:

\begin{proposition}
\label{lemmaThickeningFaceFromPrimalSkeleton}
Let \(\ell \in \{1, \dotsc, m\}\), \(\eta > 0\), \(0 < \underline{\rho}< \rho <\overline{\rho} < 1\) and \(T = \{0^\ell\} \times Q^{m-\ell}_{\rho\eta}\).
There exists a smooth function \(\lambda : \R^m \setminus T \to [1,\infty)\) such that if \(\Phi : \R^m \setminus T  \to \R^m\) is defined for \(x = (x', x'') \in (\R^\ell \times \R^{m - \ell}) \setminus T \) by
\[
\Phi(x) = (\lambda(x)x', x''),
\]
then
\begin{enumerate}[$(i)$]
\item \(\Phi\) is injective,
\label{itemthickeninglemma0.5}
\item 
\label{itemthickeninglemma2}
\(\Supp{\Phi} \subset Q^\ell_{(1-\rho)\eta} \times Q^{m-\ell}_{\rho\eta}\),
\item 
\(
 \Phi\big((Q^{\ell}_{(1-\rho)\eta }\times Q^{m-\ell}_{\underline{\rho}\eta}) \setminus T\big) \subset (Q^{\ell}_{(1-\rho)\eta } \setminus Q^{\ell}_{(1-\overline{\rho})\eta})\times Q^{m-\ell}_{\underline{\rho}\eta},
\)
\label{itemthickeninglemma1}
\item
for every \(j \in \N_*\) and for every 
\( x = (x', x'') \in (Q^\ell_{(1-\rho)\eta} \times Q^{m-\ell}_{\rho\eta}) \setminus T\), 
\[
\abs{D^j \Phi(x)} \le \frac{C\eta}{\abs{x'}^j},
\]
for some constant \(C > 0\) depending on \(j\), \(m\), \(\underline{\rho}\), \(\rho\) and \(\overline{\rho}\),
\label{itemthickeninglemma3}
\item
\label{itemthickeninglemma4}
for every \(0 < \beta < \ell\), for every \(j \in \N_*\) and for every 
\( x \in (Q^\ell_{(1-\rho)\eta} \times Q^{m-\ell}_{\rho\eta}) \setminus T\), 
\[
\eta^{j-1} \abs{D^j \Phi(x)} \le  C' \bigl(\jac{\Phi}(x)\bigr)^\frac{j}{\beta},
\]
for some constant \(C' > 0\) depending on \(\beta\), \(j\), \(m\), \(\underline{\rho}\), \(\rho\) and \(\overline{\rho}\).
\end{enumerate}
\end{proposition}

We temporarily admit Proposition~\ref{lemmaThickeningFaceFromPrimalSkeleton} and we prove Proposition~\ref{propositionthickeningfromaskeleton}.

\begin{proof}[Proof of Proposition~\ref{propositionthickeningfromaskeleton}]
We first introduce finite sequences \((\rho_i)_{\ell\leq i\leq m}\) and \((\tau_i)_{\ell\leq i\leq m}\) such that
\[
0 < \rho_m < \tau_{m-1} < \rho_{m-1} < \dotsc < \rho_{\ell+1} < \tau_{\ell} <  \rho_\ell = \rho.
\]
For \(i = m\), we take \(\Phi_m=\Id\). 
Using downward induction, we shall define for every \(i \in \{\ell, \dotsc, m-1\}\) smooth maps \(\Phi_{i} : \R^m \setminus T^{i^*} \to \R^m\) such that 
\begin{enumerate}[(a)]
\item \label{item1240} \(\Phi_i\) is injective,
\item \label{item1241} for every \(\sigma^m \in \cS^m\) and for every \(r \in \{i^*, \dotsc, m-1\}\), \(\Phi_i(\sigma^m \setminus T^r) \subset \sigma^m \setminus T^r\),
\item \label{item1242} \(\Supp{\Phi_i} \subset U^m + Q^{m}_{\rho_i\eta}\),
\item \label{item1243} \(\Phi_i(U^m \setminus T^{i^*}) \subset U^i + Q^{m}_{\rho_i\eta}\),
\item \label{item1244} for every \(x\in \R^m \setminus T^{i^*}\) and for every \(r \in \{i^*, \dots, m-2\}\),
\[
\dist(\Phi_i(x), T^{r}) \dist(x, T^{r+1}) = \dist(\Phi_i(x), T^{r+1})\dist(x,T^{r}),
\]
\item \label{item1245} for every \(j \in \N_*\) and for every 
\(x \in \R^m \setminus T^{i^*}\),
\[
\abs{D^j \Phi_i(x)} \le  \frac{C\eta}{\bigl(\dist(x, T^{i^*})\bigr)^j},
\]
for some constant \(C > 0\) depending on \(j\), \(m\) and \(\rho\),
\item \label{item1246} for every \(0 < \beta < i+1\), for every \(j \in \N_*\) and for every 
\(x \in \R^m \setminus T^{i^*}\),
\[
\eta^{j-1} \abs{D^j \Phi_i(x)} \le  C' \bigl(\jac{\Phi_i}(x)\bigr)^\frac{j}{\beta},
\]
for some constant \(C' > 0\) depending on \(\beta\), \(j\), \(m\) and \(\rho\).
\end{enumerate}
The map \(\Phi_\ell\) will satisfy the conclusion of the proposition.

\medskip

Let \( i \in \{\ell+1, \dotsc, m\}\) and let \(\Theta_{i}\) be the map obtained from Proposition~\ref{lemmaThickeningFaceFromPrimalSkeleton} with parameters \(\underline\rho = \rho_{i}\), \(\rho = \tau_{i-1}\), \(\overline\rho = \rho_{i-1}\) and \(\ell = i\).  Given \(\sigma^i \in \mathcal{U}^{i}\), we may identify \(\sigma^i\) with \(Q^{i}_{\eta} \times \{0^{m-i}\}\) and \(T^{(i-1)^*} \cap (\sigma^i + Q_{\tau_{i-1}\eta}^m)\) with \(\{0^{i}\} \times Q_{\tau_{i-1}\eta}^{m-i}\). The map \(\Theta_i\) induces by isometry a map which we shall denote by \(\Theta_{\sigma^i}\).

Let \(\Psi_i : \R^m \setminus T^{(i-1)^*} \to \R^m\) be defined for every \(x \in \R^m \setminus T^{(i-1)^*}\) by 
\[
\Psi_i(x):=\begin{cases} 
           \Theta_{\sigma^i}(x) & \text{if } x \in \sigma^i + Q^{m}_{\tau_{i-1}\eta} \text{ for some } \sigma^i \in \mathcal{U}^i,\\
            x&\text{otherwise} .
         \end{cases}
\] 
We first explain why \(\Psi_i\) is well-defined. Since \(\Theta_{\sigma^i}\) coincides with the identity map on \(\partial\sigma^i + Q^m_{\tau_{i-1}\eta}\), then for every \(\sigma^i_1, \sigma^i_2 \in \mathcal{U}^{i}\), if \(x \in (\sigma_1^i + Q^{m}_{\tau_{i-1}\eta}) \cap (\sigma_2^i + Q^{m}_{\tau_{i-1}\eta})\) and \(\sigma_1^i \ne \sigma_2^i\), then
\[
\Theta_{\sigma_1^i}(x) = x = \Theta_{\sigma_2^i}(x).
\]
One also verifies directly that \(\Psi_i\) is smooth on \(\R^m \setminus T^{(i-1)^*}\).

Assuming that \(\Phi_i\) has been defined satisfying properties \eqref{item1240}--\eqref{item1246}, we let
\[
\Phi_{i-1}=\Psi_i \circ \Phi_i.
\]
The map \(\Phi_{i-1}\) is well-defined on \(\R^m\setminus T^{(i-1)^*}\) since \(\Phi_i(\R^m \setminus T^{(i-1)^*}) \subset \R^m \setminus T^{(i-1)^*}\).

We now check that \(\Phi_{i-1}\) satisfies all required properties.

\begin{proof}[Proof of Property \eqref{item1240}]
The map \(\Phi_{i-1}\) is injective since \(\Psi_i\) and \(\Phi_i\) are injective.
\end{proof}

\begin{proof}[Proof of Property \eqref{item1241}]
For every \(r\in \{(i - 1)^*, \dotsc, m - 1\}\) and for every \(\sigma^m \in \cS^m\),
we have by induction hypothesis \(\Phi_{i}(\sigma^m\setminus T^r)\subset \sigma^m \setminus T^r\). Moreover, for any  \(\sigma^m \in \cS^m\) and any \(\tilde\sigma^i\in \cU^{i}\), the formula of \(\Theta_i\) implies that \( \Theta_{\tilde\sigma^i}(\sigma^m \setminus T^r) \subset \sigma^m\setminus T^r.\)
\end{proof}

\begin{proof}[Proof of Property \eqref{item1242}]
By induction hypothesis \(\Phi_i\) coincides with the identity map outside \(U^{m}+Q^{m}_{\rho_{i}\eta}\).
By construction, \(\Psi_i\) coincides with the identity map outside \(U^{m}+Q^{m}_{\tau_{i-1}\eta}\)  (see Proposition~\ref{lemmaThickeningFaceFromPrimalSkeleton}, property \((\ref{itemthickeninglemma2})\)). Since \(\rho_i < \tau_{i-1} < \rho_{i-1}\), we deduce that \(\Supp{\Phi_{i-1}} \subset U^m + Q^{m}_{\rho_{i-1}\eta}\). 
\end{proof}

\begin{proof}[Proof of Property \eqref{item1243}]
By induction hypothesis (property \eqref{item1243})
\[
\Phi_{i}(U^m \setminus T^{i^*}) \subset U^i + Q^m_{\rho_i\eta}
\]
and (property \eqref{item1241})
\[
\Phi_i(\R^m \setminus T^{(i-1)^*}) \subset \R^m \setminus T^{(i-1)^*}.
\]
Since \(T^{(i-1)^\ast}\supset T^{i^\ast}\), we have 
\[
\Phi_{i}(U^m \setminus T^{(i-1)^*}) \subset (U^i + Q^m_{\rho_i\eta})\setminus T^{(i-1)^\ast}.
\]
By construction of \(\Theta_i\) (see Proposition~\ref{lemmaThickeningFaceFromPrimalSkeleton}, property \((\ref{itemthickeninglemma1})\)), for every \(\sigma^i \in \mathcal{U}^i\), 
\[
\Theta_{\sigma^i}\big((\sigma^i + Q^m_{\rho_i\eta}) \setminus T^{(i-1)^*}\big) \subset \partial \sigma^i + Q^m_{\rho_{i-1}\eta}.
\]
Taking the union over all faces \(\sigma^i \in \mathcal{U}^i\),  we get
\[
\Psi_i\big((U^i + Q^m_{\rho_i\eta})\setminus T^{(i-1)^\ast}\big)\subset U^{i-1} + Q^m_{\rho_{i-1}\eta}.
\]
Combining the information for \(\Phi_i\) and \(\Psi_i\), we obtain
\[
\Phi_{i-1}(U^m \setminus T^{(i-1)^\ast}) \subset U^{i-1} + Q^m_{\rho_{i-1}\eta}.
\qedhere
\]
\end{proof}

\begin{proof}[Proof of Property \eqref{item1244}]
Let \(r \in \{(i-1)^*, \dots, m-2\}\) and \(x\in \R^m \setminus T^{(i-1)^*}\).
If \(\Phi_{i-1}(x) = \Phi_{i}(x)\), then the conclusion follows by induction. 
If \(\Phi_{i-1}(x) \ne \Phi_{i}(x)\), then there exists 
\(\sigma^i \in \cU^i\) such that 
\(\Phi_{i}(x) \in \sigma^i + Q^{m}_{\tau_{i-1}\eta}\) and \(\Phi_{i-1}(x) =\Theta_{\sigma^i} (\Phi_{i}(x))\). Since \(\Phi_i(x) \in \Supp{\Psi_i}\), 
\[
\Phi_i(x) \in (\sigma^i + Q^{m}_{\tau_{i-1}\eta}) \setminus (\partial\sigma^i + Q^{m}_{\tau_{i-1}\eta}).
\]

Up to an isometry, we may assume that \(\sigma^i = Q_\eta^{i} \times \{0^{m - i}\}\). For every \(0 < \lambda < 1\)  and for every \(y=(y',y'')\in Q^{i}_{(1-\lambda)\eta}\times Q^{m-i}_{\lambda\eta}\),
\[
\dist(y, T^r)=\dist\big((y',0), T^r\cap (Q^{i}_{(1-\lambda)\eta}\times \{0^{m - i}\})\big).
\]
In view of the formula of \(\Theta_{i}\), we deduce that for every \(y\in (\sigma^i + Q^{m}_{\tau_{i-1}\eta}) \setminus (\partial\sigma^i \times Q^{m}_{\tau_{i-1}\eta})\), 
\[
\dist(\Theta_{\sigma^i}(y), T^r)\dist(y, T^{r+1})=\dist\big(\Theta_{\sigma^i}(y), T^{r+1})\dist(y, T^{r}\big);
\]
this identity is reminiscent of Thales' intercept theorem from Euclidean geometry.
By induction hypothesis, we then get
\[
\begin{split}
\dist(\Phi_{i-1}(x), T^{r})\dist(x,T^{r+1}) 
& = \dist(\Theta_{\sigma^i}(\Phi_{i}(x)), T^{r})\dist(x,T^{r+1})\\
& = \dist(\Theta_{\sigma^i}(\Phi_{i}(x)), T^{r+1})\dist(x,T^{r})\\
& = \dist(\Phi_{i-1}(x)), T^{r+1})\dist(x,T^{r}).
\end{split}
\]
This gives the conclusion.
\end{proof}

\begin{proof}[Proof of Property \eqref{item1245}]
Let \(x \in \R^m \setminus T^{(i-1)^*}\).
If \(\Psi_i\) coincides with the identity map in a neighborhood of \(\Phi_i(x)\), then \(D^j \Phi_{i-1}(x)=D^j \Phi_{i}(x)\) and the conclusion follows from the induction hypothesis and the fact that \(T^{(i-1)^*}\supset T^{i^*}\). 

If \(\Psi_i\) does not coincide with the identity map in a neighborhood of \(\Phi_i(x)\), then there exists \(\sigma^i \in \mathcal{U}^{i}\) such that 
\[
\Phi_{i}(x)\in (\sigma^i + Q^{m}_{\tau_{i-1}\eta}) \setminus (\partial\sigma^i + Q^{m}_{\tau_{i-1}\eta})
\] 
and \(\Phi_{i-1}(x)=\Theta_{\sigma_i}(\Phi_{i}(x))\). By the chain rule for higher order derivatives,
\[
\abs{D^j\Phi_{i-1}(x)} \le
C_1 \sum_{r=1}^j \sum_{\substack{1 \le t_1 \le \dotsc \le t_r\\ t_1 + \dotsb + t_r = j}}\abs{D^r \Theta_{\sigma_i}(\Phi_i(x))}\, \abs{D^{t_1} \Phi_i(x)} \dotsm \abs{D^{t_r} \Phi_i(x)}.
\]
By construction of \(\Theta_i\) (see Proposition~\ref{lemmaThickeningFaceFromPrimalSkeleton}, property \((\ref{itemthickeninglemma3})\)),
we have for any \(y = (y',y'')\in (Q_{(1 - \tau_{i-1})\eta}^i \times Q^{m-i}_{\tau_{i-1}\eta}) \setminus (\{0^{i}\}\times Q^{m-i}_{\tau_{i-1}\eta})\),
\[
\abs{D^{r}\Theta_{i}(y)}\le \frac{C_2 \eta}{\abs{y'}^r}.
\]
This implies 
\[
\abs{D^r\Theta_{\sigma^i}(\Phi_{i}(x))} \le \frac{C_2 \eta}{(\dist \big(\Phi_{i}(x), T^{(i-1)^{\ast}} )\big)^r}.
\]
By the induction hypothesis, for every \(1 \le t_1 \le \ldots \le t_r\) such that \(t_1 + \dots + t_r = j\), 
\begin{multline*}
\abs{D^{t_1} \Phi_i(x)} \dotsm \abs{D^{t_r}\Phi_i(x)}\\
 \le C_3 \frac{\eta}{\big(\dist (x, T^{i^\ast})\big)^{t_1}} \dotsm \frac{\eta}{\big(\dist (x, T^{i^\ast})\big)^{t_r}} = C_3 \frac{\eta^r}{\big(\dist (x, T^{i^\ast})\big)^{j}}.
\end{multline*}
Thus,
\[
\abs{D^j \Phi_{i-1}(x)}\le C_4 \sum_{r=1}^{j} \frac{\eta^{r + 1}}{\big(\dist (\Phi_{i}(x), T^{(i-1)^{\ast}} )\big)^r \big(\dist (x, T^{i^\ast})\big)^j}.
\]
We recall that by property \eqref{item1245},
\[
\dist(\Phi_{i}(x), T^{(i-1)^{*}}) \dist(x, T^{i^*}) = \dist (x, T^{(i-1)^*}) \dist (\Phi_{i}(x), T^{i^*}).
\]
Since \(\Phi_{i}(x)\in (\sigma^i + Q^{m}_{\tau_{i-1}\eta}) \setminus (\partial\sigma^i + Q^{m}_{\tau_{i-1}\eta})\),
\[
\dist (\Phi_{i}(x), T^{i^*}) \ge (1-\tau_{i-1})\eta \ge (1-\rho)\eta.
\]
Thus,
\begin{multline*}
\big(\dist (\Phi_{i}(x), T^{(i-1)^{\ast}} )\big)^r \big(\dist (x, T^{i^\ast})\big)^j \\
\begin{aligned}
& = 
\big(\dist (x, T^{(i-1)^*}) \dist (\Phi_{i}(x), T^{i^*})\big)^r  \big(\dist (x, T^{i^\ast})\big)^{j - r}\\
& \ge 
\big(\dist{(x, T^{(i-1)^*})}\big)^r \big((1-\rho)\eta)^r \big(\dist (x, T^{i^\ast})\big)^{j - r}.
\end{aligned}
\end{multline*}
Since \(T^{i^*} \subset T^{(i-1)^*}\), we conclude that
\[
\abs{D^j \Phi_{i-1}(x)}\le C_5 \frac{\eta}{\bigl(\dist (x,T^{(i-1)^{\ast}})\bigr)^j}.
\qedhere
\]
\end{proof}

\begin{proof}[Proof of Property \eqref{item1246}]
Let \(j \in \N_*\) and let \(x \in \R^m \setminus T^{(i-1)^*}\).
If \(\Psi_i\) coincides with the identity map in a in a neighborhood of \(\Phi_i(x)\), then \(D^j \Phi_{i-1}(x)=D^j \Phi_{i}(x)\) and \(\jac{\Phi_{i-1}}(x) = \jac{\Phi_{i}}(x)\). The conclusion then follows from the induction hypothesis. 

Assume that \(\Psi_i\) does not coincides with the identity map in a neighborhood of \(\Phi_i(x)\). Let \(0 < \beta < i\) and \(r \in \{0, \dots, j\}\). By induction hypothesis,  if \(1 \le t_1 \le \ldots \le t_r\) and \(t_1 + \dotsb + t_r = j \), then
\[
\abs{D^{t_1} \Phi_i(x)} \dotsm \abs{D^{t_r}\Phi_i(x)} \le C_1 \frac{(\jac{\Phi_i}(x))^\frac{t_1}{\beta}}{\eta^{t_1 - 1}} \dotsm \frac{(\jac{\Phi_i}(x))^\frac{t_r}{\beta}}{{\eta^{t_r - 1}}} = C_1 \frac{(\jac{\Phi_i}(x))^\frac{j}{\beta}}{\eta^{j-r}}.
\]
Let \(\sigma^i \in \mathcal{U}^{i}\) be such that 
\[
\Phi_{i}(x)\in (\sigma^i + Q^{m}_{\tau_{i-1}\eta}) \setminus (\partial\sigma^i + Q^{m}_{\tau_{i-1}\eta})
\] 
and \(\Phi_{i-1}(x)=\Theta_{\sigma_i}\circ \Phi_{i}(x)\). 
By construction of \(\Theta_i\) (see Proposition~\ref{lemmaThickeningFaceFromPrimalSkeleton}, property \((\ref{itemthickeninglemma4})\)),  we have for any \(y\in (Q_{(1 - \tau_{i-1})\eta}^i \times Q^{m-i}_{\tau_{i-1}\eta}) \setminus (\{0^{i}\}\times Q^{m-i}_{\tau_{i-1}\eta})\),
\[
\eta^{r-1}\abs{D^{r}\Theta_{i}(y)}\le C_2(\jac{\Theta_i}(y))^{\frac{r}{\beta \frac{r}{j}}} = C_2 (\jac{\Theta_i}(y))^\frac{j}{\beta}.
\]
Thus, 
\[{}
\begin{split}
\abs{D^{r}\Theta_{\sigma^i}(\Phi_i(x))}\, \abs{D^{t_1} \Phi_i(x)} \dotsm \abs{D^{t_r}\Phi_i(x)}
& \le C_3 \frac{(\jac{\Theta_{\sigma^i}}(\Phi_i(x)))^{\frac{j}{\beta}}}{\eta^{r-1}} \frac{(\jac{\Phi_i}(x))^{\frac{j}{\beta}}}{\eta^{j-r}}\\
& = \frac{C_3}{\eta^{j-1}}\big(\jac{\Phi_{i-1}}(x) \big)^{\frac{j}{\beta}}.
\end{split}
\]
Therefore, by the chain rule for higher order derivatives,
\[
\abs{D^j\Phi_{i-1}(x)} 
\le \frac{C_4}{\eta^{j-1}} \big(\jac{\Phi_{i-1}}(x) \big)^{\frac{j}{\beta}}.
\]
This gives the conclusion.
\end{proof}

By downward induction, we conclude that properties \eqref{item1240}--\eqref{item1246} hold for every \(i \in \{\ell, \dots, m\}\). In particular, \(\Phi_\ell\) satisfies properties $(i)$--\((v)\) of Proposition~\ref{propositionthickeningfromaskeleton}.
\end{proof}

We establish a couple of lemmas in order to prove Proposition~\ref{lemmaThickeningFaceFromPrimalSkeleton}:

\begin{lemma}
\label{lemmaThickeningAroundPrimalSqueleton}
Let \(\ell\in \{1,\ldots, m\}\), let \(\eta > 0\), let \(0 < \underline{\rho} < \rho < \overline{\rho} < 1\) and \(0 < \kappa  < 1 - \overline{\rho}\).
There exists a smooth function \(\lambda : \R^m \to [1,\infty)\) such that if \(\Phi : \R^m   \to \R^m\) is defined for \(x = (x', x'') \in \R^\ell \times \R^{m - \ell}  \) by
\[
\Phi(x) = (\lambda(x)x', x''),
\]
then
\begin{enumerate}[$(i)$]
\item \(\Phi\) is a diffeomorphism,
\item \(\Supp{\Phi} \subset Q^\ell_{(1-\rho)\eta} \times Q^{m-\ell}_{\rho \eta}\),
\item 
\(
 \Phi\bigl( (Q^\ell_\eta \setminus Q^\ell_{\kappa \eta}) \times Q^{m-\ell}_{\underline{\rho} \eta }\bigr) \subset (Q^\ell_\eta \setminus Q^\ell_{(1-\overline{\rho})\eta}) \times Q^{m-\ell}_{\underline{\rho} \eta },
\)
\item 
\label{item1454}
for every \(j \in \N_*\) and for every \(x \in \R^m\),
\[
\eta^{j-1}\abs{D^j \Phi(x)} \le C,
\]
for some constant \(C > 0\) depending on \(j\), \(m\), \(\underline{\rho}\), \(\rho\), \(\overline{\rho}\) and \(\kappa\),
\item 
\label{item1455}
for every \(j \in \N_*\) and for every \(x \in \R^m\),
\[
C' \le \jac{\Phi}(x) \le C'',
\]
for some constants \(C', C'' > 0\) depending on \(m\), \(\underline{\rho}\), \(\rho\), \(\overline{\rho}\) and \(\kappa\).
\end{enumerate}
\end{lemma}

\begin{proof} By scaling, we may assume that \(\eta=1\).
Let \(\psi : \R \to [0, 1]\) be a smooth function such that 
\begin{itemize}[\(-\)]
\item \(\psi\) is nonincreasing on \(\R_+\) and nondecreasing on \(\R_-\),
\item for \(\abs{t} \le 1 - \overline{\rho}\), \(\psi(t)=1\), 
\item for \(\abs{t} \ge 1-\rho\), \(\psi(t)=0\).
\end{itemize}
Let \(\theta : \R \to [0, 1]\) be a smooth function such that
\begin{itemize}[\(-\)]
\item for \(\abs{t}\le \underline{\rho}\), \(\theta(t)=1\),
\item for \(\abs{t} \ge \rho\), \(\theta(t)=0\).
\end{itemize}
Let \(\varphi : \R^m \to \R\) be the function defined for \(x=(x_1, \dotsc, x_m) \in \R^m\) by
\[
\textstyle \varphi(x) =  \prod\limits_{i=1}^\ell \psi(x_i) \prod\limits_{i=\ell+1}^m \theta (x_i).
\]
Thus, 
\begin{itemize}[\(-\)]
\item for every \(x \in \R^m \setminus (Q^\ell_{1-\rho} \times Q^{m - \ell}_{\rho})\), \(\varphi(x) = 0\),
\item for every \(x \in Q^\ell_{1-\overline\rho} \times Q^{m - \ell}_{\underline\rho}\),  \(\varphi(x) = 1\).
\end{itemize}

We shall define the map \(\Phi\) in terms of its inverse \(\Psi\): let \(\Psi : \R^m \to \R^m\) be the function defined for \(x = (x', x'') \in \R^\ell \times \R^{m - \ell}\) by
\[
\Psi(x)= \big((1 - \alpha \varphi(x))x', x''\big),
\]
where \(\alpha \in \R\). In particular,
\begin{itemize}[\(-\)]
\item for every \(x \in \R^m \setminus (Q^\ell_{1-\rho} \times Q^{m - \ell}_{\rho})\), \(\Psi(x) = x\),
\item for every \(x = (x', x'') \in Q^\ell_{1-\overline\rho} \times Q^{m - \ell}_{\underline\rho}\), \(\Psi(x) = ((1-\alpha) x', x'')\).
\end{itemize}
In view of this second property, taking \(\alpha=1-\frac{\kappa}{1-\overline{\rho}}\), we deduce that \(\Psi\) is a bijection between \(Q^\ell_{1-\overline\rho} \times Q^{m-\ell}_{\underline\rho}\) and \(Q^\ell_{\kappa} \times Q^{m-\ell}_{\underline\rho}\).

We now prove that \(\Psi\) is injective. If \(x, y \in  \R^\ell \times \R^{m - \ell}\) satisfy \(\Psi(x) = \Psi(y)\), then \(y'' = x''\) and \(y' = t x'\)  for some \(t > 0\). Since \(\alpha \in (0, 1)\), the function 
\[
g : s \in [0, \infty) \longmapsto s(1-\alpha \varphi(sx', x'')) 
\]
is the product of an increasing function with a nondecreasing positive function. Thus, \(g\) is increasing, whence \(\Psi\) is injective. Since \(g(0)= 0\) and \(\lim\limits_{t \to +\infty}{g(t)} = +\infty\), by the Intermediate value theorem, \(g([0, \infty)) = [0, \infty)\). Thus, \(\Psi\) is surjective. Therefore, the map \(\Psi\) is a bijection. 

We claim that for every \(x \in \R^m\), \(D\Psi(x)\) is invertible. Indeed, for every \(x = (x', x'') \in \R^\ell \times \R^{m - \ell}\) and for every \(v = (v', v'') \in \R^\ell \times \R^{m - \ell}\),
\[
D\Psi(x)[v] = \bigl((1 - \alpha \varphi(x))v' - \alpha D\varphi(x)[v] x', v'' \bigr).
\]
The Jacobian of \(\Psi\) can be computed as the determinant of a nilpotent perturbation of a diagonal linear map to be
\[
\jac{\Psi(x)} = (1 - \alpha \varphi(x))^{\ell-1} \bigl(1 - \alpha \varphi(x) - \alpha D\varphi(x)[(x', 0)] \bigr).
\]
Since \(\psi\) is nonincreasing on \(\R_+\) and nondecreasing on \(\R_-\), \(D\varphi(x)[(x', 0)] \le 0\). Thus,
\[
\jac{\Psi(x)} \ge (1 - \alpha \varphi(x))^{\ell} \ge (1 - \alpha)^\ell > 0.
\]
The map \(\Phi = \Psi^{-1}\) satisfies all the desired properties. 
\end{proof}

\begin{lemma}
\label{lemmaThickeningAroundDualSqueleton}
Let \(\ell \in \{1,\dotsc, m\}\), \(\eta > 0\), \(0<\underline{\rho}<\rho<\overline{\rho}<1\) and \(T=\{0^\ell\}\times Q^{m-\ell}_{\rho\eta}.\) 
There exists a smooth function \(\lambda : \R^m \setminus T \to [1,\infty)\) such that if \(\Phi : \R^m \setminus T \to \R^m\) is defined for \(x = (x', x'') \in (\R^\ell \times \R^{m - \ell})\setminus T\) by
\[
\Phi(x) = (\lambda(x)x', x''),
\]
then
\begin{enumerate}[$(i)$]
\item \(\Phi\) is injective,
\label{item03011}
\item \(\Supp{\Phi}\subset B^\ell_{(1-\rho)\eta} \times Q^{m-\ell}_{\rho\eta}\),
\label{item03012}
\item \(\Phi\big((B^{\ell}_{(1-\rho)\eta}\times Q^{m-\ell}_{\underline\rho\eta})\setminus T\big)\subset (B^{\ell}_{(1-\rho)\eta}\setminus B^{\ell}_{(1-\overline\rho) \eta}) \times Q^{m-\ell}_{\underline{\rho}\eta} \),
\label{item03013}
\item for every \(j\in \N_{*}\) and for every 
\(x = (x', x'') \in (B^\ell_{(1-\rho)\eta} \times Q^{m-\ell}_{\rho\eta}) \setminus T\),
\begin{equation*}
\abs{D^j\Phi(x)}\leq \frac{C\eta}{\abs{x'}^j},
\end{equation*}
for some constant \(C > 0\) depending on \(j\), \(m\), \(\underline\rho\), \(\rho\) and \(\overline\rho\),
\label{item03014}
\item for every \(0 < \beta < \ell\), for every \(j \in \N_*\) and for every 
\( x \in \R^m \setminus T\), 
\[
\eta^{j-1}\abs{D^j \Phi(x)} \le  C' \bigl(\jac{\Phi}(x)\bigr)^\frac{j}{\beta},
\]
for some constant \(C' > 0\) depending on \(\beta\), \(j\), \(m\), \(\underline{\rho}\), \(\rho\) and \(\overline{\rho}\).
\label{item03015}
\end{enumerate}
\end{lemma}
\begin{proof}
By scaling, we may assume that \(\eta = 1\). Given \(b > 0\), let \(\varphi : (0, \infty) \to [1, \infty)\) be a smooth function such that
\begin{itemize}[\(-\)]
\item for \(0 < s \le 1-\overline\rho\), \(\varphi(s)= \dfrac{1-\overline\rho}{s}\Bigl(1+\frac{b}{\ln \frac{1}{s}}\Bigr)\),
\item for \(s \ge 1-\rho\), \(\varphi(s)=1\),
\item the function \(s \in (0, \infty) \mapsto s\varphi(s)\) is increasing.
\end{itemize}
This is possible for any \(b > 0\) such that 
\[
(1-\overline\rho)\Bigl(1+\frac{b}{\ln \frac{1}{1-\overline\rho}}\Bigr) < 1 - \rho.
\]

Let \(\theta : \R^{m-\ell} \to [0, 1]\) be a smooth function such that 
\begin{itemize}[\(-\)]
\item for \(y \in Q^{m-\ell}_{\underline{\rho}}\),  \(\theta(y) = 0\),
\item for \(y \in \R^{m-\ell} \setminus Q^{m-\ell}_{\rho}\), \(\theta(y)=1\). 
\end{itemize}
We now introduce  for \(x=(x', x'') \in \R^\ell\times \R^{m-\ell}\),  
\[
\zeta(x) = \sqrt{\abs{x'}^2+\theta\bigl(x''\bigr)^2}.
\]

Let \(\lambda : \R^m \setminus T \to \R\) be the function defined for \(x=(x', x'') \in \R^m \setminus T\) by  
\[
\lambda(x)= \varphi(\zeta(x)).
\]
Since \(\zeta \ne 0\) in \(\R^m\setminus T\), the function \(\lambda\) is well-defined and smooth. In addition, \(\lambda \ge 1\).

We now check that the map \(\Phi\) defined in the statement satisfies all the required properties.

\begin{proof}[Proof of Property~\((\ref{item03011})\)]
In order to check that \(\Phi\) is injective, we first observe that if \(x=(x', x''), y=(y', y'') \in B^\ell_1 \times Q^{m-\ell}_{\rho}\) and \(\Phi(x)=\Phi(y)\), then \(x''=y''\), and there exists \(t > 0\) such that \(y' = t x'\). The conclusion follows from the fact that the function
\[
h: s \in [0, \infty) \longmapsto s \varphi\bigl(\sqrt{s^2+\theta(x'')^2}\bigr)
\]
is increasing. 
\end{proof}

\begin{proof}[Proof of Property~\((\ref{item03012})\)]
For every \(x = (x', x'') \in (\R^\ell\times \R^{m-\ell})\setminus T\), if \(x' \not\in B^\ell_{1-\rho}\) or if \(x'' \not\in Q_\rho^{m - \ell}\), then \(\zeta(x) \ge 1-\rho\). Thus, \(\lambda(x) = \varphi(\zeta(x)) = 1\) and \(\Phi(x) = x\). We then have \(\Supp{\Phi}\subset B^\ell_{1-\rho} \times Q^{m-\ell}_{\rho}\).
\end{proof}

\begin{proof}[Proof of Property~\((\ref{item03013})\)]
We first observe that since the function \(s \in (0, \infty) \mapsto s\varphi(s)\) is increasing and \(\lim\limits_{s \to 0}{s\varphi(s)} = 1 - \overline{\rho}\), for every \(s > 0\),
\[
s \varphi(s) \ge 1 - \overline{\rho}.
\]
Since for every \(x = (x', x'') \in (B^{\ell}_{1-\rho}\times Q^{m-\ell}_{\underline\rho})\setminus T\), we have \(\zeta(x) = \abs{x'}\), we deduce that 
\[
\abs{\lambda(x) x'} = \varphi(\abs{x'})\abs{x'} \ge 1-\overline\rho.
\]
On the other hand, since the function \(h\) defined above is increasing,
\[
\abs{\lambda(x) x'} = h(\abs{x'}) \le h(1-\rho) = 1-\rho.
\]
We conclude that \(\lambda(x)x' \in B^\ell_{1-\rho} \setminus B^\ell_{1-\overline\rho}\).
\end{proof}

\begin{proof}[Proof of Property~\((\ref{item03014})\)]
By the chain rule,
\[
\abs{D^{j}\lambda(x)}
\le C_1 \sum_{i=1}^{j} \sum_{\substack{1 \le t_1 \le \dotsc \le t_i\\ t_1 + \dotsb + t_i = j}} \abs{\varphi^{(i)}(\zeta(x))}\, \abs{D^{t_1} \zeta(x)} \dotsm \abs{D^{t_i} \zeta(x)}.
\]
For every \(i \in \N_{\ast}\) and for every \(s > 0\),
\begin{equation*}
\abs{\varphi^{(i)}(s)} \leq \frac{C_2}{s^{i+1}}
\end{equation*}
and for every \(x \in (B^{\ell}_1\times \R^{m-\ell})\setminus T\),
\begin{equation*}
\abs{D^i \zeta(x)} \leq \frac{C_3}{\zeta(x)^{i-1}}.
\end{equation*}
Thus, for every \(1 \le t_1 \le \ldots \le t_i\) such that \(t_1 + \dotsb + t_i = j\),
\[
\abs{D^{t_1} \zeta(x)} \dotsm \abs{D^{t_i} \zeta(x)} \le  \frac{C_4}{\zeta(x)^{t_1-1} \dotsm \zeta(x)^{t_i-1}} = \frac{C_4}{\zeta(x)^{j - i}}.
\]
By the chain rule,
\[
\begin{split}
\abs{D^{j} \lambda(x)}
& \le C_5
\sum_{i=1}^{j} \frac{1}{\zeta(x)^{i+1}} \frac{1}{\zeta(x)^{j - i}} =  \frac{C_5 j}{\zeta(x)^{j+1}}.
\end{split}
\]
Hence, by the Leibniz rule, for any \(x\in (B^{\ell}_{1}\times \R^{m-\ell})\setminus T\), 
\begin{equation}
\label{equationEstimationDeriveeThickening}
\abs{D^j\Phi(x)}\le\frac{C_6}{\zeta(x)^j}.
\end{equation}
Since \(\zeta(x) \ge \abs{x'}\), the conclusion follows.
\end{proof}

\begin{proof}[Proof of Property~\((\ref{item03015})\)]
For every \( x=(x', x'') \in (\R^\ell\times \R^{m-\ell})\setminus T\) and \(v=(v', v'') \in \R^\ell \times \R^{m-\ell}\),
\[
D\Phi(x)[v]=\Bigl(\varphi\bigl(\zeta(x)\bigr)v'
+\varphi^{(1)}\bigl(\zeta(x)\bigr)\\
\frac{x' \cdot v'+\theta(x'') D\theta(x'')[v'']}{\zeta(x)} x', v''\Bigr).
\]
The Jacobian can be computed as the determinant of a nilpotent perturbation of a diagonal linear map to be
\[ 
\begin{split}
\jac \Phi(x)
& =\varphi(\zeta(x))^{\ell-1}\Bigl(\varphi(\zeta(x))+\varphi^{(1)}(\zeta(x))\frac{\abs{x'}^2}{\zeta(x)} \Bigr)\\
& =\varphi(\zeta(x))^{\ell-1}\Bigl(\varphi(\zeta(x)) \Bigl(1 - \frac{\abs{x'}^2}{\zeta(x)^2}\Bigr) + \big(\varphi^{(1)}(\zeta(x))\zeta(x) + \varphi(\zeta(x)) \big)\frac{\abs{x'}^2}{\zeta(x)^2} \Bigr).
\end{split}
\]
Since for every \(s > 0\), 
\[
s \varphi^{(1)}(s) + \varphi(s) = (s\varphi(s))^{(1)} \ge 0
\]
and since there exists \(c_1 > 0\) such that for every \(s > 0\), 
\[
\varphi(s) \ge \frac{c_1}{s},
\]
we have
\[
\jac \Phi(x) \ge \varphi(\zeta(x))^{\ell} \Bigl(1 - \frac{\abs{x'}^2}{\zeta(x)^2}\Bigr) \ge  \frac{c_2} {\zeta(x)^\ell} \Bigl(1 - \frac{\abs{x'}^2}{\zeta(x)^2}\Bigr).
\]
If \(\abs{x'} \le \theta(x'')\), then \(\zeta(x) \ge \sqrt{2}\abs{x'}\) and we get
\[
\jac \Phi(x) \ge  \frac{c_3}{\zeta(x)^\ell}.
\]
On the other hand, by direct inspection, for every \(\alpha < 1\), there exists a constant \(c_4 > 0\) depending on \(\alpha\) such that for every \(s > 0\), 
\[
s \varphi^{(1)}(s) + \varphi(s) \ge \frac{c_4}{s^\alpha}.
\]
Thus,
\[
\jac \Phi(x)
\ge \varphi(\zeta(x))^{\ell-1} \big(\varphi^{(1)}(\zeta(x))\zeta(x) + \varphi(\zeta(x)) \big)\frac{\abs{x'}^2}{\zeta(x)^2} \ge \frac{c_5}{\zeta(x)^{\ell -1 + \alpha}} \frac{\abs{x'}^2}{\zeta(x)^2}.
\]
If \(\abs{x'} > \theta(x'')\), then \(\zeta(x) \le \sqrt{2}\abs{x'}\) and we get
\[
\jac \Phi(x)
\ge  \frac{c_6}{\zeta(x)^{\ell -1 + \alpha}}.
\]
In both cases, we deduce that for every \(\beta < \ell\) and for every \( x \in \R^m \setminus T\),
\[
\jac \Phi(x)
\ge \frac{c_7}{\zeta(x)^{\beta}}.
\]
Thus, by estimate \eqref{equationEstimationDeriveeThickening} in the proof of property~\((\ref{item03014})\) above, when  \(x\in (B^{\ell}_{1-\rho}\times Q^{m-\ell}_\rho)\setminus T\),
\[
\abs{D^j\Phi(x)}\le\frac{C_5}{\zeta(x)^j} \le \frac{C_5}{(c_7)^\frac{j}{\beta}} (\jac \Phi(x))^\frac{j}{\beta}.
\qedhere
\]
\end{proof}

The proof of Lemma~\ref{lemmaThickeningAroundDualSqueleton} is complete.
\end{proof}

\begin{proof}[Proof of Proposition~\ref{lemmaThickeningFaceFromPrimalSkeleton}]
Define \(\Phi\) to be the composition of the map  \(\Phi_1\)  given by Lemma~\ref{lemmaThickeningAroundPrimalSqueleton} with any parameter \(\kappa\le \frac{1-\overline{\rho}}{\sqrt{\ell}}\)  together with the map  \(\Phi_2\)  given by Lemma~\ref{lemmaThickeningAroundDualSqueleton}; more precisely, \(\Phi=\Phi_1\circ \Phi_2\).  By composition, the map \(\Phi\) is injective and \(\Supp{\Phi}\subset Q^{\ell}_{(1-\rho)\eta }\times Q^{m-\ell}_{\rho\eta}\). Moreover, the choice of \(\kappa\) implies that \(Q^{\ell}_{\kappa\eta}\subset B^{\ell}_{(1-\overline\rho)\eta}\). Hence,
\[
 \Phi\big((Q^{\ell}_{(1-\rho)\eta }\times Q^{m-\ell}_{\underline{\rho}\eta}) \setminus T\big) \subset (Q^{\ell}_{(1-\rho)\eta } \setminus Q^{\ell}_{(1-\overline{\rho})\eta})\times Q^{m-l}_{\underline{\rho}\eta}.
\]
By the chain rule for higher order derivatives and by the estimate of the derivatives of \(\Phi_1\) (Lemma~\ref{lemmaThickeningAroundPrimalSqueleton}, see property (\(\ref{item1454}\))),
\[
\begin{split}
\abs{D^{j}\Phi(x)}
& \le C_1 \sum_{i=1}^{j} \sum_{\substack{1 \le t_1 \le \dotsc \le t_i\\ t_1 + \dotsb + t_i = j}}  \abs{D^i\Phi_1(\Phi_2(x))}\, \abs{D^{t_1} \Phi_2(x)} \dotsm \abs{D^{t_i} \Phi_2(x)}\\
& \le C_2 \sum_{i=1}^{j} \sum_{\substack{1 \le t_1 \le \dotsc \le t_i\\ t_1 + \dotsb + t_i = j}} \frac{\abs{D^{t_1} \Phi_2(x)} \dotsm \abs{D^{t_i} \Phi_2(x)}}{\eta^{i-1}}.
\end{split}
\]
The estimate for \(D^j \Phi\)  is a consequence of the estimates of the derivatives of \(\Phi_2\) (see  Lemma~\ref{lemmaThickeningAroundDualSqueleton}, property (\(\ref{item03014}\))). The estimate for \(\jac{\Phi}\) is a consequence of the estimate for \(\jac{\Phi_2}\) given by property (\(\ref{item03015}\)) of Lemma~\ref{lemmaThickeningAroundDualSqueleton} and the lower bound for \(\jac{\Phi_1}\) given by property (\(\ref{item1455}\)) of Lemma~\ref{lemmaThickeningAroundPrimalSqueleton}.
\end{proof}

%%%%%%%%%%%%%%%%%%%%%%%%%%%%%%%%%%%%%%%%%%%%%%%%%%%%%
%%%%%%%%%%%%%%%%%%%%%%%%%%%%%%%%%%%%%%%%%%%%%%%%%%%%%
%%%%%%%%%%%%%%%%%%%%%%%%%%%%%%%%%%%%%%%%%%%%%%%%%%%%%

\section{Density of the class \boldmath$R_{m-\floor{kp}-1}(Q^m; N^n)$}
\label{sectionDensityR}

In this section, we prove that the class \(R_{m-\floor{kp}-1} (Q^m; N^n)\) is dense in \(W^{k, p}(Q^m; N^n)\) regardless of the topology of the manifold \(N^n\).

\begin{theorem}
\label{theoremDensityManifoldNontrivialwithouttopologicalcondition}
If $kp < m$, then $R_{m-\floor{kp}-1}(Q^m; N^n)$ is strongly dense in $W^{k, p}(Q^m; N^n)$.
\end{theorem}
 This  result implies the \textit{if} part of Theorem~\ref{theoremDensityManifoldNontrivial}.

\begin{proof}[Proof of Theorem~\ref{theoremDensityManifoldNontrivialwithouttopologicalcondition}]
First observe that if \(u \in W^{k, p}(Q^m; N^n)\), then the restrictions to \(Q^m\) of the maps \(u_\gamma \in W^{k, p}(Q_{1 + 2\gamma}^m; N^n)\) defined for \(x \in Q_{1 + 2 \gamma}^m\) by \(u_\gamma(x) = u (x/(1 + 2 \gamma))\)  converge strongly to \(u\) in \(W^{k, p}(Q^m; N^n)\) when \(\gamma\) tends to \(0\). 
We can thus assume from the beginning that \(u \in W^{k, p}(Q_{1 + 2\gamma}^m; N^n)\). We apply successively the opening, smoothing and thickening constructions to this map \(u\). 

We divide the proof in four parts:

\begin{Part}
Construction of a map \(u^\mathrm{th}_\eta \in W^{k, p}(Q^{m}_{1+\gamma}; \R^\nu) \cap C^\infty(Q^{m}_{1+\gamma} \setminus T^{\ell^*}_\eta; \R^\nu)\) such that for every \(j \in \{1, \dots, k\}\),
\begin{multline*}
\eta^{j} \norm{D^j u^\mathrm{th}_\eta - D^j u}_{L^p(Q^{m}_{1+\gamma})}\\
\leq \sup_{v \in B_1^m}{\eta^{j} \norm{\tau_{\psi_\eta v}(D^j u) - D^j u}_{L^p(Q^{m}_{1+\gamma})}} + C \sum_{i=1}^j  \eta^{i} \norm{D^i u}_{L^p(U^m_\eta + Q^m_{2\rho\eta})},
\end{multline*}
where \(\cU^m_\eta\) is a subskeleton of \(Q^{m}_{1+\gamma}\) and \(\cT^{\ell^*}_\eta\) is the dual skeleton of \(\cU^\ell_\eta\).
\end{Part}

\textsl{Using the terminology presented in the Introduction, the subskeleton \(\cU^m_\eta\) will be chosen to be the set of all bad cubes together with the set of good cubes which intersect some bad cube. The precise choice of \(\cU^m_\eta\) will be made in Part~2.}

\medskip

Let \(\cK^m_\eta\) be a cubication of \(Q_{1+\gamma}^m\) of radius \(0 < \eta \le \gamma\) and let \(\cU^m_\eta\) be a subskeleton of \(\cK^m_\eta\).
Let \(0 < \rho < \frac{1}{2}\); thus, 
\[
2\rho\eta \le \gamma.
\]
Given \(\ell \in \{0, \dots, m - 1\}\), we begin by opening the map \(u\) in a neighborhood of \(U^\ell_\eta\). More precisely, let \(\Phi^\mathrm{op} : \R^m \to \R^m\) be the smooth map given by Proposition~\ref{openingpropGeneral} and consider the map
\[
u^\mathrm{op}_\eta = u \circ \Phi^\mathrm{op}.
\]
In particular, \(u^\mathrm{op}_\eta \in W^{k, p}(Q^m_{1+2\gamma}; N^n)\) and \(u^\mathrm{op}_\eta = u\) in the complement of \(U^\ell_\eta + Q^m_{2\rho\eta}\). For every \(j \in \{1, \dots, k\}\), 
\begin{equation}
\label{inequalityMainOpening}
\begin{split}
\eta^j\norm{D^j u^\mathrm{op}_\eta - D^j u}_{L^p(Q^m_{1+2\gamma})} 
& = \eta^j\norm{D^j u^\mathrm{op}_\eta - D^j u}_{L^p(U^\ell_\eta + Q^m_{2\rho\eta})}\\
& \le \eta^j\norm{D^j u^\mathrm{op}_\eta}_{L^p(U^\ell_\eta + Q^m_{2\rho\eta})} + \eta^j\norm{D^j u}_{L^p(U^\ell_\eta + Q^m_{2\rho\eta})}\\
& \le \NewConstant \sum_{i = 1}^j \eta^{i} \norm{D^i u}_{L^p(U^\ell_\eta + Q^m_{2\rho\eta})}.
\end{split}
\end{equation}

We next consider a smooth function \(\psi_\eta \in C^\infty(Q^m_{1+2\gamma})\) such that 
\[
\boxed{0<  \psi_\eta \le \rho \eta.}
\]
Given a mollifier \(\varphi \in C_c^\infty(B_1^m)\),
let for every \(x\in Q^{m}_{1+\gamma + \rho\eta}\),
\[
  u^\mathrm{sm}_\eta(x) = (\varphi_{\psi_\eta(x)} \ast u^\mathrm{op}_\eta)(x).
\]
Since \(0 < \psi_\eta \le \rho\eta\), the map \(u^\mathrm{sm}_\eta : Q_{1 + \gamma+\rho\eta}^m \to \R^\nu\) is well-defined and smooth.
If 
\[
\boxed{\norm{D\psi_\eta}_{L^\infty(Q^m_{1+2\gamma})} \le \beta}
\]
for some \(\beta < 1\) and if for every \(i \in \{2, \dotsc, k\}\),
\[
\boxed{ \eta^{i} \norm{D^i \psi_\eta}_{L^\infty(Q^m_{1+2\gamma})} \le \eta, }
\]
then by Proposition~\ref{lemmaConvolutionEstimates} with \(\omega = Q^m_{1+\gamma}\), we have for every \(j \in \{1, \dots, k\}\),
\begin{multline*}
\eta^{j} \norm{D^j u^\mathrm{sm}_\eta - D^j u^\mathrm{op}_\eta}_{L^p(Q^{m}_{1+\gamma})}\\
\leq \sup_{v \in B_1^m}{\eta^{j} \norm{\tau_{\psi_\eta v}(D^j u^\mathrm{op}_\eta) - D^j u^\mathrm{op}_\eta}_{L^p(Q^{m}_{1+\gamma})}} + \Constant \sum_{i=1}^j  \eta^{i} \norm{D^i u^\mathrm{op}_\eta}_{L^p(A)},
\end{multline*}
where
\(
A = \bigcup\limits_{x \in Q^{m}_{1+\gamma} \cap \supp{D\psi_\eta}}B_{\psi_\eta(x)}^m(x).
\)
For every \(v \in B_1^m\),
\begin{multline*}
\eta^{j} \norm{\tau_{\psi_\eta v}(D^j u^\mathrm{op}_\eta) - D^j u^\mathrm{op}_\eta}_{L^p(Q^m_{1+\gamma})}
\\
\begin{aligned}
& \le \eta^{j} \norm{\tau_{\psi_\eta v}(D^j u^\mathrm{op}_\eta) - \tau_{\psi_\eta v}(D^j u)}_{L^p(Q^m_{1+\gamma})}\\
& \qquad +\eta^{j} \norm{\tau_{\psi_\eta v}(D^j u) - D^j u}_{L^p(Q^m_{1+\gamma})}+\eta^{j} \norm{D^j u^\mathrm{op}_\eta - D^j u}_{L^p(Q^m_{1+\gamma})}
\end{aligned}
\end{multline*}
and, by the change of variable formula,
\[
\norm{\tau_{\psi_\eta v}(D^j u^\mathrm{op}_\eta) - \tau_{\psi_\eta v}(D^j u)}_{L^p(Q^m_{1+\gamma})}\le \Constant \norm{D^j u^\mathrm{op}_\eta - D^j u}_{L^p(Q^m_{1+2\gamma})}.
\]

If we further assume that 
\[
\boxed{\supp{D\psi_\eta}\subset U^m_\eta,}
\]
then since \(\psi_\eta \le \rho\eta\), we have \(A \subset U^m_\eta + Q^m_{\rho\eta}\). By Proposition~\ref{openingpropGeneral}, we then have
\[
\sum_{i=1}^j  \eta^{i} \norm{D^i u^\mathrm{op}_\eta}_{L^p(A)} \le \sum_{i=1}^j  \eta^{i} \norm{D^i u^\mathrm{op}_\eta}_{L^p(U^m_\eta + Q^m_{\rho\eta})} \le \Constant \sum_{i=1}^j  \eta^{i} \norm{D^i u}_{L^p(U^m_\eta + Q^m_{2\rho\eta})}.
\]
Thus, for every \(j \in \{1, \dots, k\}\),
\begin{multline}
\label{inequalityMainSmoothening}
\eta^{j} \norm{D^j u^\mathrm{sm}_\eta - D^j u^\mathrm{op}_\eta}_{L^p(Q^m_{1+\gamma})}\\
\le \sup_{v \in B_1^m}{\eta^{j} \norm{\tau_{\psi_\eta v}(D^j u) - D^j u}_{L^p(Q^m_{1+\gamma})}}
+ {\Constant} \sum_{i=1}^j  \eta^{i} \norm{D^i u}_{L^p(U^m_\eta + Q^m_{2\rho\eta})}.
\end{multline}

Given \(0 < \underline{\rho} < \rho\), we apply thickening to the map \(u^\mathrm{sm}_\eta\) in a neighborhood of \(U^\ell_\eta\) of size \(\underline{\rho}\eta\). More precisely, denote by \(\Phi^\mathrm{th} : \R^m \to \R^m\) the smooth map given by Proposition~\ref{propositionthickeningfromaskeleton} with the parameter \(\underline{\rho}\) and let
\[
u^\mathrm{th}_\eta = u^\mathrm{sm}_\eta \circ \Phi^\mathrm{th}.
\] 
Then, \(u^\mathrm{th}_\eta = u^\mathrm{sm}_\eta\) in the complement of \(U^m_\eta + Q^m_{\underline{\rho}\eta}\).
Assuming in addition that
\[
\boxed{\ell + 1 > kp,}
\]
then by Corollary~\ref{corollaryEstimateThickening}, \(u^\mathrm{th}_\eta \in W^{k, p}(K^m_\eta; \R^\nu)\) and for every \(j \in \{1, \dots, k\}\),
\[
\begin{split}
\eta^j\norm{D^j u^\mathrm{th}_\eta - D^j u^\mathrm{sm}_\eta}_{L^p(K^m_\eta)} 
& \leq \eta^j\norm{D^j u^\mathrm{th}_\eta - D^j u^\mathrm{sm}_\eta}_{L^p(U^m_\eta + Q^m_{\underline{\rho}\eta})}\\
& \leq \eta^j\norm{D^j u^\mathrm{th}_\eta}_{L^p(U^m_\eta + Q^m_{\underline{\rho}\eta})} + \eta^j\norm{D^j u^\mathrm{sm}_\eta}_{L^p(U^m_\eta + Q^m_{\underline{\rho}\eta})}\\
& \le \Constant \sum_{i = 1}^j \eta^{i} \norm{D^i u^\mathrm{sm}_\eta}_{L^p(U^m_\eta + Q^m_{\underline{\rho}\eta})}.
\end{split}
\]
Thus, by Proposition~\ref{lemmaConvolutionEstimates} and by Proposition~\ref{openingpropGeneral},
\begin{equation}
\label{inequalityMainThickening}
\begin{split}
\eta^j\norm{D^j u^\mathrm{th}_\eta - D^j u^\mathrm{sm}_\eta}_{L^p(K^m_\eta)} 
& \le \Constant \sum_{i = 1}^j \eta^{i} \norm{D^i u^\mathrm{op}_\eta}_{L^p(U^m_\eta + Q^m_{(\underline \rho + \rho)\eta})}\\
& \le \Constant \sum_{i = 1}^j \eta^{i} \norm{D^i u}_{L^p(U^m_\eta + Q^m_{2\rho\eta})}.
\end{split}
\end{equation}
By the triangle inequality, we deduce from \eqref{inequalityMainOpening}, \eqref{inequalityMainSmoothening} and \eqref{inequalityMainThickening} that for every \(j \in \{1, \dots, k\}\),
\begin{multline*}
\eta^{j} \norm{D^j u^\mathrm{th}_\eta - D^j u}_{L^p(K^m_\eta)}\\
\leq \sup_{v \in B_1^m}{\eta^{j} \norm{\tau_{\psi_\eta v}(D^j u) - D^j u}_{L^p(Q^{m}_{1+\gamma})}} + C \sum_{i=1}^j  \eta^{i} \norm{D^i u}_{L^p(U^m_\eta + Q^m_{2\rho\eta})}.
\end{multline*}
This gives the estimate we claimed since \(K^{m}_{\eta}=Q^m_{1+\gamma}\). We observe that \(u^\mathrm{th}_\eta\) is smooth except on \((U^m_\eta + Q^m_{\underline\rho \eta}) \cap T^{\ell^*}_\eta\) where \(T^{\ell^*}_\eta\) is the dual skeleton corresponding to the cubication \(\mathcal{K}^{m}_\eta\). 
\qed

\medskip

The map \(u^\mathrm{th}_\eta\) need not have its values on the manifold \(N^n\), so we need to estimate the distance between the image of \(u^\mathrm{th}_\eta\) and \(N^n\).

\begin{Part}
The directed Hausdorff distance from the image of the map \(u^\mathrm{th}_\eta\) to the manifold \(N^n\) satisfies the estimate
\begin{multline*}
\Dist_{N^n}{(u^\mathrm{th}_\eta(K^m_\eta\setminus T^{\ell^*}_\eta))}
\le \max \biggl\{  \max_{\sigma^m \in \cK^m_\eta \setminus \cE^m_\eta} \frac{C'}{\eta^{\frac{m}{kp} - 1}} \norm{Du}_{L^{kp}(\sigma^m + Q_{2\rho\eta}^m)},\\
\sup_{x \in U^\ell_\eta + Q_{\underline{\rho}\eta}^m}\frac{C''}{\abs{Q_{s}^m}^2} \int\limits_{Q_{s}^m(x)}\int\limits_{Q_{s}^m(x)} \abs{u^\mathrm{op}_\eta(y) - u^\mathrm{op}_\eta(z)} \dif y\dif z\biggr\},
\end{multline*}
where the directed Hausdorff distance from a set \(S \subset \R^\nu\) to \(N^n\) is
\[
\Dist_{N^n}{(S)} = \sup{\big\{ \dist{(x, N^n)} : x \in S \big\}},
\]
\(\cE^m_\eta\) is a subskeleton of \(\cU^m_\eta\), and \(0 < s < \eta\).

\end{Part}

\textsl{The subskeleton \(\cE^m_\eta\) will be chosen at the end of  Part~2 as the set of bad cubes and \(\cK^m_\eta \setminus \cE^m_\eta\) will be the set of good cubes.
This estimate implies that for every \(\eta > 0\) sufficiently small, the image of \(u^\mathrm{th}_\eta\) is contained in a small tubular neighborhood of \(N^n\).}

\medskip

We first observe that by Proposition~\ref{propositionthickeningfromaskeleton} \((\ref{itempropositionthickeningfromaskeleton2})\), \(\Phi^\textrm{th}(K^m_\eta\setminus (T^{\ell^\ast}\cup U^m_\eta)) \subset K^m_\eta\setminus U^m_\eta\) while by  Proposition~\ref{propositionthickeningfromaskeleton}~\((\ref{itempropositionthickeningfromaskeleton3})\), \(\Phi^\textrm{th}(U^m_\eta\setminus T^{\ell^\ast})\subset U^\ell_\eta+Q^{m}_{\underline{\rho}\eta}\). Hence,
\[
\Phi^\mathrm{th}(K^m_\eta \setminus T^{\ell^*}_\eta) \subset (K^m_\eta \setminus U^m_\eta) \cup (U^\ell_\eta + Q^m_{\underline{\rho}\eta}).
\]
In terms of the directed Hausdorff distance we have
\[
\Dist_{N^n}{(u^\mathrm{th}_\eta(K^m_\eta\setminus T^{\ell^*}_\eta))} \le \Dist_{N^n}\Bigl( \big(u^\mathrm{sm}_\eta \big((K^m_\eta \setminus U^m_\eta) \cup (U^\ell_\eta + Q^m_{\underline{\rho}\eta})\big) \Bigr).
\]
Since the image of the map \(u^\mathrm{op}_\eta\) obtained by opening \(u\) is contained in \(N^n\) (see Lemma~\ref{lemmaOpeningLp}), for every \(x \in K^m_\eta\) we have
\[
\dist{(u^\mathrm{sm}_\eta(x), N^n)} 
 \le \frac{1}{\abs{Q_{\psi_\eta(x)}^m}} \int\limits_{Q_{\psi_\eta(x)}^m(x)} \abs{u^\mathrm{sm}_\eta(x) - u^\mathrm{op}_\eta(z)} \dif z.
\]
On the other hand, since \(u^\mathrm{sm}_\eta\) is the convolution of \(u^\mathrm{op}_\eta\) with a mollifier,
\[
\begin{split}
  \abs{u^\mathrm{sm}_\eta(x) - u^\mathrm{op}_\eta(z)} 
  &\le \frac{1}{\psi_\eta(x)^m} \int\limits_{B_{\psi_\eta(x)}^m(x)} \varphi \Bigl(\frac{x - y}{\psi_\eta (x)}\Bigr) \abs{u^\mathrm{op}_\eta(y) - u^\mathrm{op}_\eta(z)} \dif y\\
  &\le \frac{\NewConstant}{\abs{Q_{\psi_\eta(x)}^m}} \int\limits_{Q_{\psi_\eta(x)}^m(x)} \abs{u^\mathrm{op}_\eta(y) - u^\mathrm{op}_\eta(z)} \dif y .
\end{split}
\]
Thus,
\begin{equation}
\label{equationDistanceBasic}
\dist{(u^\mathrm{sm}_\eta(x), N^n)} 
\le \frac{\SameConstant}{\abs{Q_{\psi_\eta(x)}^m}^2} \int\limits_{Q_{\psi_\eta(x)}^m(x)}\int\limits_{Q_{\psi_\eta(x)}^m(x)} \abs{u^\mathrm{op}_\eta(y) - u^\mathrm{op}_\eta(z)} \dif y\dif z.
\end{equation}

Since \(N^n\) is a compact subset of \( \R^\nu \), \(u\) is bounded. 
By the Gagliardo-Nirenberg interpolation inequality (see \cite{Gagliardo, Nirenberg1959}),
\(D u \in L^{kp}(Q^m_{1 + 2\gamma})\). By the Poincar\'e-Wirtinger inequality,
\begin{multline*}
 \frac{1}{\abs{Q_{\psi_\eta(x)}^m}^2} \int\limits_{Q_{\psi_\eta(x)}^m(x)}\int\limits_{Q_{\psi_\eta(x)}^m(x)} \abs{u^\mathrm{op}_\eta(y) - u^\mathrm{op}_\eta(z)} \dif y\dif z \\
 \le \frac{\Constant}{\psi_\eta(x)^{\frac{m}{kp} - 1}} \norm{Du^\mathrm{op}_\eta}_{L^{kp}(Q_{\psi_\eta(x)}^m(x))}.
\end{multline*}
Since \(\psi_\eta \le \rho \eta\), if \(\sigma^m \in \cK^m_\eta\) is such that \(x \in \sigma^m\), then \(Q_{\psi_\eta(x)}^m(x) \subset \sigma^m + Q_{\rho\eta}^m\). Hence,
\[
\begin{split}
\dist{(u^\mathrm{sm}_\eta(x), N^n)} 
& \le \frac{\Constant}{\psi_\eta(x)^{\frac{m}{kp} - 1}} \norm{Du^\mathrm{op}_\eta}_{L^{kp}(Q_{\psi_\eta(x)}^m(x))} \\
& \le \frac{\SameConstant}{\psi_\eta(x)^{\frac{m}{kp} - 1}} \norm{Du^\mathrm{op}_\eta}_{L^{kp}(\sigma^m + Q_{\rho\eta}^m)}.
\end{split}
\]
Thus, by Addendum~\ref{addendumW1kp} to Proposition~\ref{openingpropGeneral}, 
\[
\dist{(u^\mathrm{sm}_\eta(x), N^n)} \le \frac{\Constant}{\psi_\eta(x)^{\frac{m}{kp} - 1}} \norm{Du}_{L^{kp}(\sigma^m + Q_{2\rho\eta}^m)}.
\]
We rewrite this estimate for every \(x \in K^m_\eta\) as
\begin{equation}
\label{equationDistanceKU}
\dist{(u^\mathrm{sm}_\eta(x), N^n)} \le \Big(\frac{\eta}{\psi_\eta(x)}\Big)^{\frac{m}{kp} - 1} \frac{\SameConstant}{\eta^{\frac{m}{kp} - 1}} \norm{Du}_{L^{kp}(\sigma^m + Q_{2\rho\eta}^m)}.
\end{equation}

If \(x\in (U^{\ell}_\eta+Q^{m}_{\underline{\rho}\eta}) \cap U^{m}_{\eta}\), then \(x\in \sigma^m\) for some cube \(\sigma^m \in \cU^m_\eta\).  If 
\[
\boxed{\psi_\eta (x) \le (\rho - \underline\rho) \eta,}
\] 
then \(Q_{\psi_\eta(x)}^m(x) \subset U^\ell_\eta + Q^m_{{\rho}\eta}\). 
By Addendum~\ref{addendumVMO} to Proposition~\ref{openingpropGeneral}, we have
\begin{multline*}
\frac{1}{\abs{Q_{\psi_\eta(x)}^m}^2} \int\limits_{Q_{\psi_\eta(x)}^m(x)}\int\limits_{Q_{\psi_\eta(x)}^m(x)} \abs{u^\mathrm{op}_\eta(y) - u^\mathrm{op}_\eta(z)} \dif y\dif z \\
\le (\psi_\eta(x))^{1 - \frac{\ell}{kp}} \frac{\Constant}{\eta^{\frac{m - \ell}{kp}}} \norm{Du}_{L^{kp}(\sigma^m + Q^m_{2\rho\eta})}.
\end{multline*}
Therefore, 
\begin{equation*}
\dist{(u^\mathrm{sm}_\eta(x), N^n)} \le (\psi_\eta(x))^{1 - \frac{\ell}{kp}} \frac{\SameConstant}{\eta^{\frac{m - \ell }{kp}}} \norm{Du}_{L^{kp}(\sigma^m + Q^m_{2\rho\eta})}.
\end{equation*}
We rewrite this estimate for every \(x\in (U^{\ell}_\eta+Q^{m}_{\underline{\rho}\eta}) \cap U^{m}_{\eta}\) as
\begin{equation}\label{equationDistanceKUbis}
\dist{(u^\mathrm{sm}_\eta(x), N^n)} \le \Big(\frac{\psi_\eta(x)}{\eta}\Big)^{1 - \frac{\ell}{kp}} \frac{\SameConstant}{\eta^{\frac{m}{kp} - 1}} \norm{Du}_{L^{kp}(\sigma^m + Q^m_{2\rho\eta})}.
\end{equation}

We now describe the function \(\psi_\eta\) that we shall take. 
Given two parameters \(0 < s < t\) and given a function \(\zeta \in C^\infty(Q^m_{1+2\gamma})\),
we define 
\[
\psi_\eta = t\zeta+s (1 - \zeta).
\]
More precisely, let \(\cE^m_\eta\) be a subskeleton of \(\cU^m_\eta\) such that 
\[
\boxed{E^m_\eta \subset \Int{U^m_\eta}}
\]
in the relative topology of \(Q^{m}_{1+\gamma}\). Since \(\dist{(E^m_\eta, K^m_\eta \setminus U^m_\eta)} \ge \eta\), we take a function \(\zeta \in C^\infty(K^m_\eta)\) such that
\begin{enumerate}[$(i)$]
\item \(0 \le \zeta \le 1\) in \(K^m_\eta\),
\item \(\zeta = 1\) in \(K^m_\eta \setminus U^m_\eta\),
\item \(\zeta = 0\) in \(E^m_\eta\),
\item for every \(j \in \{1, \dots, k\}\), \(\eta^j\norm{D^j\zeta}_{L^\infty} \le \tilde C\), for some constant \(\tilde C > 0\) depending only on \(m\).
\end{enumerate}
Thus, \(\supp{D\psi_\eta} \subset U_\eta^m\) and
\[
\eta^j\norm{D^j\psi_\eta}_{L^\infty} \le \tilde C t.
\]
In order to apply Proposition~\ref{lemmaConvolutionEstimates} and to have \(\psi_\eta \le (\rho - \underline{\rho}) \eta\), we choose 
\[
t = \min \Bigl\{ \frac{\kappa}{\tilde C}, \rho - \underline{\rho}\Bigr\} \, \eta,
\]
for some fixed number \(0 < \kappa < 1\).

Since \(\psi_\eta = t \) in \(K^m_\eta \setminus U^m_\eta\) and \(t \ge c\eta\) for some constant \(c > 0\) independent of \(\eta\), we have from \eqref{equationDistanceKU},
\[
\Dist_{N^n}{\big(u^\mathrm{sm}_\eta(K^m_\eta \setminus U^m_\eta)\big)} \le  \max_{\sigma^m \in \cK^m_\eta \setminus \cU^m_\eta} \frac{\Constant}{\eta^{\frac{m}{kp} - 1}} \norm{Du}_{L^{kp}(\sigma^m + Q_{2\rho\eta}^m)}.
\]
Since \(\psi_\eta = s \) in \(E^m_\eta\), we have from \eqref{equationDistanceBasic},
\begin{multline*}
\Dist_{N^n}{\Bigl(u^\mathrm{sm}_\eta\big((U^\ell_\eta + Q^m_{\underline{\rho}\eta})\cap E^{m}_{\eta}\big)\Bigr)}\\
\le  \sup_{x \in U^\ell_\eta + Q_{\underline{\rho}\eta}^m}\frac{C_1}{\abs{Q_{s}^m}^2} \int\limits_{Q_{s}^m(x)}\int\limits_{Q_{s}^m(x)} \abs{u^\mathrm{op}_\eta(y) - u^\mathrm{op}_\eta(z)} \dif y\dif z.
\end{multline*}
Finally, if 
\[
\boxed{\ell \le kp,}
\]
then by \eqref{equationDistanceKUbis} and by the estimate \(\psi_\eta(x) \leq t = \Constant \eta\), we get
\[
\Dist_{N^n}{\Bigl(u^\mathrm{sm}_\eta\big((U^\ell_\eta + Q^m_{\underline{\rho}\eta})\cap (U^{m}_{\eta} \setminus E^m_{\eta} \big)\Bigr)}
\le  \max_{\sigma^m \in \cU^m_\eta \setminus \cE^m_\eta} \frac{\Constant}{\eta^{\frac{m}{kp} - 1}} \norm{Du}_{L^{kp}(\sigma^m + Q_{2\rho\eta}^m)}.
\]
Since we have already required that \(\ell+1>kp\), we are thus led to take 
\[
\boxed{\ell = \floor{kp}.}
\] 
We deduce that
\begin{multline*}
\Dist_{N^n}{(u^\mathrm{th}_\eta(K^m_\eta\setminus T^{\ell^*}_\eta))}
\le \max \biggl\{ \max_{\sigma^m \in \cK^m_\eta \setminus \cE^m_\eta} \frac{C'}{\eta^{\frac{m}{kp} - 1}} \norm{Du}_{L^{kp}(\sigma^m + Q_{2\rho\eta}^m)},\\
 \sup_{x \in U^\ell_\eta + Q_{\underline{\rho}\eta}^m}\frac{C''}{\abs{Q_{s}^m}^2} \int\limits_{Q_{s}^m(x)}\int\limits_{Q_{s}^m(x)} \abs{u^\mathrm{op}_\eta(y) - u^\mathrm{op}_\eta(z)} \dif y\dif z\biggr\}.
\end{multline*}
This gives the estimate we claimed.

The nearest point projection \(\Pi\) onto \(N^n\) is well-defined and  smooth on  a tubular neighborhood of \(N^n\) of radius \(\iota > 0\). 
We now choose the subskeleton \(\cE^m_\eta\) used in the definition of \(\zeta\) and \(\psi_\eta\) as the set of cubes \(\sigma^m \in \cK^m_\eta\) such that
\[
\frac{C'}{\eta^{\frac{m}{kp} - 1}} \norm{Du}_{L^{kp}(\sigma^m + Q_{2\rho\eta}^{m})} > \iota.
\]
Thus,
\[
\max_{\sigma^m \in \cK^m_\eta \setminus \cE^m_\eta} \frac{C'}{\eta^{\frac{m}{kp} - 1}} \norm{Du}_{L^{kp}(\sigma^m + Q_{2\rho\eta}^m)} \le \iota.
\]

We then take the subskeleton \(\cU^m_\eta\) used in the constructions of opening and thickening as the the set of cubes \(\sigma^m \in \cK^m_\eta\) which intersect some cube in \(\cE^m_\eta\); in particular \(\Int{E^m_\eta} \subset U^m_\eta\) in the relative topology of \(Q_{1 + \gamma}^m\). 

In view of the uniform limit of Addendum~\ref{addendumVMO} to Proposition~\ref{openingpropGeneral}, since
\(\ell \le kp\), for every \(s > 0\) small enough,
\[
\sup_{x \in U^\ell_\eta + Q_{\underline{\rho}\eta}^m}\frac{C''}{\abs{Q_{s}^m}^2} \int\limits_{Q_{s}^m(x)}\int\limits_{Q_{s}^m(x)} \abs{u^\mathrm{op}_\eta(y) - u^\mathrm{op}_\eta(z)} \dif y\dif z \le \iota.
\]
We conclude that \(u^\mathrm{th}_\eta(K^m_\eta \setminus T^{\ell^*}_\eta)\) is contained in a tubular neighborhood of \(N^n\) of radius \(\iota\).
\qed

\begin{Part}
The maps \(\Pi \circ u^\mathrm{th}_\eta\) converge to \(u\) in \(W^{k, p}(Q_1^m; N^n)\) as \(\eta\) tends to \(0\).
\end{Part}

Using the estimate from Part~1, we show that for every \(j \in \{1, \dots, k\}\),
\[
\lim\limits_{\eta \to 0} \norm{D^j u^\mathrm{th}_\eta - D^j u}_{L^p(Q^m_{1+\gamma})} = 0.
\]
By continuity of the translation operator in \(L^p\) (see the remark following Proposition~\ref{lemmaConvolutionEstimatesLp}),
\begin{equation}\label{equationLpConvergenceTranslates}
 \lim_{\eta \to 0}
\sup_{v \in B_1^m}{\norm{\tau_{\psi_\eta v}(D^j u) - D^j u}_{L^p(Q^{m}_{1+\gamma})}}
= 0.
\end{equation}

We now need to show that 
\[
\lim_{\eta \to 0} \sum_{i=1}^j  \eta^{i-j} \norm{D^i u}_{L^p(U^m_\eta + Q^m_{2\rho\eta})} = 0.
\]
By the Gagliardo-Nirenberg interpolation inequality,
for every \(i \in \{1, \dots, k-1\}\), \(D^i u \in L^{\frac{kp}{i}}(Q^m_{1 + 2\gamma})\). 
By H\"older's inequality, for every \(i \in \{1, \dots, k\}\) we have
\[
\begin{split}
\eta^{i-j} \norm{D^i u}_{L^p(U^m_\eta + Q^m_{2\rho\eta})} 
& \le \eta^{i-j} \abs{U^m_\eta + Q^m_{2\rho\eta}}^{\frac{k-i}{kp}} \norm{D^i u}_{L^{\frac{kp}{i}}(U^m_\eta + Q^m_{2\rho\eta})}\\
& = \eta^{k-j} \left(\frac{\abs{U^m_\eta + Q^m_{2\rho\eta}}}{\eta^{kp}} \right)^{\frac{k - i}{kp}} \norm{D^i u}_{L^{\frac{kp}{i}}(U^m_\eta + Q^m_{2\rho\eta})}.
\end{split}
\]
From this estimate, we need that \(\abs{U^m_\eta + Q^m_{2\rho\eta}} = O(\eta^{kp})\) as \(\eta \to 0\).
We observe that \(\abs{U^m_\eta + Q^m_{2\rho\eta}}\) satisfies the following estimate in terms of the number of elements \(\#\cU^m_\eta\) of the subskeleton \(\cU^m_\eta\),
\[
\bigabs{U^m_\eta + Q^m_{2\rho\eta}} \le 2^m (\eta + 2\rho\eta)^m (\#\cU^m_\eta) = \NewConstant \eta^m (\#\cU^m_\eta).
\]
Note that for every cube \(\sigma^m \in \cU^m_\eta\), if \(\tau^m \in \cE^m_\eta\) intersects \(\sigma^m\), then \(\tau^m + Q_{2\rho\eta}^m \subset \sigma^m + Q^m_{2(1 + \rho)\eta}\). Denoting \(\sigma^m\) by \(Q^m_\eta(a)\), we have
\(
\tau^m + Q_{2\rho\eta}^m \subset Q^m_{\alpha\eta} (a),
\)
where \(\alpha = 3 + 2 \rho\), whence
\[
\tau^m + Q_{2\rho\eta}^m \subset Q^m_{\alpha\eta} (a) \cap Q_{1 + 2\gamma}^m.
\]
By the definition of \(\cE^m_\eta\),
\[
\iota < \frac{C'}{\eta^{\frac{m}{kp} - 1}} \norm{Du}_{L^{kp}(\tau^m + Q_{2\rho\eta}^m)} \le \frac{C'}{\eta^{\frac{m}{kp} - 1}} \norm{Du}_{L^{kp}(Q^m_{\alpha\eta} (a)  \cap Q_{1 + 2\gamma}^m)}.
\]
Thus, for every \(Q^m_{\eta} (a) \in \cU^m_\eta\),
\[
1 < \frac{\Constant}{\eta^{m - kp}} \int\limits_{Q^m_{\alpha\eta} (a)  \cap Q_{1 + 2\gamma}^m} \abs{Du}^{kp}.
\]
Since the cubes \(Q^m_{\alpha\eta} (a)\) intersect each other finitely many times  and the number of overlaps only depend on \(\alpha\) and on the dimension \(m\),
\[
\#\cU^m_\eta \le \frac{\SameConstant}{\eta^{m - kp}} \sum_{Q^m_{\eta} (a) \in \cU^m_\eta} \int\limits_{Q^m_{\alpha\eta} (a)  \cap Q_{1 + 2\gamma}^m} \abs{Du}^{kp} \le \frac{\Constant}{\eta^{m - kp}}  \int\limits_{Q^m_{1 + 2 \gamma}} \abs{Du}^{kp}.
\]
We deduce that
\[
\bigabs{U^m_\eta + Q^m_{2\rho\eta}} 
\le
\Constant \eta^m \frac{1}{\eta^{m - kp}} \int\limits_{Q^m_{1 + 2 \gamma}} \abs{Du}^{kp} = \SameConstant  \eta^{kp} \int\limits_{Q^m_{1 + 2 \gamma}}  \abs{Du}^{kp}.
\]
This means that 
\[
  \limsup_{\eta \to 0}{\frac{\bigabs{U^m_\eta + Q^m_{2\rho\eta}}}{\eta^{kp}}} < \infty.
\]
Hence, by Lebesgue's dominated convergence theorem,
\[
\lim_{\eta \to 0}\norm{D^i u}_{L^{\frac{kp}{i}}(U^m_\eta + Q^m_{2\rho\eta})} = 0.
\]
In view of \eqref{equationLpConvergenceTranslates} and the estimate from Part~1, we have \(\lim\limits_{\eta \to 0} \norm{D^j u^\mathrm{th}_\eta - D^j u}_{L^p(Q^m_{1+\gamma})} = 0\).

Recall that \(u^\mathrm{th}_\eta = u^\mathrm{sm}_\eta\) in the complement of \(U^m_\eta + Q^m_{\underline\rho\eta}\). Since \(u^\mathrm{sm}_\eta \to u\) in measure and \(\abs{U^m_\eta + Q^m_{\underline\rho\eta }}\to 0\) as \(\eta \to 0\),  \(u^\mathrm{th}_\eta \to u\) in measure as \(\eta \to 0\). Hence, \(u^\mathrm{th}_\eta\) converges to \(u\) in \(L^p(Q^m_{1+\gamma})\) and
\[
\lim_{\eta \to 0}{\norm{u^\mathrm{th}_\eta - u}_{W^{k, p}(Q_{1+\gamma}^m)}} = 0.
\]
Therefore,
\[
\lim_{\eta \to 0}{\norm{\Pi\circ u^\mathrm{th}_\eta - u}_{W^{k, p}(Q_{1+\gamma}^m)}} = 0.
\]
This gives the conclusion of this part.
\qed

\begin{Part}
The map \(\Pi \circ u^\mathrm{th}_\eta\) belongs to the class \(R_{\ell^*}(Q^m; N^n)\).
\end{Part}

It suffices to prove the pointwise estimates of \(D^j (\Pi\circ u^\mathrm{th}_\eta)\). Since \(\Pi \circ u^\mathrm{th}_\eta = (\Pi\circ u^\mathrm{sm}_\eta) \circ \Phi^\mathrm{th}\) and the map \(\Pi \circ u^\mathrm{sm}_\eta\) is smooth in \(K^m_\eta\), by the chain rule for higher order derivatives,
\begin{align*}
\abs{D^j (\Pi\circ u^\mathrm{th}_\eta)} 
& \le \Constant \sum_{i = 1}^j \sum_{\substack{1 \le \alpha_1 \le \dotsc \le \alpha_i\\\alpha_1 + \dots + \alpha_i = j}} \abs{D^i(\Pi \circ u^\mathrm{sm}_\eta)} \abs{D^{\alpha_1}\Phi^\mathrm{th}} \dotsm \abs{D^{\alpha_i}\Phi^\mathrm{th}}\\
& \le \Constant \sum_{i = 1}^j \sum_{\substack{1 \le \alpha_1 \le \dotsc \le \alpha_i\\\alpha_1 + \dots + \alpha_i = j}}  \abs{D^{\alpha_1}\Phi^\mathrm{th}} \dotsm \abs{D^{\alpha_i}\Phi^\mathrm{th}}.
\end{align*}
By Proposition~\ref{propositionthickeningfromaskeleton}~\((\ref{itempropositionthickeningfromaskeleton5})\), we have for \(x \in K^m_\eta \setminus T^{\ell^*}_\eta\),
\begin{align*}
\abs{D^j (\Pi\circ u^\mathrm{th}_\eta)(x)} 
& \le \Constant \sum_{i = 1}^j \sum_{\substack{1 \le \alpha_1 \le \dotsc \le \alpha_i\\\alpha_1 + \dots + \alpha_i = j}}  \frac{\eta}{\big(\dist{(x, T^{\ell^*}_\eta)}\big)^{\alpha_1}} \dotsm \frac{\eta}{\big(\dist{(x, T^{\ell^*}_\eta)}\big)^{\alpha_i}}\\
& \le  \frac{\Constant}{\big(\dist{(x, T^{\ell^*}_\eta)}\big)^{j}}.
\end{align*}
This concludes the proof of the theorem.
\end{proof}

%%%%%%%%%%%%%%%%%%%%%%%%%%%%%%%%%%%%%%%%%%%%%%%%%%%%%
%%%%%%%%%%%%%%%%%%%%%%%%%%%%%%%%%%%%%%%%%%%%%%%%%%%%%
%%%%%%%%%%%%%%%%%%%%%%%%%%%%%%%%%%%%%%%%%%%%%%%%%%%%%

\section{Proof of Theorem~\ref{theoremDensityManifoldNontrivial}}

Let \(kp < m\).
It is a consequence of Theorem~\ref{theoremDensityManifoldNontrivialwithouttopologicalcondition}  that  \(R_{m-\floor{kp}-1}(Q^m ; N^n)\) is dense in \(W^{k,p}(Q^m ; N^n)\).
In this section, we prove that if \(\pi_{\floor{kp}} (N^n) \not \simeq \{0\}\) and if \(i\in \{0, \dotsc, m-1\}\) is such that 
\begin{enumerate}[\((a)\)]
\item\(R_{i}(Q^m; N^n) \subset W^{k, p}(Q^m; N^n)\),
\item \(R_{i}(Q^m; N^n)\) is dense in \(W^{k, p}(Q^m; N^n)\),
\end{enumerate}
then \(i = m-\floor{kp}-1\).

\medskip
We first prove that \(i <  m-\floor{kp}\). 
For this purpose, let \(\gamma : \R \to N^n\) be a geodesic in \(N^n\).
Given \(i \ge  m-\floor{kp}\), the map \(u : \overline{Q}^m \to N^n\) defined for \(x=(x', x'')\in \overline Q^{m-i} \times \overline Q^i\) by
\[
u(x)
= \gamma \big(\log{\abs{x'}} \big)
\]
belongs to \(R_{i}(Q^m; N^n)\). 
Taking \(\gamma\) parametrized by arc-length, we have
\[
|Du(x)|= \frac{1}{\abs{x'}}.
\]
Since \(i \ge  m-\floor{kp}\), it follows that \(Du \not\in L^{\floor{kp}}(Q^m)\).
By the Gagliardo-Nirenberg interpolation inequality, we deduce that \(R_{i}(Q^m; N^n) \not\subset W^{k,p}(Q^m ; N^n)\).

\medskip
We now prove that \(i >  m-\floor{kp}-2\).
Given a smooth map \(\varphi : \S^{\floor{kp}} \to  N^n\),
we define \(u : \overline Q^m \to \N^n\) for \(x=(x', x'') \in \overline Q^{\floor{kp}+1} \times \overline Q^{m-\floor{kp}-1}\) by
\[
u(x)= \varphi\Big(\frac{x'}{\abs{x'}}\Big).
\] 
Then, \(u \in W^{k, p}(Q^m; N^n)\). 
Given \(i \in \{0, \dotsc, m-\floor{kp}-2\}\), 
assume by contradiction that there exists a sequence \((u_j)_{j\in \N}\) in \(R_i(Q^m; N^n)\)  converging to \(u\) in \(W^{k, p}(Q^m; N^n)\). 
Passing to a subsequence if necessary, for almost every \(x'' \in Q^{m-\floor{kp}-1}\) and for almost every \(\rho \in (0, 1)\), the sequence \((u_j|_{\S^{\floor{kp}}_{\rho}\times \{x''\}})_{j \in \N}\) converges to \(u|_{\S^{\floor{kp}}_{\rho}\times \{x''\}}\) in \(W^{k,p}(\S^{\floor{kp}}_{\rho}; N^n)\), whence in \(\textrm{BMO}(\S_\rho^{\floor{kp}} ; N^n)\).

For every \(j \in \N\), denote by \(T_j\) a finite union of \(i\) dimensional planes such that \(u_j \in C^{\infty}(\overline{Q}^m \setminus T_j ; N^n)\).  
Since \(i \le  m-\floor{kp}-2\), for every \((x'', \rho) \in Q^{m-\floor{kp}-1}\times (0, 1)\) such that \(\S^{\floor{kp}}_{\rho}\times \{x''\} \subset \overline{Q}^m \setminus T_j\), there exist \(a\in \overline{Q}^m \setminus T_j \) and  a continuous map \(h : [0,1] \times (\S^{\floor{kp}}_{\rho}\times \{x''\}) \to \overline{Q}^m \setminus T_j \) such that  for every \(y \in \S^{\floor{kp}}_{\rho}\times \{x''\}\), \(h(0, y) = y \) and \(h(1,y) = a\). 
This implies that \(u_j|_{\S^{\floor{kp}}_{\rho}\times \{x''\}}\)  is homotopic to a constant.

We recall that homotopy classes are preserved under BMO convergence:

\begin{claim}
\label{standardtopologicalresult}
Let \((v_j)_{j\in \N} \) be a sequence in \( C^{0}(\S^{\floor{kp}} ; N^n)\) which converges to \(v  \in C^{0}(\S^{\floor{kp}} ; N^n)\) in \(\mathrm{BMO}(\S^{\floor{kp}} ; N^n)\). 
Then, for every \(j \in \N\) sufficiently large, \(v_j\) is homotopic to \(v\) in \(C^{0}(\S^{\floor{kp}}; N^n)\).
 \end{claim}
 
This claim is essentially \cite[Lemma~A.19]{Brezis-Nirenberg} but we present a proof for the convenience of the reader.
 
\begin{proof}[Proof of the Claim]
For every \(\epsilon >0\), we consider the map \(v_{\epsilon} : \S^{\floor{kp}} \to N^n\) defined for \(x \in \S^{\floor{kp}}\) by
\[
v_{\epsilon}(x) 
= \frac{1}{\abs{D^{\floor{kp}}_{\epsilon}(x)}} \int\limits_{D^{\floor{kp}}_{\epsilon}(x)}v
\] 
where \(D^{\floor{kp}}_{\epsilon}(x) =\S^{\floor{kp}} \cap Q^{\floor{kp}+1}_{\epsilon}(x)\).
Accordingly, for every \(j \in \N\) we define \(v_{j, \epsilon}\), with \(v\) replaced by \(v_j\).

The nearest point projection \(\Pi\) is well-defined and smooth on a tubular neighborhood of \(N^n\) of radius \(\iota >0\).
Since \(v_\epsilon\) converges uniformly to \(v\) as \(\epsilon\) tends to \(0\), there exists \(\epsilon_1 > 0\) such that for every \(0 < \epsilon \le \epsilon_1\), \(\Pi \circ v_\epsilon\) is well-defined and is homotopic to \(v\). 

Next, for every \(j\in \N\) and for every \(x\in \S^{\floor{kp}}\), since \(v_j(x) \in N^n\),
\[
\begin{split}
\dist{(v_{j,\epsilon}(x) , N^n)} 
& \leq  \frac{1}{\abs{D^{\floor{kp}}_{\epsilon}(x)}} \int\limits_{D^{\floor{kp}}_{\epsilon}(x)}
\biggl| {v_j(y) - \frac{1}{\abs{D^{\floor{kp}}_{\epsilon}(x)}}\int\limits_{D^{\floor{kp}}_{\epsilon}(x)} v_j}
\biggr| \dif y \\
& \leq \norm{v_j - v}_{\textrm{BMO}(\S^{\floor{kp}})} + 
2\sup_{y \in D^{\floor{kp}}_{\epsilon}(x)}{\abs{v(y) - v(x)}}.
\end{split}
\]
Since the sequence \((v_j)_{j \in \N}\) converges to \(v\) in \(\textrm{BMO}(\S^{\floor{kp}})\) and \(v\) is uniformly continuous, there exist \(J \in \N\) and \(\epsilon_2 > 0\) such that for every \(j \ge J\) and for every \(0 < \epsilon \le \epsilon_2\),
\[
\dist{(v_{j, \epsilon}(x) , N^n)} \le \iota.
\]
In particular, \(\Pi \circ v_{j, \epsilon}\) is well-defined and the continuous extension of the function \(t \in (0, 1] \mapsto \Pi \circ v_{j, t\epsilon} \) gives a homotopy between \(\Pi \circ v_{j, \epsilon}\) and \(v_j\).

Finally, for every \(\epsilon > 0\) the sequence \((v_{j, \epsilon})_{j \in \N}\) converges uniformly to \(v_{\epsilon}\).
For \(0 < \epsilon < \min{\{\epsilon_1, \epsilon_2\}}\) and for \(j \ge J\), the functions \(\Pi \circ v_{j, \epsilon}\) are well-defined and converge uniformly to \(\Pi \circ v_\epsilon\) as \(j\) tends to infinity.
Thus, there exists \(\overline J \ge J\) such that for every \(j \in \N\) with \(j \ge \overline{J}\), \(\Pi \circ v_{j, \epsilon}\) is homotopic to \(\Pi \circ v_{\epsilon}\).
By transitivity of the homotopy relation, we conclude from the above that for every such \(j\), \(v_j\) is homotopic to \(v\).
\end{proof}

We deduce from the claim that \(u|_{\S_\rho^{\floor{kp}} \times \{x''\}}\) is homotopic to a constant,
whence \(\varphi : \S^{\floor{kp}} \to  N^n\) is homotopic to a constant.
Since \(\pi_{\floor{kp}} (N^n) \not \simeq \{0\}\) and \(\varphi : \S^{\floor{kp}} \to N^n\) is an arbitrary smooth function, we get a contradiction.
This completes the proof of Theorem~\ref{theoremDensityManifoldNontrivial}.
\qed

%%%%%%%%%%%%%%%%%%%%%%%%%%%%%%%%%%%%%%%%%%%%%%%%%%%%%
%%%%%%%%%%%%%%%%%%%%%%%%%%%%%%%%%%%%%%%%%%%%%%%%%%%%%
%%%%%%%%%%%%%%%%%%%%%%%%%%%%%%%%%%%%%%%%%%%%%%%%%%%%%

%\section{Tools for the proof of Theorem~\ref{theoremDensityManifoldMain}}

\section{Continuous extension property}

From Theorem~\ref{theoremDensityManifoldNontrivialwithouttopologicalcondition} we are able to approximate a map by another map which is smooth except on a dual skeleton of dimension \(\floor{kp}^*\). We would like to modify our approximation near this singular set in order to obtain a smooth map. An important tool will be the following:

\begin{proposition}
\label{propositionSmoothExtension}
Let \(\cK^m\) be a skeleton of radius \(\eta > 0\),  \(\ell \in \{0, \ldots, m-1 \}\), \(\cT^{\ell^*}\) be the dual skeleton of \(\cK^\ell\) and let \(u \in C^\infty(K^m \setminus T^{\ell^*}; N^n)\). If there exists \(f \in  C^0(K^m; N^n)\) such that \(f|_{K^\ell} = u|_{K^\ell}\), then for every \(0 < \mu < 1\), there exists \(v \in C^\infty(K^m; N^n)\) such that \(v = u\) on \(K^m \setminus (T^{\ell^*} + Q^m_{\mu\eta})\).
\end{proposition}

In the proof of Proposition~\ref{propositionSmoothExtension}, we shall rely on the fact that \(K^\ell\) is a homotopy retract of \(K^m \setminus T^{\ell^*}\), 
that is, there exists a continuous retraction of \(K^m \setminus T^{\ell^*}\) onto \(K^\ell\) which is homotopic to the identity map in \(K^m \setminus T^{\ell^*}\):

\begin{fact}
\label{factHomotopyRetraction}
There exists a continuous homotopy \(H_\ell : [0, 1] \times (K^m \setminus T^{\ell^*}) \to K^m \setminus T^{\ell^*}\) such that 
\begin{enumerate}[$(i)$]
\item for every \(x \in K^m\setminus T^{\ell^*}\), \(H_\ell(0, x) = x\),
\label{item15431}
\item for every \(x \in K^m \setminus T^{\ell^*}\), \(H_\ell(1, x) \in K^\ell\),
\label{item15432}
\item for every \(x \in K^\ell\), \(H_\ell(1, x) = x\).
\label{item15433}
\end{enumerate}
\end{fact}

\begin{proof}[Proof of Proposition~\ref{propositionSmoothExtension}]
Given \(0 < \underline\delta < \delta < \overline\delta < \mu\), let \(\varphi : K^m \to [0, 1 ]\) be a continuous function such that
\begin{itemize}[\(-\)]
\item for every \(x \in K^m \setminus (T^{\ell^*} + Q_{\overline\delta\eta}^m)\), \(\varphi(x) = 0\),
\item for every \(x \in \partial(T^{\ell^*} + Q_{\delta\eta}^m)\), \(\varphi(x) = 1\).
\item for every \(x \in T^{\ell^*} + Q_{\underline\delta \eta}^m\), \(\varphi(x) = 0\).
\end{itemize}
We define \(w : K^m \to N^n\) by
\[
w(x) = 
\begin{cases}
(u \circ H_\ell)(\varphi(x), x)		& \text{if \(x \in K^m \setminus (T^{\ell^*} + Q_{\delta\eta}^m)\)
},\\
(f \circ H_\ell)(\varphi(x), x)		& \text{if \(x \in (T^{\ell^*} + Q_{\delta\eta}^m) \setminus T^{\ell^*}\)},\\
f(x)		& \text{if \(x \in T^{\ell^*}\)}.
\end{cases}
\]
By properties~\((\ref{item15431})\) and \((\ref{item15432})\) of Fact~\ref{factHomotopyRetraction}, \(w\) is well-defined and continuous on \(K^m\), and \(w = u\) on \(K^m \setminus (T^{\ell^*} + Q_{\overline\delta\eta}^m)\). Let \(\overline w : \R^m \to \R^\nu\) be a continuous extension of \(w\). Given a mollifier \(\varphi \in C_c^\infty(B_1^m)\), there exists a nonnegative function \(\psi \in C^\infty(\R^m)\) such that for any \(\iota > 0\),
\begin{itemize}[\(-\)]
\item \(\supp{\psi} \subset T^{\ell^*} + Q_{\mu\eta}^m\),
\item \(\psi > 0\) in a neighborhood of \(T^{\ell^*} + Q_{\overline\delta\eta}^m\),
\item \(\norm{\varphi_\psi \ast \overline w - \overline w}_{L^\infty(\R^m)} \le \iota\).
\end{itemize}
If the nearest point projection \(\Pi\) onto \(N^n\) is well-defined and  smooth on  a tubular neighborhood of \(N^n\) of radius \(\iota > 0\), then the map \(\Pi \circ (\varphi_\psi \ast \overline w)\) restricted to \(K^m\) satisfies all the required properties.
\end{proof}

The natural question that arises is whether a continuous extension of \(u|_{K^\ell}\) to \(K^m\) exists. This property depends on the skeleton \(\cK^m\) and on the manifold \(N^n\).

\begin{proposition}
\label{propositionContinuousExtensionProperty}
Let \(\cK^m\) be a skeleton of radius \(\eta > 0\) and \(\ell \in \{0, \ldots, m-1\}\).
If \(K^m\) is a cube and if \(\pi_\ell(N^n)  \simeq \{0\}\), then for every \(u \in C^{0}(K^\ell; N^n)\) there exists \(f \in C^{0}(K^m; N^n)\) such that \(f|_{K^\ell}=u\).
\end{proposition}

We will use the fact that it is always possible to find a continuous extension, regardless of \(N^n\), by losing one dimension.  This property has been introduced as the \(\ell\) extension property by Hang and Lin \cite[Definition~2.3]{Hang-Lin}.

\begin{proposition}
\label{lemmaContinuousExtensionProperty}
Let \(\cK^m\) be a skeleton of radius \(\eta > 0\) and \(\ell \in \{0, \ldots, m-1\}\).
If \(K^m\) is a cube, then for every \(u \in C^{0}(K^{\ell+1}; N^n)\), there exists \(g\in C^{0}(K^m; N^n)\) such that \(g|_{K^{\ell}}= u|_{K^{\ell}}\).
\end{proposition}

In the proof of Proposition~\ref{lemmaContinuousExtensionProperty}, we shall assume that if \(K^m\) is a cube, then the identity map on \(K^{\ell}\) is homotopic to a constant with respect to \(K^{\ell+1}\):

\begin{fact}
\label{factHomotopyConstant}
If \(K^m\) is a cube, then there exists a continuous homotopy \(G_\ell : [0, 1] \times K^{\ell} \to K^{\ell+1}\) such that 
\begin{enumerate}[$(i)$]
\item for every \(x \in K^{\ell}\), \(G_\ell(0, x) = x\),
\item there exists \(a \in K^{\ell}\) such that for every \(x \in K^{\ell}\), \(G_\ell(1, x) = a\).
\end{enumerate}
\end{fact}

\begin{proof}[Proof of Proposition~\ref{lemmaContinuousExtensionProperty}]
Let \(\varphi : K^m \to [0, 1]\) be a continuous function such that
\begin{itemize}[\(-\)]
\item for every \(x \in K^{\ell}\), \(\varphi(x) = 0\),
\item for every \(x \in T^{\ell^*}\), \(\varphi(x) = 1\).
\end{itemize}
We define \(g : K^m \to N^n\) by
\[
g(x) = 
  \begin{cases}
    u \bigl(G_{\ell}(\varphi(x), H_{\ell}(1, x))\bigr) & \text{if \(x \in K^m \setminus T^{\ell^*}\),}\\
    u(a) & \text{if \(x \in T^{\ell^*}\).}
  \end{cases}
\]
where \(H_{\ell} : [0, 1] \times K^m \setminus T^{\ell^*} \to K^m \setminus T^{\ell^*}\) is the homotopy retraction of Fact~\ref{factHomotopyRetraction}.
The map \(g\) is continuous and by property~\((\ref{item15433})\) of Fact~\ref{factHomotopyRetraction} we have for every \(x \in K^{\ell}\), \(g(x) = u(x)\).
\end{proof}

\begin{proof}[Proof of Proposition~\ref{propositionContinuousExtensionProperty}]
Let \(u \in C^{0}(K^\ell; N^n)\). 
Since \(\pi_\ell(N^n) \simeq \{0\}\), for every \(\sigma^{\ell + 1} \in \cK^{\ell+1}\), the restriction \(u|_{\partial\sigma^{\ell + 1}}\) has a continuous extension \(u_{\sigma^{\ell + 1}}\) to \(\sigma^{\ell + 1}\). 
Let \(v : K^{\ell + 1} \to N^n\) be the map defined for every \(x \in K^{\ell + 1}\) by \(v(x) = u_{\sigma^{\ell + 1}}(x)\), where \(\sigma^{\ell + 1} \in \cK^{\ell + 1}\) is  such that \(x \in \sigma^{\ell + 1}\). 
The map \(v\) is well-defined and continuous; moreover, \(v|_{K^{\ell}} = u\). 
By Proposition~\ref{lemmaContinuousExtensionProperty} applied to \(v\),  there exists \(f : K^m \to N^n\) such that \(f|_{K^{\ell}} = v|_{K^{\ell}}\); hence \(f\) is a continuous extension of \(u\) to \(K^m\). 
\end{proof}

%%%%%%%%%%%%%%%%%%%%%%%%%%%%%%%%%%%%%%%%%%%%%%%%%%%%%
%%%%%%%%%%%%%%%%%%%%%%%%%%%%%%%%%%%%%%%%%%%%%%%%%%%%%
%%%%%%%%%%%%%%%%%%%%%%%%%%%%%%%%%%%%%%%%%%%%%%%%%%%%%

\section{Shrinking}
Given a map \(u \in W^{k, p} (K^m; \R^\nu)\) whose energy is controlled outside a neighborhood of the dual skeleton \(T^{\ell^*}\), we construct for every \(\tau > 0\) a map \(u \circ \Phi\) whose energy will be controlled on the whole \(K^m\) when \(\tau\) is small enough.
This shrinking construction is very similar to the thickening construction. 
In both cases, the dimension of the dual skeleton \(T^{\ell^*}\) must satisfy \(\ell^* < m - kp\), or equivalently, \(l+1 > kp\).
The main differences are that shrinking only acts in a neighborhood of the dual skeleton \(T^{\ell^*}\) and does not create singularities.  Shrinking can be thought of as desingularized thickening and requires more careful estimates.

As for thickening, we begin by constructing the diffeomorphism \(\Phi\) regardless of \(u\):

\begin{proposition}
\label{lemmaThickeningFaceNearDualSkeletonGlobal}
Let \(\ell \in \{0, \dotsc, m-1\}\), \(\eta > 0\), \(0 < \mu <\frac{1}{2}\), \(0 < \tau < \frac{1}{2}\), \(\mathcal{S}^m\) be a cubication of \(\R^m\)  of radius \(\eta\) and \(\cT^{\ell^*}\) be the dual skeleton of \(\cS^\ell\). There exists a smooth map \(\Phi : \R^m  \to \R^m\) such that 
\begin{enumerate}[$(i)$]
\item \(\Phi\) is injective,
\label{itempropositionshrinkingfromaskeleton15}
\item 
for every \(\sigma^m \in \cS^m\),
\(\Phi(\sigma^m)\subset \sigma^m\),
\label{itempropositionshrinkingfromaskeleton2}
\item \(\Supp{\Phi}\subset T^{\ell^*} + Q^m_{2\mu\eta}\)
and \(\Phi\big(T^{\ell^*} + Q^m_{\tau\mu\eta}\big) \supset T^{\ell^*} + Q^m_{\mu\eta}\),
\label{itempropositionshrinkingfromaskeleton1}
\item for every \(0 < \beta < \ell + 1\), for every \(j \in \N_*\) and for every \(x\in \R^m\), 
\[
(\mu\eta)^{j-1}\abs{D^j \Phi(x)} \le  C \bigl(\jac{\Phi}(x)\bigr)^\frac{j}{\beta},
\]
for some constant \(C > 0\) depending on \(\beta\), \(j\) and \(m\),
\label{itempropositionshrinkingfromaskeleton5}
\item for every \(0 < \beta < \ell + 1\), for every \(j \in \N_*\) and for every \(x\in \Phi^{-1}(T^{\ell^*} + Q^m_{\mu\eta})\), 
\[
(\mu\eta)^{j-1}\abs{D^j \Phi(x)} \le  C' \tau^{j(\frac{\ell + 1}{\beta} - 1)} \bigl(\jac{\Phi}(x)\bigr)^\frac{j}{\beta},
\]
for some constant \(C' > 0\) depending on \(\beta\), \(j\) and \(m\).
\label{itempropositionshrinkingfromaskeleton4}
\end{enumerate}
\end{proposition}

As a consequence of the estimates of Proposition~\ref{lemmaThickeningFaceNearDualSkeletonGlobal}, we have the following \(W^{k, p}\) estimates that will be applied in the proof of Theorem~\ref{theoremDensityManifoldMain} with \(\ell = \floor{kp}\). 

\begin{corollary}
\label{corollaryShrinkingWkpEstimate}
Let \(\Phi : \R^m \to \R^m\) be the map given by Proposition~\ref{lemmaThickeningFaceNearDualSkeletonGlobal} and let \(\mathcal{K}^m\) be a subskeleton of \(\cS^m\). If  \(\ell+1 > kp\), then for every \(u \in W^{k, p}(K^m \cap(T^{\ell^*} + Q^m_{2\mu\eta}); \R^\nu)\), \(u \circ \Phi \in W^{k, p}(K^m \cap(T^{\ell^*} + Q^m_{2\mu\eta}); \R^\nu)\) and for every \(j \in \{1, \dotsc, k\}\),
\begin{multline*}
(\mu \eta)^{j} \norm{D^j(u \circ \Phi)}_{L^p(K^m \cap (T^{\ell^*} + Q^m_{2\mu\eta}))}\\
\begin{aligned}
& \leq 
 C'' \sum_{i=1}^j  (\mu\eta)^{i} \norm{D^i u}_{L^p(K^m \cap (T^{\ell^*} + Q^m_{2\mu\eta}) \setminus (T^{\ell^*} + Q^m_{\mu\eta}))}\\
& \qquad 
+ C'' \tau^{\frac{\ell+1-kp}{p}}\sum_{i=1}^j  (\mu\eta)^{i} \norm{D^i u}_{L^p(K^m \cap (T^{\ell^*} + Q^m_{\mu\eta}))},
\end{aligned}
\end{multline*}
for some constant \(C'' > 0\) depending on \(m\), \(k\) and \(p\).
\end{corollary}

\begin{proof}
We first establish the estimate for a map \(u\) in \(C^{\infty}(K^m \cap(T^{\ell^*} + Q^m_{2\mu\eta}) ; \R^{\nu})\). By the chain rule for higher-order derivatives, for every \(j \in \{1, \dotsc, k\}\) and for every \(x \in K^m\),
\[
\abs{D^j (u \circ \Phi) (x)}^p \le \NewConstant \sum_{i=1}^j \sum_{\substack{1 \le t_1 \le \dotsc \le t_i\\ t_1 + \dotsb + t_i = j}} \abs{D^i u(\Phi(x))}^p \abs{D^{t_1} \Phi(x)}^p \dotsm \abs{D^{t_i} \Phi(x)}^p.
\]

As in the proof of Corollary~\ref{corollaryEstimateThickening}, if \(1 \le t_1 \le \ldots \le t_i\) and \(t_1 + \dotsb + t_i = j \), then for every \(x \in K^m \cap(T^{\ell^*} + Q^m_{2\mu\eta})\),
\[
\abs{D^{t_1} \Phi(x)}^p \dotsm \abs{D^{t_i}\Phi(x)}^p \le \Constant \frac{\jac{\Phi}(x)}{\eta^{(j-i)p}}
\]
and this implies
\[
\eta^{jp}\abs{D^j (u \circ \Phi) (x)}^p \le \Constant \sum_{i=1}^j \eta^{ip}\abs{D^i u(\Phi(x))}^p \jac{\Phi}(x).
\]
Let \(\sigma^m \in \cK^m\).
Since \(\Phi\) is injective, by the change of variable formula,
\begin{multline*}
\int\limits_{\Phi^{-1}(\sigma^m \cap (T^{\ell^*} + Q^m_{2\mu\eta}) \setminus (T^{\ell^*} + Q^m_{\mu\eta}))} (\mu\eta)^{jp}\abs{D^j (u \circ \Phi)}^p\\
\begin{aligned} 
& \le  \SameConstant \sum_{i=1}^j \int\limits_{\Phi^{-1}(\sigma^m \cap (T^{\ell^*} + Q^m_{2\mu\eta}) \setminus (T^{\ell^*} + Q^m_{\mu\eta}))} (\mu\eta)^{ip}\abs{(D^i u) \circ \Phi}^p \jac{\Phi}\\
&\le \SameConstant \sum_{i=1}^j \int\limits_{\sigma^m \cap (T^{\ell^*} + Q^m_{2\mu\eta}) \setminus (T^{\ell^*} + Q^m_{\mu\eta})} (\mu\eta)^{ip}\abs{D^i u}^p.
\end{aligned}
\end{multline*}

Let \(0 < \beta < \ell + 1\).
If \(1 \le t_1 \le \dotsc \le t_i\) and \(t_1 + \dotsb + t_i = j \), then by property~\((\ref{itempropositionshrinkingfromaskeleton4})\) of Proposition~\ref{lemmaThickeningFaceNearDualSkeletonGlobal} we have for every \(x \in \Phi^{-1}(K^m \cap (T^{\ell^*} + Q^m_{\mu\eta}))\),
\begin{multline*}
\abs{D^{t_1} \Phi(x)}^p \dotsm \abs{D^{t_i}\Phi(x)}^p \\
\begin{aligned}
& \le \Constant \tau^{t_1 p(\frac{\ell + 1}{\beta} - 1)} \frac{\bigl(\jac{\Phi}(x)\bigr)^\frac{t_1 p}{\beta}}{(\mu\eta)^{(t_1 - 1)p}} \dotsm \tau^{t_i p(\frac{\ell + 1}{\beta} - 1)} \frac{\bigl(\jac{\Phi}(x)\bigr)^\frac{t_i p}{\beta}}{{(\mu\eta)^{(t_i - 1)p}}}\\
& = \SameConstant \tau^{jp(\frac{\ell + 1}{\beta} - 1)}\frac{\bigl(\jac{\Phi}(x)\bigr)^\frac{jp}{\beta}}{(\mu\eta)^{(j-i)p}}.
\end{aligned}
\end{multline*}
Taking \(\beta = jp\), we have
\[
\abs{D^{t_1} \Phi(x)}^p \dotsm \abs{D^{t_i}\Phi(x)}^p \le \SameConstant \tau^{\ell + 1 - jp} \frac{\jac{\Phi}(x)}{(\mu\eta)^{(j-i)p}}
\]
and this implies
\[
(\mu\eta)^{jp}\abs{D^j (u \circ \Phi) (x)}^p \le \Constant \tau^{\ell + 1 - jp} \sum_{i=1}^j (\mu\eta)^{ip}\abs{D^i u(\Phi(x))}^p \jac{\Phi}(x).
\]
Since \(\Phi\) is injective, by the change of variable formula,
\begin{multline*}
\int\limits_{\Phi^{-1}(\sigma^m \cap (T^{\ell^*} + Q^m_{\mu\eta}))} (\mu\eta)^{jp}\abs{D^j (u \circ \Phi)}^p\\
\begin{aligned}
& \le  \SameConstant \tau^{\ell + 1 - jp} \sum_{i=1}^j \int\limits_{\Phi^{-1}(\sigma^m \cap (T^{\ell^*} + Q^m_{\mu\eta}))} (\mu\eta)^{ip}\abs{(D^i u) \circ \Phi}^p \jac{\Phi}\\
& =  \SameConstant \tau^{\ell + 1 - jp} \sum_{i=1}^j \int\limits_{\sigma^m \cap (T^{\ell^*} + Q^m_{\mu\eta})} (\mu\eta)^{ip} \abs{D^i u}^p.
\end{aligned}
\end{multline*}

Since \(\sigma^m \cap (T^{\ell^*} + Q^m_{2\mu\eta}) \subset \Phi^{-1}\big(\sigma^m \cap (T^{\ell^*} + Q^m_{2\mu\eta})\big)\), by additivity of the integral we then have
\begin{multline*}
\int\limits_{\sigma^m \cap (T^{\ell^*} + Q^m_{2\mu\eta})} (\mu\eta)^{jp}\abs{D^j (u \circ \Phi)}^p\\
 \le 
 C_3 \sum_{i=1}^j \int\limits_{\sigma^m \cap (T^{\ell^*} + Q^m_{2\mu\eta}) \setminus (T^{\ell^*} + Q^m_{\mu\eta})} (\mu\eta)^{ip} \abs{D^i u}^p\\
 + \SameConstant \tau^{\ell + 1 - jp} \sum_{i=1}^j \int\limits_{\sigma^m \cap (T^{\ell^*} + Q^m_{\mu\eta})} (\mu\eta)^{ip} \abs{D^i u}^p.
\end{multline*}
We may take the union over all faces \(\sigma^m \in \cK^m\) and we deduce the estimate for smooth maps. By density of smooth maps in \(W^{k, p}(K^m \cap(T^{\ell^*} + Q^m_{2\mu\eta}); \R^\nu)\), we deduce that for every \(u\) in \(W^{k, p}(K^m \cap(T^{\ell^*} + Q^m_{2\mu\eta}); \R^\nu)\), the function \(u \circ \Phi\) also belongs to this space and satisfies the estimate above.
\end{proof}

We first describe the construction of the map \(\Phi\) in the case of only one \(\ell\) dimensional cube.

\begin{proposition}
\label{lemmaThickeningFaceNearDualSkeleton}
Let \(\ell \in \{1, \dotsc, m\}\), \(\eta > 0\), \(0 < \underline{\mu} < \mu < \overline{\mu} < 1\) and \(0 <\tau < \underline\mu/\mu\).
There exists a smooth function \(\lambda : \R^m\to [1,\infty)\) such that if \(\Phi : \R^m  \to \R^m\) is defined for \(x = (x', x'') \in \R^\ell \times \R^{m - \ell} \) by
\[
\Phi(x) = (\lambda(x)x', x''),
\]
then
\begin{enumerate}[$(i)$]
\item \(\Phi\) is injective,
\label{itemlemmaThickeningFaceNearDualSkeletonidentity0}
\item \(\Supp{\Phi}\subset Q^\ell_{\mu\eta} \times Q^{m-\ell}_{(1-\mu)\eta}\),
\label{itemlemmaThickeningFaceNearDualSkeletonidentity}
\item \(\Phi\big(Q^\ell_{\tau\mu\eta} \times Q^{m-\ell}_{(1-\overline{\mu})\eta} \big) \supset Q^\ell_{\underline{\mu}\eta} \times Q^{m-\ell}_{(1-\overline{\mu})\eta}\),
\label{itemlemmaThickeningFaceNearDualSkeletonidentity1}
\item for every \(0 < \beta < \ell\), for every \(j \in \N_*\) and for every 
\(x\in \R^m\), 
\[
(\mu\eta)^{j-1}\abs{D^j \Phi(x)} \le  C \bigl(\jac{\Phi}(x)\bigr)^\frac{j}{\beta},
\]
for some constant \(C > 0\) depending on \(\beta\), \(j\), \(m\), \(\mu/\underline{\mu}\) and \(\overline{\mu}/ \mu\),
\label{itemlemmaThickeningFaceNearDualSkeletonidentity2}
\item for every \(\beta > 0\), for every \(j \in \N_*\) and for every \(x\in Q^{\ell}_{\tau\mu\eta}\times Q^{m-\ell}_{(1-\overline\mu)\eta}\), 
\[
(\mu\eta)^{j-1}\abs{D^j \Phi(x)} \le  C' \tau^{j(\frac{\ell}{\beta} - 1)} \bigl(\jac{\Phi}(x)\bigr)^\frac{j}{\beta},
\]
for some constant \(C' > 0\) depending on \(\beta\), \(j\), \(m\), \(\mu/\underline{\mu}\) and \(\overline{\mu}/ \mu\).
\label{itemlemmaThickeningFaceNearDualSkeletonidentity3}
\end{enumerate}
\end{proposition}

We postpone the proof of Proposition~\ref{lemmaThickeningFaceNearDualSkeleton} and we proceed to establish Proposition~\ref{lemmaThickeningFaceNearDualSkeletonGlobal}.

\begin{proof}[Proof of Proposition~\ref{lemmaThickeningFaceNearDualSkeletonGlobal}]
We first introduce  finite sequences \( (\mu_i)_{\ell\leq i \leq m}\) and \((\nu_i)_{\ell\leq i \leq m}\) such that
\[
0<\mu_\ell=\mu<\nu_{\ell+1}<\mu_{\ell+1}<\dotsc<\mu_{m-1}<\nu_m<\mu_m\leq 2\mu.
\]

Let \(\Phi_m = \Id\). Using downward induction, we shall define  maps \(\Phi_i:\R^m\to \R^m\)  for  \(i \in \{\ell, \dotsc, m-1\}\) such that \(\Phi_i\) satisfies  the following properties: 

\begin{enumerate}[(a)]
\item \(\Phi_i\) is injective,
\label{item1281}
\item for every \(\sigma^m \in \cS^m\), \(\Phi_i(\sigma^m) \subset \sigma^m\),
\label{item1282}
\item \(\Supp{\Phi_i}\subset T^{i^*} + Q^m_{2 \mu\eta}\),
\item for every \(r \in \{i^*, \dots, m - 1\}\), \(\Phi_i \big(T^{r} + Q^m_{\tau\mu\eta} \big) \supset T^{r} + Q^m_{\tau\mu\eta}\),
\label{item1283}
\item \(\Phi_i \big(T^{i^*} + Q^m_{\tau\mu\eta} \big) \supset T^{i^*} + Q^m_{\mu_i\eta}\),
\label{item1284}
\item 
for every \(0 < \beta < i+1\), for every \(j \in \N_*\) and for every 
\(x \in \R^m\),
\[
(\mu\eta)^{j-1} \abs{D^j \Phi_i(x)} \le  C \bigl(\jac{\Phi_i}(x)\bigr)^\frac{j}{\beta},
\]
for some constant \(C > 0\) depending on \(\beta\), \(j\), and \(m\),
\label{item1285}
\item for every \(0 < \beta < i+1\), for every \(j \in \N_*\) and for every \(x\in \Phi_i^{-1}(T^{i^*} + Q^m_{\mu_i\eta})\), 
\[
(\mu\eta)^{j-1}\abs{D^j \Phi_i(x)} \le  C' \tau^{j(\frac{i+1}{\beta} - 1)} \bigl(\jac{\Phi_i}(x)\bigr)^\frac{j}{\beta},
\]
for some constant \(C' > 0\) depending on \(\beta\), \(j\), and \(m\).
\label{item1286}
\end{enumerate}

The map \(\Phi_\ell\) will satisfy the conclusion of the proposition.

\medskip

Let \( i \in \{\ell+1, \dotsc, m\}\) and let \(\Theta_{i}\) be the map obtained from Proposition~\ref{lemmaThickeningFaceNearDualSkeleton} with parameters \(\ell = i\), \(\underline\mu = \mu_{i - 1}\), \(\mu = \nu_{i}\), \(\overline\mu = \mu_{i}\) and \(\frac{\tau\mu}{\nu_i}\).  Given \(\sigma^i \in \mathcal{S}^{i}\), we may identify \(\sigma^i\) with \(Q^{i}_{\eta} \times \{0^{m-i}\}\) and \(T^{(i-1)^*} \cap (\sigma^i + Q_{2\mu\eta}^m)\) with \(\{0^{i}\} \times Q_{2\mu\eta}^{m-i}\). The map \(\Theta_i\) induces by isometry a map which we shall denote by \(\Theta_{\sigma^i}\).

Let \(\Psi_i : \R^m  \to \R^m\) be defined for every \(x \in \R^m\) by
\[
\Psi_i(x):=\begin{cases} 
           \Theta_{\sigma^i}(x) & \text{if } x \in \sigma^i + Q^{m}_{(1-\nu_i)\eta} \text{ for some } \sigma^i \in \mathcal{S}^i,\\
            x&\text{otherwise} .
         \end{cases}
\] 
We explain why \(\Psi_i\) is well-defined. Since \(\Theta_{\sigma^i}\) coincides with the identity map on \(\partial\sigma^i + Q^m_{(1-\nu_{i})\eta}\), then for every \(\sigma^i_1, \sigma^i_2 \in \mathcal{S}^{i}\), if \(x \in (\sigma_1^i + Q^{m}_{(1-\nu_{i})\eta}) \cap (\sigma_2^i + Q^{m}_{(1-\nu_{i})\eta})\) and \(\sigma_1^i \ne \sigma_2^i\), then
\[
\Theta_{\sigma_1^i}(x) = x = \Theta_{\sigma_2^i}(x).
\]
One also verifies that \(\Psi_i\) is smooth.

Assuming that \(\Phi_i\) has been defined satisfying properties \eqref{item1281}--\eqref{item1286}, we let
\[
\Phi_{i-1}=\Psi_i \circ \Phi_i.
\]

We check that \(\Phi_{i-1}\) satisfies all required properties. Up to an exchange of coordinates, for every \(\sigma^i \in \cS^i\), we may assume that \(\sigma^i = Q_\eta^i \times \{0^{m-i}\}\) and \(\Theta_{\sigma^i}\) can be written as \(\Theta_{\sigma^i}(x) = (\lambda(x)x', x'')\), with \(\lambda(x) \ge 1\). Hence, for every \(0 < s \le 1\) and every \(r\in \{0, \dotsc, m-1\}\),
\begin{equation}
\label{equation1218}
\Psi_i(T^r + Q_{s\eta}^m) \supset T^r + Q_{s\eta}^m.
\end{equation}
Moreover, in the new coordinates, the set 
\begin{equation}
\label{equation1627}
(\sigma^i \times Q^{m-i}_{\eta}) \cap \big( (T^{(i-1)^*} + Q^m_{\tau\mu\eta}) \setminus (T^{i^*} + Q^m_{\mu_i\eta}) \big)
\end{equation}
becomes
\begin{equation}
\label{equation1628}
Q^i_{\tau\mu\eta} \times Q^{m-i}_{(1 - \mu_i)\eta}.
\end{equation}
In view of properties~\((\ref{itemlemmaThickeningFaceNearDualSkeletonidentity0})\) and~\((\ref{itemlemmaThickeningFaceNearDualSkeletonidentity1})\) of Proposition~\ref{lemmaThickeningFaceNearDualSkeleton},
\[
\Theta_{\sigma^i}(Q^i_{\tau\mu\eta} \times Q^{m-i}_{(1 - \mu_i)\eta}) \supset Q^i_{\mu_{i-1}\eta} \times Q^{m-i}_{(1 - \mu_i)\eta}.
\]
Since this property holds for every \(\sigma^i \in \cS^i\),
\begin{equation}
\label{equation1219}
\Psi_i\big( (T^{(i-1)^*} + Q^m_{\tau\mu\eta}) \setminus (T^{i^*} + Q^m_{\mu_i\eta}) \big) \supset 
(T^{(i-1)^*} + Q^m_{\mu_{i-1}\eta}) \setminus (T^{i^*} + Q^m_{\mu_i\eta}) .
\end{equation}

\begin{proof}[Proof of Property \eqref{item1283}]
Let \(r \in \{(i-1)^*, \dots, m-1\}\).
By induction hypothesis and by equation~\eqref{equation1218} with \(s = \tau\mu\),
\[
\Phi_{i-1}(T^r + Q^m_{\tau\mu\eta}) \supset \Psi_i(T^r + Q^m_{\tau\mu\eta}) \supset T^r + Q^m_{\tau\mu\eta}.
\qedhere
\]
\end{proof}

\begin{proof}[Proof of Property \eqref{item1284}]
By induction hypothesis (properties~\eqref{item1283} and \eqref{item1284}),
\[
\Phi_{i}(T^{(i-1)^*} + Q^m_{\tau\mu\eta}) \supset (T^{(i-1)^*} + Q^m_{\tau\mu\eta})\cup  (T^{i^*} + Q^m_{\mu_i\eta}).
\]
Thus,
\[
\Phi_{i-1}(T^{(i-1)^*} + Q^m_{\tau\mu\eta}) \supset \Psi_i (T^{(i-1)^*} + Q^m_{\tau\mu\eta}) \cup \Psi_i (T^{i^*} + Q^m_{\mu_i\eta}). 
\]
By inclusion~\eqref{equation1219} and by inclusion~\eqref{equation1218} with \(r = i^*\) and \(s = \mu_i\),
\[
\begin{split}
 \Phi_{i-1}(T^{(i-1)^*} + Q^m_{\tau\mu\eta})
& \supset
\big((T^{(i-1)^*} + Q^m_{\mu_{i-1}\eta}) \setminus (T^{i^*} + Q^m_{\mu_i\eta})\big) \cup \big(T^{i^*} + Q^m_{\mu_i\eta}\big)  \\
& = T^{(i-1)^*} + Q^m_{\mu_{i-1}\eta}.
\end{split}
\]
This gives the conclusion.
\end{proof}

\begin{proof}[Proof of Property \eqref{item1286}]
Let \(j \in \N_*\) and \(0 < \beta < i\).
By the chain rule for higher order derivatives, we have for every \(x \in \R^m\),
\[
\abs{D^j\Phi_{i-1}(x)} \le
C_1 \sum_{r=1}^j \sum_{\substack{1 \le t_1 \le \dotsc \le t_r\\ t_1 + \dotsb + t_r = j}}\abs{D^r \Psi_i(\Phi_i(x))}\, \abs{D^{t_1} \Phi_i(x)} \dotsm \abs{D^{t_r} \Phi_i(x)}.
\]
Let \(x \in \Phi_{i-1}^{-1}(T^{(i-1)^*} + Q^{m}_{\mu_{i-1}\eta})\). 
By induction hypothesis (property~\eqref{item1285}), for every \(r \in \{1, \dots, j\}\), if \(1 \le t_1 \le \ldots \le t_r\) and \(t_1 + \dotsb + t_r = j \), then
\[
\abs{D^{t_1} \Phi_i(x)} \dotsm \abs{D^{t_r}\Phi_i(x)} 
\le C_2 \frac{(\jac{\Phi_i}(x))^\frac{j}{\beta}}{(\mu\eta)^{j-r}}.
\]
If in addition \(x \in \Phi_{i-1}^{-1}\big((T^{(i-1)^*} + Q^{m}_{\mu_{i-1}\eta}) \setminus (T^{i^*} + Q^{m}_{\mu_{i}\eta})\big)\), then \(\Phi_i(x) \in \Psi_{i}^{-1}\big((T^{(i-1)^*} + Q^{m}_{\mu_{i-1}\eta}) \setminus (T^{i^*} + Q^{m}_{\mu_{i}\eta})\big)\). 
By the correspondence between the sets given by \eqref{equation1627} and \eqref{equation1628}, by inclusion~\eqref{equation1219}, and  by property~\((\ref{itemlemmaThickeningFaceNearDualSkeletonidentity3})\) of Proposition~\ref{lemmaThickeningFaceNearDualSkeleton}, we have for every \(0 < \alpha < i\),
\[
\abs{D^{r} \Psi_i(\Phi_i(x))} \le 
C_3 \tau^{r(\frac{i}{\alpha} - 1)} \frac{\big(\jac{\Psi_i}(\Phi_i(x))\big)^\frac{r}{\alpha}}{(\mu\eta)^{r-1}}.
\]
Take \(\alpha = \beta\frac{r}{j}\). Since \(r \le j\) and \(\tau \le 1\), we get
\[
\abs{D^{r} \Psi_i(\Phi_i(x))} \le 
C_3 \tau^{r(\frac{ij}{\beta r} - 1)} \frac{\big(\jac{\Psi_i}(\Phi_i(x))\big)^\frac{j}{\beta}}{(\mu\eta)^{r-1}} \le C_3 \tau^{j(\frac{i}{\beta} - 1)} \frac{\big(\jac{\Psi_i}(\Phi_i(x))\big)^\frac{j}{\beta}}{(\mu\eta)^{r-1}}.
\]
Thus, for every \(x \in \Phi_{i-1}^{-1}\big((T^{(i-1)^*} + Q^{m}_{\mu_{i-1}\eta}) \setminus (T^{i^*} + Q^{m}_{\mu_{i}\eta})\big)\),
\[
\begin{split}
\abs{D^j\Phi_{i-1}(x)} 
& \le
C_4 \tau^{j(\frac{i}{\beta} - 1)} \frac{\big(\jac{\Psi_i}(\Phi_i(x))\big)^\frac{j}{\beta}}{(\mu\eta)^{r-1}} \frac{(\jac{\Phi_i}(x))^\frac{j}{\beta}}{(\mu\eta)^{j-r}}\\
& = C_4 \tau^{j(\frac{i}{\beta} - 1)} \frac{(\jac{\Phi_{i-1}(x)})^\frac{j}{\beta}}{(\mu\eta)^{j-1}}.
\end{split}
\]

On the other hand, if \(x \in \Phi_{i-1}^{-1}(T^{i^*} + Q^{m}_{\mu_{i}\eta})\),
then \(\Phi_i(x) \in \Psi_{i}^{-1}(T^{i^*} + Q^{m}_{\mu_{i}\eta})\). By inclusion  \eqref{equation1218} with \(r = i^*\) and \(s = \mu_i\), \(\Phi_i(x) \in T^{i^*} + Q^{m}_{\mu_{i}\eta}\). By induction hypothesis (property \eqref{item1286}), we deduce that
for every \(r \in \{1, \dots, j\}\), if \(1 \le t_1 \le \ldots \le t_r\) and \(t_1 + \dotsb + t_r = j \), then
\[
\abs{D^{t_1} \Phi_i(x)} \dotsm \abs{D^{t_r}\Phi_i(x)} \le 
C_5 \tau^{j(\frac{i}{\beta} - 1)} \frac{(\jac{\Phi_i}(x))^\frac{j}{\beta}}{(\mu\eta)^{j-r}}.
\]
By property~\((\ref{itemlemmaThickeningFaceNearDualSkeletonidentity2})\) of Proposition~\ref{lemmaThickeningFaceNearDualSkeleton},
\[
\abs{D^r\Psi_{i}(\Phi_i(x))} \le
C_6  \frac{(\jac{\Psi_i}(\Phi_i(x)))^\frac{j}{\beta}}{(\mu\eta)^{r-1}}.
\]
We deduce as above that
\[
\abs{D^j\Phi_{i-1}(x)} \le
C_7 \tau^{j(\frac{i}{\beta} - 1)} \frac{(\jac{\Phi_{i-1}}(x))^\frac{j}{\beta}}{(\mu\eta)^{j-1}}.
\]
This gives the conclusion.
\end{proof}
The other properties can be checked as in the proof of Proposition~\ref{propositionthickeningfromaskeleton}.

By downward induction, we conclude that properties \eqref{item1281}--\eqref{item1286} hold for every \(i \in \{\ell, \dots, m-1\}\). In particular, we deduce properties $(i)$--\((v)\) of Proposition~\ref{lemmaThickeningFaceNearDualSkeleton}.
\end{proof}

We need a couple of lemmas in order to prove Proposition~\ref{lemmaThickeningFaceNearDualSkeleton}: 

\begin{lemma}
\label{lemmaThickeningAroundPrimalSqueletonmu}
Let \(\eta > 0\), let \(0<\underline{\mu}<\mu<\overline{\mu}<1\)  and \(0 < \kappa < \underline{\mu}/\mu\).
There exists a smooth function \(\lambda : \R^m \to [1,\infty)\) such that if \(\Phi : \R^m   \to \R^m\) is defined for \(x = (x', x'') \in \R^\ell \times \R^{m - \ell}  \) by
\[
\Phi(x) = (\lambda(x)x', x''),
\]
then
\begin{enumerate}[$(i)$]
\item \(\Phi\) is a diffeomorphism,
\item \(\Supp{\Phi}\subset Q^\ell_{\mu\eta} \times Q^{m-\ell}_{(1-\mu)\eta}\),
\item \(\Phi\bigl( Q^\ell_{\kappa\mu\eta} \times Q^{m-\ell}_{(1-\overline{\mu})\eta } \bigr) \supset Q^\ell_{\underline{\mu} \eta} \times Q^{m-\ell}_{(1-\overline{\mu})\eta }\),
\item for every \(j\in \N_{*}\) and for every 
\(x\in \R^m\),
\begin{equation*}
(\mu\eta)^{j-1}\abs{D^j\Phi(x)}\leq C,
\end{equation*}
for some constant \(C > 0\) depending on \(j\), \(m\), \(\mu / \underline{\mu}\),  \(\overline{\mu}/\mu\) and \(\kappa\),
\item for every \(j\in \N_{*}\) and every 
\(x\in \R^m\),
\[
C' \le \jac \Phi(x) \le C'',
\]
\end{enumerate}
for some constants \(C', C'' > 0\) depending on  \(m\), \(\mu / \underline{\mu}\),  \(\overline{\mu}/\mu\) and \(\kappa\).
\end{lemma}
\begin{proof}
By scaling, we may assume that \(\mu\eta=1\).
Let \(\psi : \R \to [0, 1]\) be a smooth function such that 
\begin{itemize}[\(-\)]
\item the function \(\psi\) is nonincreasing on \(\R_+\) and nondecreasing on \(\R_-\),
\item for \(\abs{t} \le \underline{\mu}/\mu\), \(\psi(t)=1\), 
\item for \(\abs{t} \ge 1\), \(\psi(t)=0\).
\end{itemize}
Let \(\theta : \R \to [0, 1]\) be a smooth function such that
\begin{itemize}[\(-\)]
\item for \(\abs{t}\le \frac{1-\overline{\mu}}{\mu}\), \(\theta(t)=1\),
\item for \(\abs{t} \ge \frac{1-\mu}{\mu}\), \(\theta(t)=0\).
\end{itemize}
Since \(\frac{1-\mu}{\mu}-\frac{1-\overline{\mu}}{\mu}= \overline{\mu}/ \mu - 1\), we may require that for every  \(j\in \N_*\) and for every \(t\geq 0\), \(\abs{D^j \theta (t)}\leq C\), for some constant \(C > 0\) depending only on \(j\) and \(\overline{\mu}/\mu\).

Let \(\varphi : \R^m \to \R\) be the function defined for \(x=(x_1, \dotsc, x_m) \in \R^m\) by
\[
\textstyle \varphi(x) =  \prod\limits_{i=1}^\ell \psi(x_i) \prod\limits_{i=\ell+1}^m \theta (x_i).
\]
Let \(\Psi : \R^m \to \R^m\) be the function defined for \(x = (x', x'') \in \R^\ell \times \R^{m - \ell}\) by
\[
\Psi(x)= \big((1 - \alpha \varphi(x))x', x''\big),
\]
where \(\alpha \in \R\). In particular, for every \(x = (x', x'') \in Q^\ell_{\underline\mu/\mu} \times Q^{m - \ell}_{\frac{1 - \overline\mu}{\mu}}\), 
\[
\Psi(x) = ((1 - \alpha) x', x'').
\]
Taking \(\alpha=1-\frac{\kappa\mu}{\underline{\mu}}\), we deduce that \(\Psi\) is a bijection between \(Q^\ell_{\underline\mu/\mu} \times Q^{m-\ell}_{\frac{1 - \overline\mu}{\mu}}\) and \(Q^\ell_{\kappa} \times Q^{m-\ell}_{\frac{1 - \overline\mu}{\mu}}\).
As in Lemma~\ref{lemmaThickeningAroundPrimalSqueleton}  we can prove that \(\Phi=\Psi^{-1}\) satisfies the required properties.
\end{proof}

\begin{lemma}
\label{lemmaThickeningAroundDualSqueletonmu}
Let \(\ell \in \{1,\dotsc, m\}\), \(\eta > 0\), \(0<\underline{\mu}<\mu<\overline{\mu}<1\) and \(0 < \tau < \underline{\mu}/\mu\).
There exists a smooth function \(\lambda : \R^m \to [1, \infty)\) such that if \(\Phi : \R^m \to \R^m\) is defined for \(x = (x', x'') \in \R^\ell \times \R^{m - \ell}\) by
\[
\Phi(x) = (\lambda(x)x', x''),
\]
then
\begin{enumerate}[$(i)$]
\item \(\Phi\) is injective,
\item \(\Supp{\Phi}\subset Q^\ell_{\mu\eta} \times Q^{m-\ell}_{(1-\mu)\eta}\),
\item \(\Phi(B^{\ell}_{\tau\mu\eta}\times Q^{m-\ell}_{(1-\overline{\mu})\eta})\supset  B^{\ell}_{\underline{\mu} \eta} \times Q^{m-\ell}_{(1-\overline{\mu})\eta} \), 
\label{item2169}
\item for every \(0 < \beta < \ell\), for every \(j \in \N_*\) and for every 
\(x\in \R^m\), 
\[
(\mu\eta)^{j-1}\abs{D^j \Phi(x)} \le  C \bigl(\jac{\Phi}(x)\bigr)^\frac{j}{\beta},
\]
for some constant \(C > 0\) depending on \(\beta\), \(j\), \(m\), \(\mu/\underline{\mu}\)   and \(\overline{\mu}/\mu\),
\item for every \(\beta > 0\), for every \(j \in \N_*\) and for every 
\(x\in B^{\ell}_{\tau\mu\eta}\times Q^{m-\ell}_{(1-\overline\mu)\eta}\), 
\[
(\mu\eta)^{j-1}\abs{D^j \Phi(x)} \le  C' \tau^{j(\frac{\ell}{\beta}- 1)}\bigl(\jac{\Phi}(x)\bigr)^\frac{j}{\beta},
\]
for some constant \(C' > 0\) depending on \(\beta\), \(j\), \(m\), \(\mu/\underline{\mu}\)   and \(\overline{\mu}/\mu\).
\label{item1132}
\end{enumerate}
\end{lemma}
\begin{proof}
By scaling, we may assume that \(\mu\eta = 1\). Given \(\varepsilon > 0\) and \(b > 0\), let
\(\varphi : (0, \infty) \to [1, \infty)\) be a smooth function such that
\begin{itemize}[\(-\)]
\item for \(0 < s \le \tau \sqrt{1 + \varepsilon}\), \(\varphi(s)= \dfrac{\underline{\mu}/\mu}{s} \sqrt{1 + \varepsilon} \, \Bigl(1+\frac{b}{\ln \frac{1}{s}}\Bigr)\),
\item for \(s \ge 1\), \(\varphi(s)=1\),
\item the function \(s \in (0, \infty) \mapsto s\varphi(s)\) is increasing.
\end{itemize}
Note that such function \(\varphi\) exists if we take \(\varepsilon > 0\) such that 
\[
(\underline{\mu}/\mu) \sqrt{1 + \varepsilon} < 1
\]
and thus
\(\tau \sqrt{1 + \varepsilon} < 1\) and if we take \(b > 0\) such that
\[
(\underline{\mu}/{\mu}) \sqrt{1 + \varepsilon} \,  \Bigl(1+\frac{b}{\ln \frac{1}{(\underline{\mu}/{\mu})\sqrt{1 + \varepsilon}}}\Bigr) < 1.
\]
Let \(\theta : \R^{m-\ell} \to [0, 1]\) be a smooth function such that 
\begin{itemize}[\(-\)]
\item for \(y \in Q^{m-\ell}_{\frac{1-\overline{\mu}}{\mu}}\),  \(\theta(y) = 0\),
\item for \(y \in \R^{m-\ell} \setminus Q^{m-\ell}_{\frac{1 - {\mu}}{\mu}}\), \(\theta(y)=1\). 
\end{itemize}
We now introduce  for \(x=(x', x'') \in \R^\ell\times \R^{m-\ell}\),  
\[
\zeta(x) = \sqrt{\abs{x'}^2 + \theta\bigl(x''\bigr)^2 + \varepsilon\tau^2}.
\]

Let \(\lambda : \R^m\to \R\) be the function defined for \(x \in \R^m\) by  
\[
\lambda(x)= \varphi(\zeta(x)).
\]
As in the proof of Lemma~\ref{lemmaThickeningAroundDualSqueleton}, one may check that the map \(\Phi\) defined in the statement satisfies all the required properties:

\begin{proof}[Proof of statement \((\ref{item2169})\)]
Let \(x \in B^\ell_{{\underline{\mu}}/{\mu}} \times Q^{m-\ell}_{\frac{1 - \overline{\mu}}{\mu}}\). For every \(s \ge 0\),
\[
\Phi(sx', x'') = \Big( s \varphi(\sqrt{s^2\abs{x'}^2 + \varepsilon \tau^2}) x', x'' \Big).
\]
Consider the function \(h : [0, \infty) \to \R\) defined by
\[
h(s) = s \varphi(\sqrt{s^2 + \varepsilon \tau^2}).
\]
Then, assuming that \(x' \ne 0\),
\[
\Phi(sx', x'') = \Big( h(s\abs{x'}) \frac{x'}{\abs{x'}}, x'' \Big).
\]
We have \(h(0) = 0\) and \(h(\tau) > \underline{\mu}/\mu \ge \abs{x'}\). By the Intermediate value theorem, there exists \(t \in (0, \tau )\) such that \(h(t) = \abs{x'}\). Thus, \(t \frac{x'}{\abs{x'}} \in B_\tau^\ell\) and \(\Phi(t \frac{x'}{\abs{x'}}, x'') = x\).
\end{proof}

\begin{proof}[Proof of statement \((\ref{item1132})\)]
Proceeding as in the proof of Lemma~\ref{lemmaThickeningAroundDualSqueleton}, one gets for every \(x \in B^{\ell}_{1}\times \R^{m-\ell}\),
\begin{equation*}
\abs{D^j\Phi(x)}\le\frac{\NewConstant}{\zeta(x)^j}.
\end{equation*}
Since \(\zeta(x) \ge \tau\sqrt{\varepsilon}\), we deduce that
\begin{equation*}
\abs{D^j\Phi(x)}\le\frac{\SameConstant}{(\tau\sqrt{\varepsilon})^j} \le \frac{\Constant}{\tau^j}.
\end{equation*}
On the other hand,
\[
\jac \Phi(x)
 =\varphi(\zeta(x))^{\ell-1}\Bigl(\varphi(\zeta(x)) \Bigl(1 - \frac{\abs{x'}^2}{\zeta(x)^2}\Bigr) + \big(\varphi^{(1)}(\zeta(x))\zeta(x) + \varphi(\zeta(x)) \big)\frac{\abs{x'}^2}{\zeta(x)^2} \Bigr).
\]
Since for every \(s > 0\), \(\varphi^{(1)}(s)s + \varphi(s) \ge 0\), we have
\[
\jac \Phi(x) \ge \varphi(\zeta(x))^{\ell} \Bigl(1 - \frac{\abs{x'}^2}{\zeta(x)^2}\Bigr) \ge  \frac{\Constant}{\zeta(x)^\ell} \Bigl(1 - \frac{\abs{x'}^2}{\zeta(x)^2}\Bigr).
\]
If \(x\in B^{\ell}_{\tau}\times Q^{m-\ell}_{\frac{1-\overline{\mu}}{\mu}}\), then \(\zeta(x) \leq \tau \sqrt{1 + \varepsilon}\) and \(\zeta(x)^2 \ge (1 + \varepsilon)|x'|^2\). Thus,
\[
\jac \Phi(x) \ge \frac{\SameConstant} {(\tau \sqrt{1 + \varepsilon})^\ell} \frac{\varepsilon}{1 + \varepsilon} = \frac{\Constant}{\tau^\ell}.
\]
Combining the estimates of \(|D^j \Phi|\) and \(\jac{\Phi}\), we have the conclusion.
\end{proof}

In order to establish  the remaining properties stated in Lemma~\ref{lemmaThickeningAroundDualSqueletonmu}, we only need to repeat the proof  of Lemma~\ref{lemmaThickeningAroundDualSqueleton} with obvious modifications.
\end{proof}

\begin{proof}[Proof of Proposition~\ref{lemmaThickeningFaceNearDualSkeleton}]
Define \(\Phi\) to be the composition of the map  \(\Phi_1\)  given by Lemma~\ref{lemmaThickeningAroundPrimalSqueletonmu} with \(\kappa=\frac{\underline{\mu}}{\mu\sqrt{\ell}}\) together with the map  \(\Phi_2\)  given by Lemma~\ref{lemmaThickeningAroundDualSqueletonmu}; more precisely, \(\Phi=\Phi_1\circ \Phi_2\).  The propeties of \(\Phi\) can be established as in the case of thickening.
\end{proof}

%%%%%%%%%%%%%%%%%%%%%%%%%%%%%%%%%%%%%%%%%%%%%%%%%%%%%%
%%%%%%%%%%%%%%%%%%%%%%%%%%%%%%%%%%%%%%%%%%%%%%%%%%%%%
%%%%%%%%%%%%%%%%%%%%%%%%%%%%%%%%%%%%%%%%%%%%%%%%%%%%%

\section{Proof of Theorem~\ref{theoremDensityManifoldMain}}

Let \(\cK^m\) be a cubication of \(Q_1^m\) of radius \(\eta > 0\) and let \(\cT^{\ell^*}\) be the dual skeleton with respect to \(\cK^\ell\) for some \(\ell \in \{0, \dots, m-1\}\). 

\begin{claim}
Let \(v \in C^\infty(K^m \setminus T^{\ell^*}; N^n) \cap W^{k, p}(K^m; N^n)\). 
If \(\pi_\ell(N^n) \simeq \{0\}\) and if \(\ell^* < m - kp\), then there exists a family of smooth maps \(v^\mathrm{sh}_{\tau_\mu, \mu} : K^m \to N^n\) such that
\[
\lim_{\mu \to 0}{\norm{ v^\mathrm{sh}_{\tau_\mu, \mu} -  v}_{W^{k, p}(K^m)}} = 0.
\]
\end{claim}

This claim is a removable singularity property of topological nature for \(W^{k, p}\) maps. 
Theorem~\ref{theoremDensityManifoldMain} follows from Theorem~\ref{theoremDensityManifoldNontrivialwithouttopologicalcondition} and this claim. 
Indeed, by Theorem~\ref{theoremDensityManifoldNontrivialwithouttopologicalcondition} the class of maps \(v\) in the claim is dense in \(W^{k, p}(K^m; N^n)\) when \(\ell = \lfloor kp \rfloor\). Since the maps \(v^\mathrm{sh}_{\tau_\mu, \mu}\) are smooth and converge to \(v\) in \(W^{k, p}\), we deduce that smooth maps are dense in \(W^{k, p}(K^m; N^n)\).

\begin{proof}[Proof of the Claim]
Assuming that \(\pi_\ell(N^n) \simeq \{0\}\), we can modify \(v\) in a neighborhood of \(T^{\ell^*}\) in order to obtain a smooth map \(v^\mathrm{ex}_\mu : K^m \to N^n\). More precisely, for every  \(0 < \mu < 1\), by  Proposition~\ref{propositionSmoothExtension} and Proposition~\ref{propositionContinuousExtensionProperty}, there exists \(v^\mathrm{ex}_\mu \in C^\infty(K^m; N^n)\) such that \(v^\mathrm{ex}_\mu = v\) in \(K^m \setminus (T^{\ell^*} + Q_{\mu\eta}^m)\). 

Although \(v\) and \(v^\mathrm{ex}_\mu\) coincide in a large set, \(\norm{v^\mathrm{ex}_\mu}_{W^{k, p}(K^m)}\) can be much larger than \(\norm{v}_{W^{k, p}(K^m)}\) since the extension is of topological nature and does not take into account the values of \(v\) in a neighborhood of \(T^{\ell^*}\). 
In order to get a better extension of \(v\), we have to shrink \(T^{\ell^*} + Q_{\mu\eta}^m\) into a smaller neighborhood of \(T^{\ell^*}\).

Assume that \(\mu < \frac{1}{2}\) and take \(0< \tau< \frac{1}{2}\). Let \(\Phi^\mathrm{sh}_{\tau, \mu} :\R^m \to \R^m\) be the smooth diffeomorphism given by Proposition~\ref{lemmaThickeningFaceNearDualSkeletonGlobal}. 
Define 
\[
v^\mathrm{sh}_{\tau, \mu}= (v^\mathrm{ex}_\mu \circ \Phi^\mathrm{sh}_{\tau, \mu}).
\]
In particular \(v^\mathrm{sh}_{\tau, \mu} \in C^\infty(K^m; N^n)\). 

Since \(v^\mathrm{sh}_{\tau, \mu} = v\) in the complement of \(T^{\ell^\ast} + Q^{m}_{2\mu\eta}\), for every \(j \in \N_*\),
\begin{align*}
\norm{D^j v^\mathrm{sh}_{\tau, \mu} - D^j v}_{L^p(K^m)} 
& = \norm{D^j v^\mathrm{sh}_{\tau, \mu} - D^j v}_{L^p(K^m\cap(T^{\ell^\ast} + Q^{m}_{2\mu\eta}))}\\
& \le \norm{D^j v^\mathrm{sh}_{\tau, \mu}}_{L^p(K^m\cap(T^{\ell^\ast} + Q^{m}_{2\mu\eta}))} + \norm{D^j v}_{L^p(K^m\cap(T^{\ell^\ast} + Q^{m}_{2\mu\eta}))}.
\end{align*}
If \(\ell^* < m-kp\), or equivalently if \(\ell + 1 > kp\), then by Corollary~\ref{corollaryShrinkingWkpEstimate} we have for every \(j \in \{1, \dotsc, k\}\),
\begin{multline*}
(\mu \eta)^{j} \norm{D^jv^\mathrm{sh}_{\tau, \mu}}_{L^p(K^m \cap (T^{\ell^*} + Q^m_{2\mu\eta}))}\\
\leq 
\NewConstant \sum_{i=1}^j  (\mu\eta)^{i} \norm{D^i v^\mathrm{ex}_\mu}_{L^p(K^m \cap (T^{\ell^*} + Q^m_{2\mu\eta}) \setminus (T^{\ell^*} + Q^m_{\mu\eta}))}\\
 \qquad + \SameConstant \tau^{\frac{\ell+1-kp}{p}}\sum_{i=1}^j  (\mu\eta)^{i} \norm{D^i v^\mathrm{ex}_\mu}_{L^p(K^m \cap (T^{\ell^*} + Q^m_{\mu\eta}))}.
\end{multline*}
Since \(v^\mathrm{ex}_{\mu} = v\) in the complement of \(T^{\ell^\ast} + Q^{m}_{\mu\eta}\), we deduce that
\begin{multline*}
(\mu \eta)^{j} \norm{D^j v^\mathrm{sh}_{\tau, \mu} - D^j v}_{L^p(K^m)}
\le \Constant \sum_{i=1}^j  (\mu\eta)^{i} \norm{D^i v}_{L^p(K^m \cap (T^{\ell^*} + Q^m_{2\mu\eta}) )}\\
+ C_1 \tau^{\frac{\ell+1-kp}{p}}\sum_{i=1}^j  (\mu\eta)^{i} \norm{D^i v^\mathrm{ex}_\mu}_{L^p(K^m \cap (T^{\ell^*} + Q^m_{\mu\eta}))}.
\end{multline*}

We show that
\begin{equation}
\label{equationConvergenceNearDualSkeleton}
\lim_{\mu \to 0}{\sum_{i=1}^j  (\mu\eta)^{i-j} \norm{D^i v}_{L^p(K^m \cap (T^{\ell^*} + Q^m_{2\mu\eta}) )}} = 0.
\end{equation}

Since \(N^n\) is a compact subset of \(\R^\nu\), \(v\) is  bounded. 
By the Gagliardo-Nirenberg interpolation inequality, for every \(i \in \{1, \dots, k-1\}\), \(D^i v \in L^{\frac{kp}{i}}(K^m)\). 
By H\"older's inequality, for every \(i\in \{1, \dotsc, k\}\) we then have
\begin{multline*}
(\mu\eta)^{i-j} \norm{D^i v}_{L^p(K^m\cap( T^{\ell^\ast} + Q^{m}_{2\mu\eta}))} \\
\begin{aligned}
& \le (\mu\eta)^{i-j} \bigabs{K^m\cap( T^{\ell^\ast} + Q^{m}_{2\mu\eta})}^{\frac{k-i}{kp}} \norm{D^i v}_{L^{\frac{kp}{i}}(K^m\cap( T^{\ell^\ast} + Q^{m}_{2\mu\eta}))}\\
& = \eta^{i-j} \mu^{k - j + (\ell + 1 - kp) \frac{k - i}{kp}} \bigg( \frac{\bigabs{K^m\cap( T^{\ell^\ast} + Q^{m}_{2\mu\eta})}}{\mu^{\ell + 1}} \bigg)^\frac{k-i}{kp} \times
\end{aligned}\\
\times \norm{D^i v}_{L^{\frac{kp}{i}}(K^m\cap( T^{\ell^\ast} + Q^{m}_{2\mu\eta}))}.
\end{multline*}
Since \(\bigabs{K^m\cap( T^{\ell^\ast} + Q^{m}_{2\mu\eta})}\leq \Constant \mu^{\ell+1}\), the limit follows.

For every \(0 < \mu < \frac{1}{2}\), take \(0 < \tau_\mu < \frac{1}{2}\) such that
\begin{equation}
\label{equationConvergenceTau}
\lim_{\mu \to 0}{\tau_\mu^{\frac{\ell+1-kp}{p}}\sum_{i=1}^j  (\mu\eta)^{i-j} \norm{D^i v^\mathrm{ex}_\mu}_{L^p(K^m \cap (T^{\ell^*} + Q^m_{\mu\eta}))}} = 0.
\end{equation}
From \eqref{equationConvergenceNearDualSkeleton} and \eqref{equationConvergenceTau}, we deduce that for every \(j \in \{1, \dots, k\}\),
\[
\lim_{\mu \to 0}{\norm{D^j v^\mathrm{sh}_{\tau_\mu, \mu} -  D^j v}_{L^{p}(K^m)}} = 0.
\]
Since \(v^\mathrm{sh}_{\tau_\mu, \mu}\) converges in measure to \(v\) as \(\mu\) tends to \(0\), we then have
\[
\lim_{\mu \to 0}{\norm{ v^\mathrm{sh}_{\tau_\mu, \mu} -  v}_{W^{k, p}(K^m)}} = 0.
\]
This establishes the claim.
\end{proof}

%%%%%%%%%%%%%%%%%%%%%%%%%%%%%%%%%%%%%%%%%%%%%%%%%%%%%
%%%%%%%%%%%%%%%%%%%%%%%%%%%%%%%%%%%%%%%%%%%%%%%%%%%%%
%%%%%%%%%%%%%%%%%%%%%%%%%%%%%%%%%%%%%%%%%%%%%%%%%%%%%

%%%%%%%%%%%%%%%%%%%%%%%%%%%%%%%%%%%%%%%%%%%%%%%%%%%%%%%
%%%%%%%%%%%%%%%%%%%%%%%%%%%%%%%%%%%%%%%%%%%%%%%%%%%%%%%
%%%%%%%%%%%%%%%%%%%%%%%%%%%%%%%%%%%%%%%%%%%%%%%%%%%%%%%
%%%%%%%%%%%%%%%%%%%%%%%%%%%%%%%%%%%%%%%%%%%%%%%%%%%%%%%

\section{Concluding remarks}

\subsection{Other domains}

The proof of Theorem~\ref{theoremDensityManifoldMain} can be adapted to more general domains \(\Omega \subset \R^m\). 
In order to apply the extension argument at the beginning of the proof of Theorem~\ref{theoremDensityManifoldNontrivialwithouttopologicalcondition}, it suffices that \(\Omega\) be starshaped.

Concerning Theorem~\ref{theoremDensityManifoldMain}, the crucial tool is the extension property of Proposition~\ref{lemmaContinuousExtensionProperty}. 
This can be enforced by assuming that for every \(\ell \in \{0, \dots, \floor{kp} - 1\}\),
\[
\pi_\ell(\Omega) \simeq \{0\}.
\]
This contains in particular the case where \(\Omega\) is starshaped. Another option is to require that for some CW-complex structure, \(\overline{\Omega}\) has the \(\floor{kp} - 1\) extension property with respect to \(N^n\). More precisely, for every \(u \in C^0(\overline{\Omega}^{\floor{kp}}; N^n)\), the restriction \(u \vert_{\overline{\Omega}^{\floor{kp}-1}}\) of \(u\) to the skeleton of \(\overline\Omega\) of dimension \(\floor{kp} - 1\) has a continuous extension to \(\overline{\Omega}\). It can be showed that this property does not depend on the CW-complex structure of \(\overline{\Omega}\) (see remark following \cite[Definition~2.3]{Hang-Lin}).

\subsection{Complete manifolds}

The proofs of Theorems~\ref{theoremDensityManifoldMain} and~\ref{theoremDensityManifoldNontrivialwithouttopologicalcondition} still apply 
for complete manifolds \(N^n\) that are embedded in \(\R^\nu\) and for which there exists a projection \(\Pi\) defined on a uniform neighborhood of size \(\iota\) around  \(N^n\). The compactness of \(N^n\) ensures the Gagliardo-Nirenberg interpolation inequality that for every \(i \in \{1, \dotsc, k-1\}\), \(D^i u \in L^\frac{kp}{i}(Q^m_1)\). 
This inequality still holds if the assumption \(u \in L^\infty\) is replaced by \(u \in W^{1, kp}\). 
In this case, one proves that if \(\pi_{\floor{kp}(N^n)} \simeq \{0\}\), then for every 
\(
u \in W^{k, p}(Q^m; N^n) \cap W^{1, kp}(Q^m; N^n)
\)
there exists a family of maps \(u_\eta \in C^\infty(Q^m; N^n)\) such that for every \(i \in \{1, \dotsc, k\}\),
\[ 
\lim_{\eta \to 0} \norm{D^i u_\eta - D^i u}_{L^\frac{kp}{i}(Q^m)} = 0
\]
and \(u_\eta\) converges to \(u\) in measure as \(\eta\) tends to \(0\). Hence,
\[ 
\lim_{\eta \to 0} \norm{u_\eta - u}_{W^{k, p}(Q^m) \cap W^{1, kp}(Q^m)} = 0.
\]

\subsection{$\floor{kp}$ simply connected manifolds}

Under the additional assumption that for every \(\ell \in \{0, \dots, \floor{kp}\}\),
\[
\pi_{\ell}(N^n) \simeq \{0\},
\]
it is possible to give a simpler proof of \(H^{k, p}(Q^m; N^n) = W^{k, p}(Q^m; N^n)\)
without relying on the density of maps in \(R_{m - \floor{kp} - 1}(Q^m; N^n)\).
This approach is inspired by previous works of Escobedo~\cite{Escobedo} and Haj\l asz~\cite{Hajlasz};
see \cite{Bousquet-Ponce-VanSchaftingen-2013} for details.

\bigskip
\footnotesize
\noindent\textit{Acknowledgments.}
The authors would like to thank Petru Mironescu for enlightening discussions and for his encouragement.  
The second (ACP) and third (JVS) authors were supported by the Fonds de la Recherche scientifique---FNRS.

%%%%%%%%%%%%%%%%%%%%%%%%%%%%%%%%%%%%%%%%%%%%%%%%%%%%%%%
%%%%%%%%%%%%%%%%%%%%%%%%%%%%%%%%%%%%%%%%%%%%%%%%%%%%%%%
%%%%%%%%%%%%%%%%%%%%%%%%%%%%%%%%%%%%%%%%%%%%%%%%%%%%%%%
%%%%%%%%%%%%%%%%%%%%%%%%%%%%%%%%%%%%%%%%%%%%%%%%%%%%%%%

%%%%%%%%%%%%%%%%%%%%%%%%%%%%%%%%%%%%%%%%%%%%%%%%%%%%%%%
%%%%%%%%%%%%%%%%%%%%%%%%%%%%%%%%%%%%%%%%%%%%%%%%%%%%%%%
%%%%%%%%%%%%%%%%%%%%%%%%%%%%%%%%%%%%%%%%%%%%%%%%%%%%%%%
%%%%%%%%%%%%%%%%%%%%%%%%%%%%%%%%%%%%%%%%%%%%%%%%%%%%%%%

%%%%%%%%%%%%%%%%%%%%%%%%%%%%%%%%%%%%%%%%%%%%%%%%%%%%%%%
%%%%%%%%%%%%%%%%%%%%%%%%%%%%%%%%%%%%%%%%%%%%%%%%%%%%%%%
%%%%%%%%%%%%%%%%%%%%%%%%%%%%%%%%%%%%%%%%%%%%%%%%%%%%%%%
%%%%%%%%%%%%%%%%%%%%%%%%%%%%%%%%%%%%%%%%%%%%%%%%%%%%%%%

\end{document}